\documentclass[11pt]{amsart}
\pdfoutput=1
\usepackage{amsmath,amssymb,amsthm,array,color,multirow,pdflscape,graphicx,pigpen,stmaryrd,comment}
\usepackage[all]{xypic}
\usepackage[normalem]{ulem}

\usepackage{lmodern}

\usepackage{tikz}
\usepackage{tikz-cd}
\tikzset{shorten <>/.style={shorten >=#1,shorten <=#1}}

\usepackage{float}

\usepackage{todonotes}

\usepackage[margin=1.5in]{geometry}

\usepackage[all,arc,2cell]{xy}
\UseAllTwocells

\usepackage{microtype}

\usepackage{xcolor}
\definecolor{darkgreen}{rgb}{0,0.30,0} 
\definecolor{darkred}{rgb}{0.75,0,0}
\definecolor{darkblue}{rgb}{0,0,0.6} 
\usepackage[pdfborder=0,pagebackref,colorlinks,citecolor=darkgreen,linkcolor=darkgreen,urlcolor=darkblue]{hyperref}
\renewcommand*{\backref}[1]{}
\renewcommand*{\backrefalt}[4]{({%
    \ifcase #1 Not cited.%
          \or On p.~#2%
          \else On pp.~#2%
    \fi%
    })}

\def\makeautorefname#1#2{\expandafter\def\csname#1autorefname\endcsname{#2}}
\makeautorefname{equation}{Equation}%
\makeautorefname{footnote}{footnote}%
\makeautorefname{item}{item}%
\makeautorefname{figure}{Figure}%
\makeautorefname{table}{Table}%
\makeautorefname{part}{Part}%
\makeautorefname{appendix}{Appendix}%
\makeautorefname{chapter}{Chapter}%
\makeautorefname{section}{Section}%
\makeautorefname{subsection}{Section}%
\makeautorefname{subsubsection}{Section}%
\makeautorefname{paragraph}{Paragraph}%
\makeautorefname{subparagraph}{Paragraph}%
\makeautorefname{theorem}{Theorem}%
\makeautorefname{thm}{Theorem}%
\makeautorefname{addm}{Addendum}%
\makeautorefname{mainthm}{Main theorem}%
\makeautorefname{corollary}{Corollary}%
\makeautorefname{cor}{Corollary}%
\makeautorefname{lemma}{Lemma}%
\makeautorefname{lem}{Lemma}%
\makeautorefname{sublemma}{Sublemma}%
\makeautorefname{sublem}{Sublemma}%
\makeautorefname{subl}{Sublemma}%
\makeautorefname{prop}{Proposition}%
\makeautorefname{property}{Property}
\makeautorefname{pro}{Property}
\makeautorefname{sch}{Scholium}%
\makeautorefname{step}{Step}%
\makeautorefname{conject}{Conjecture}%
\makeautorefname{conj}{Conjecture}%
\makeautorefname{questn}{Question}
\makeautorefname{quest}{Question}
\makeautorefname{qn}{Question}
\makeautorefname{definition}{Definition}%
\makeautorefname{defn}{Definition}%
\makeautorefname{defi}{Definition}%
\makeautorefname{def}{Definition}%
\makeautorefname{dfn}{Definition}%
\makeautorefname{notation}{Notation}
\makeautorefname{notn}{Notation}
\makeautorefname{rem}{Remark}%
\makeautorefname{rems}{Remarks}%
\makeautorefname{rmk}{Remark}%
\makeautorefname{rk}{Remark}%
\makeautorefname{remarks}{Remarks}%
\makeautorefname{rems}{Remarks}%
\makeautorefname{rmks}{Remarks}%
\makeautorefname{rks}{Remarks}%
\makeautorefname{example}{Example}%
\makeautorefname{examp}{Example}%
\makeautorefname{exmp}{Example}%
\makeautorefname{exam}{Example}%
\makeautorefname{exa}{Example}%
\makeautorefname{axiom}{Axiom}%
\makeautorefname{axi}{Axiom}%
\makeautorefname{ax}{Axiom}%
\makeautorefname{case}{Case}%
\makeautorefname{claim}{Claim}%
\makeautorefname{clm}{Claim}%
\makeautorefname{assumpt}{Assumption}%
\makeautorefname{asses}{Assumptions}%
\makeautorefname{conclusion}{Conclusion}%
\makeautorefname{concl}{Conclusion}%
\makeautorefname{conc}{Conclusion}%
\makeautorefname{cond}{Condition}%
\makeautorefname{const}{Construction}%
\makeautorefname{con}{Construction}%
\makeautorefname{criterion}{Criterion}%
\makeautorefname{criter}{Criterion}%
\makeautorefname{crit}{Criterion}%
\makeautorefname{exercise}{Exercise}%
\makeautorefname{exer}{Exercise}%
\makeautorefname{exe}{Exercise}%
\makeautorefname{problem}{Problem}%
\makeautorefname{problm}{Problem}%
\makeautorefname{prob}{Problem}%
\makeautorefname{prob}{Problem}%
\makeautorefname{soln}{Solution}%
\makeautorefname{sol}{Solution}%
\makeautorefname{sum}{Summary}%
\makeautorefname{oper}{Operation}%
\makeautorefname{obs}{Observation}%
\makeautorefname{ob}{Observation}%
\makeautorefname{conv}{Convention}%
\makeautorefname{cvn}{Convention}%
\makeautorefname{warn}{Warning}%
\makeautorefname{note}{Note}%
\makeautorefname{fact}{Fact}%
\makeautorefname{ouch0}{Counterexample}%
\newtheorem{thm}{Theorem}[section]
\newtheorem{cor}{Corollary}[section]
\newtheorem{prop}{Proposition}[section]
\newtheorem{lem}{Lemma}[section]

\theoremstyle{definition}
\newtheorem{defn}{Definition}[section]

\newtheorem{exmp}{Example}[section]
\newtheorem{notn}{Notation}[section]

\newtheorem{rem}{Remark}[section]

\numberwithin{equation}{section}
\numberwithin{figure}{section}

\makeatletter
\let\c@cor=\c@thm
\let\c@prop=\c@thm
\let\c@lem=\c@thm
\let\c@conj=\c@thm
\let\c@defn=\c@thm
\let\c@df=\c@thm
\let\c@exmp=\c@thm
\let\c@notn=\c@thm
\let\c@rem=\c@thm
\let\c@sch=\c@thm
\let\c@equation\c@thm
\makeatother

\newcommand{\R}{\mathbb R}

\newcommand{\Z}{\mathbb Z}

\newcommand{\Tub}{\textup{Tub}}

\newcommand{\id}{\textup{id}}
\newcommand{\colim}{\textup{colim}\,}

\newcommand{\ra}{\longrightarrow}

\newcommand{\sma}{\wedge}

\newcommand{\ti}{\widetilde}

\newcommand{\mc}{\mathcal}

\newcommand{\Diff}{\textup{Diff}\,}
\newcommand{\tipartial}{\widetilde\partial}

\newcommand{\Fun}{\textup{Fun}}

\newcommand{\sSet}{\textup{sSet}}
\newcommand{\ssSet}{\textup{ssSet}}
\newcommand{\mirror}{{mirror}}

\newcommand{\mr}{\textup{mr}}

\newcommand{\Sm}{\mathrm{Sm}}
\newcommand{\Emb}{\mathrm{Emb}}
\newcommand{\Mansm}{\mc M^\mathrm{Sm}}
\newcommand{\Man}{\mc M^\mathrm{Emb}}
\newcommand{\Manst}{\mc M^\mathrm{Stab}}
\newcommand{\Hanst}{\mc H^\mathrm{Stab}}

\newcommand{\Manth}{\mc M^\mathrm{\mc U\textup{-}Emb}}

\newcommand{\Top}{\textup{Top}}
\newcommand{\PL}{\textup{PL}}

\newcommand{\Sing}{\rm Sing}

\newcommand{\reg}{\rm reg}
\DeclareMathOperator{\im}{\rm im}

\newcommand{\sF}{\mathcal{F}}

\newcommand{\St}{\mathrm{St}}

\newcommand{\po}{\ar@{}[dr]|{\text{\pigpenfont R}}}
\newcommand{\pb}{\ar@{}[dr]|{\text{\pigpenfont J}}}

\newcommand{\orange}[1]{{\color{black}{#1}}}

\newcommand{\sbt}{\,\begin{picture}(-1,1)(0.5,-1)\circle*{1.8}\end{picture}\hspace{.05cm}}

\newcommand{\ourdefn}[1]{{\it #1}}

\newcommand{\rh}{encasing function }
\newcommand{\therh}{the encasing function }

\newcommand{\arhperiod}{an encasing function. }

\newcommand{\rhs}{encasing functions }
\newcommand{\rhsperiod}{encasing functions. }

\newlength{\storeparskip}
\title{On the functoriality of the space of equivariant smooth $h$-cobordisms}
\author{Thomas Goodwillie}
\address{Department of Mathematics, Brown University}
\email{thomas\_goodwillie@brown.edu}
\author{Kiyoshi Igusa}
\address{Department of Mathematics, Brandeis University}
\email{igusa@brandeis.edu}
\author{Cary Malkiewich}
\address{Department of Mathematics, Binghamton University}
\email{cmalkiew@binghamton.edu}
\author{Mona Merling}
\address{Department of Mathematics, The University of Pennsylvania}
\email{mmerling@math.upenn.edu}

\begin{document}

\maketitle

\begin{abstract}
We construct an $(\infty,1)$-functor that takes each smooth $G$-manifold with corners $M$ to the space of equivariant smooth $h$-cobordisms $\mathcal H_{\Diff}(M)$. We also give a stable analogue $\mathcal H^{\mc U}_{\Diff}(M)$ where the manifolds are stabilized with respect to representation discs. The functor structure is subtle to construct, and relies on several new ideas. In the non-equivariant case $G=e$, our $(\infty,1)$-functor agrees with previous constructions of the smooth $h$-cobordism space as a functor to the homotopy category.
\end{abstract}

\begingroup%
\setlength{\parskip}{\storeparskip}
\makeatletter
\def\l@subsection{\@tocline{2}{0pt}{2.5pc}{5pc}{}}
\def\l@subsubsection{\@tocline{2}{0pt}{5pc}{7.5pc}{}}
\makeatother
\tableofcontents
\endgroup%

\section{Introduction}

The celebrated parametrized $h$-cobordism theorem, envisioned by Waldhausen and brought to fruition in seminal work of Waldhausen, Jahren, and Rognes, states that the stable space of piecewise-linear or topological $h$-cobordisms on a manifold $M$ is equivalent to the fiber of the $A$-theory assembly map,
\[ \xymatrix{
	\mc H^\infty_{\PL}(M) \simeq \mc H^\infty_{\Top}(M) \ar[r] &
	\Omega^\infty(A(*) \sma M_+) \ar[r] &
	\Omega^\infty A(M),
} \]
\orange{and this fiber sequence is natural in $M$.} Furthermore, the stable space of smooth $h$-cobordisms on a smooth manifold $M$ fits in a similar fiber sequence
\[ \xymatrix{
	\mc H^\infty_{\Diff}(M) \ar[r] &
	\Omega^\infty\Sigma^\infty_+ M \ar[r] &
	\Omega^\infty A(M).
} \]
\cite{wjr} gives a precise and detailed proof of the stable parametrized $h$-cobordism theorem in the PL case, which is then used to deduce the smooth version. 

\orange{The functoriality of the $h$-cobordism spaces $\mc H_{\Top}(M)$ and $\mc H_{\PL}(M)$ and their stabilizations, on the category of topological or PL manifolds and embeddings, is easy to establish. However, the functoriality of $\mc H_{\Diff}(M)$ and its stabilization $\mc H^\infty_{\Diff}(M)$, on the category of smooth manifolds and smooth embeddings, is much more subtle and there does not seem to be a complete treatment in the literature. \cite{hatcher_concordance, waldhausen_manifold} provide sketches of how to define $\mc H_{\Diff}(M)$ as a functor to the homotopy category. Even defining the stabilization maps $\mc H_{\Diff}(M) \to \mc H_{\Diff}(M \times I)$ in the smooth case is a delicate problem, which is treated carefully in \cite{igusa}. }


To illustrate the problem, let us describe the standard method for making the unstable space of smooth $h$-cobordisms $\mc H_{\Diff}(M)$ into a functor on smooth manifolds and smooth embeddings. Given a smooth $h$-cobordism $W_0$ over $M_0$, and an embedding $M_0 \to M_1$ with normal bundle $\nu$, we define a new $h$-cobordism $W_1$ on $M_1$ by taking the fiber product $W_0 \times_{M_0} D(\nu)$, an $h$-cobordism over $D(\nu)$. It is not trivial on the sides, so we ``pull up'' a collar of the bottom to make it trivial. Equivalently, we bend the fiber product $W_0 \times_{M_0} D(\nu)$ into a U-shape and glue in trivial regions above and below, as shown in \autoref{fig:intro_stab}. (This idea was first introduced in \cite{igusa}.)

\begin{figure}[h]
	\centering
	\def\svgwidth{4.2in}
	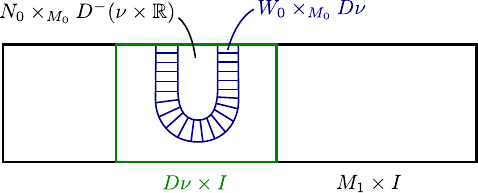
	\caption{\orange{The U-shaped stabilization of an $h$-cobordism $W_0$ from $M_0$ to $N_0$ along an embedding $M_0 \to M_1$ with normal bundle $\nu$. Each strip in the $U$-shape given by the disk bundle $D\nu$ is a copy of $W_0$. In the top region given by the lower hemisphere disk bundle $D^-(\nu\times \R)$ we glue in copies of the top manifold $N_0$}.}\label{fig:intro_stab}
\end{figure}

This depends on choices, but the choices form a contractible space. So, we get a well-defined homotopy class of maps
\[ \mc H_{\Diff}(M_0) \to \mc H_{\Diff}(M_1). \]
For composable embeddings $M_0 \to M_1 \to M_2$, we need to show that these maps respect composition up to homotopy. If one ignores smooth structure, then this is not too difficult. Morally, the choice of smooth structure on the tricky bits is contractible, and so it should be possible to show that the rule respects composition up to homotopy. If we can prove this, then we stabilize by repeatedly embedding the manifolds $M$ into $M \times I$. 

This may seem like a satisfactory sketch, but there is a significant issue. Even if all of the steps described above are accomplished and written down in detail, it only defines $\mc H^\infty_{\Diff}(M)$ as a functor \emph{to the homotopy category} of spaces. To get the full strength of the naturality result in \cite[Thm 0.3]{wjr}, it is necessary to have a functor to the actual category of spaces, or at least an $(\infty,1)$-functor to the $(\infty,1)$-category of spaces. In other words, our functor does not have to respect the composition $M_0 \to M_1 \to M_2$ strictly, only up to a homotopy. If we then take a composite of three embeddings and the homotopies between all the two-fold compositions, those homotopies have to be coherent with each other. And so on.

This makes the problem easier, but even so, the sketch given above does not lead to a proof that $\mc H^\infty_{\Diff}(-)$ is an $(\infty,1)$-functor. It is not enough to know that the choices of data are contractible---one has to link the contractible choices together, showing that they are preserved under composition. And the most obvious ways of defining the contractible choices, e.g. choosing collars for $W_0$ and choosing the shape of the U-band, turn out to \emph{not} be closed under composition, so they do not give an $(\infty,1)$-functor structure. We therefore have a nontrivial problem, that requires new ideas to solve.

Our main theorem is as follows. Since $(\infty,1)$-functors can be modeled by simplicially enriched functors, we state the result in terms of simplicially enriched functors. Let $G$ be a finite group, and let $\mc H_{\Diff}(M)$ denote the space of $G$-equivariant $h$-cobordisms over a compact $G$-manifold $M$. Let $\Man$ be the simplicial category of compact smooth $G$-manifolds and smooth equivariant embeddings.

\begin{thm}\label{unstable_intro}
There is a simplicially enriched functor $$\mc H_{\Diff}(-)\colon \Man \to \sSet\, ,$$ sending each compact $G$-manifold $M$ to a space equivalent to $\mc H_{\Diff}(M)$, and each homotopy class of equivariant embeddings $M_0 \to M_1$ to the homotopy class of maps $\mc H_{\Diff}(M_0) \to \mc H_{\Diff}(M_1)$ given by the stabilization depicted in \autoref{fig:intro_stab}.
\end{thm}

Stabilizing the input $M$ by representation discs, we get a second functor $\mc H^{\mc U}_{\Diff}(M)$. We also prove that this stable $h$-cobordism space extends to all $G$-spaces:

\begin{thm}\label{stable_intro}
The functor $$\mc H^{\mc U}_{\Diff}(-)\colon \Man \to \sSet\, ,$$ extends up to equivalence to a functor on the simplicial category of all $G$-CW complexes and equivariant continuous maps.
\end{thm}

In the non-equivariant case $G = e$, \autoref{unstable_intro} is closely related to the main result of the unpublished thesis \cite{pieper}. Pieper describes an $(\infty,1)$-functor structure on the smooth pseudoisotopy space, using an elaborate obstruction theory to show that one can simultaneously interpolate between different composites of U-bands. He obtains as a result the naturality of the $h$-cobordism splitting
\[ A(X) \simeq Wh_{\Diff}(X) \times \Omega^\infty\Sigma^\infty_+ X, \qquad \mc P^\infty_{\Diff}(X) \simeq \Omega^2 Wh_{\Diff}(X). \]
It is a dense and technically impressive treatment, and unfortunately it appears that it will remain unpublished.

 Our motivation for the current paper comes from current work of the first two authors on equivariant Reidemeister torsion, and separately of the last two authors on an equivariant stable parametrized $h$-cobordism theorem \cite{CaryMona, CaryMona3}. Both of these projects require a functor structure on $\mc H_{\Diff}(M)$ in the equivariant case $G \neq e$. The prospect of expanding the treatment in \cite{pieper} to include equivariance seems daunting. Instead, we give a new approach that develops the $(\infty,1)$-functor structure in a simpler and more streamlined way.

We highlight three key ideas that make our approach work.

Instead of using U-shaped bands to stabilize cobordisms, we use polar stabilization, as pictured in \autoref{fig:intro_polar_stab}.\footnote{This is diffeomorphic to the stabilization via U-bands, as we can see by taking a collar on the top and bottom of the cobordism $W_0$.}
\begin{figure}[h]
	\centering
	\def\svgwidth{4.0in}
	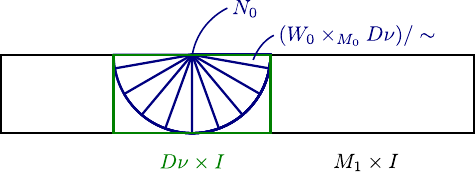
	\caption{\orange{The polar stabilization of an $h$-cobordism $W_0$ from $M_0$ to $N_0$ along an embedding $M_0 \to M_1$ with normal bundle $\nu$. As in \autoref{fig:intro_stab}, each strip along the U-shape given by the disk bundle $D\nu$ is a copy of $W_0$, but now the copies of the top manifold $N_0$ are identified at the cone point.}}\label{fig:intro_polar_stab}
\end{figure}
This is an old idea, but there is a new feature: \orange{we equip the $h$-cobordisms} with additional structure, ensuring that its polar stabilization is smooth and has the same kind of additional structure. \orange{This allows the stabilization to be} repeated multiple times without re-choosing collars, making it feasible to check coherence directly instead of, for example, using a complicated and indirect obstruction theory as in \cite{pieper}.

The extra structure is easy to describe: it is a choice of smooth structure on the manifold obtained by doubling $W_0$ along its top $N_0$, that extends the original smooth structure on each half and is preserved by the involution that interchanges the two halves. We call this a ``mirror'' structure. It exists and it is unique up to a contractible choice. The stabilization of a mirror $h$-cobordism is again canonically a mirror $h$-cobordism.

The second key idea is that, when working with normal bundles of embeddings $M_0 \to M_1$, we relax their structure, treating them as ``round'' disk bundles rather than the usual linear disk bundles. Round bundles are smooth disk bundles for which the structure group is the group of diffeomorphisms of the disk that preserve distance to the center---see \autoref{sec:round}. Round bundles have the advantage that they compose in a natural way along successive embeddings $M_0 \to M_1 \to M_2$, whereas vector bundles require additional contractible choices, and these choices are not easily made to be closed under composition.

One might imagine that switching to round bundles makes it more difficult to define the stabilization maps $\mc H_{\Diff}(M_0) \to \mc H_{\Diff}(M_1)$. On the contrary, it is actually a little easier. Pictorially, this means that it is more appropriate to think of the above polar stabilization using concentric circles, rather than rays, and to think of each circle as holding the points in $W_0$ that are at a single ``height'' along the cobordism.

\begin{figure}[h]
	\centering
	\def\svgwidth{4.0in}
	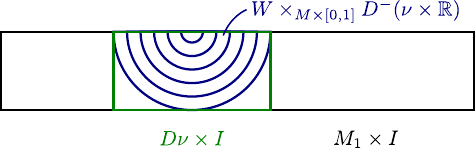
	\caption{Polar stabilization defined using round bundles.}\label{fig:intro_polar_stab_2}
\end{figure}

Lastly, we make use of the straightening-unstraightening theorem, originally due to Lurie \cite[3.2.2]{lurie_htt}. It allows us to avoid defining the $(\infty,1)$-functor directly by specifying the spaces, maps, and coherent sets of homotopies. Rather, we define it indirectly by giving a fibration of $\infty$-categories with the correct lifting properties. The desired maps and homotopies then arise by universal properties.

To be more precise, we define the $h$-cobordism functor by specifying a map from a category object in simplicial sets to a simplicially enriched category,
\begin{equation}\label{intro_left_fibration}
	\Hanst \to \Manst.
\end{equation}
Taking nerves gives a map of Segal spaces, which we show is a left fibration. It follows from a version of the straightening-unstraightening theorem, specificially the version in \cite{pedro,nima}, that this is equivalent to a simplicial functor from $\Manst$ to spaces. Since $\Manst$ is equivalent to the category of smooth manifolds and smooth embeddings, this gives the desired functor $\mc H_{\Diff}(-)$.

\autoref{unstable_intro} and \autoref{stable_intro} give all of the functoriality one could hope for in the space of smooth $h$-cobordisms. In a subsequent paper, we use this functoriality to prove that the equivariant stable space of $h$-cobordisms splits into a product of non-equivariant stable $h$-cobordism spaces. The proof of this splitting takes a long detour through \emph{isovariant} $h$-cobordism spaces, whose definition generalizes that of this paper, but with several additional subtleties.  The construction of the stable equivariant $h$-cobordism space and the forthcoming splitting result are central to work in progress of the first two authors on equivariant Reidemeister torsion, and of the latter two authors on the equivariant stable parametrized $h$-cobordism theorem \cite{CaryMona, CaryMona3}. 


\subsection{Outline}
In \autoref{prelims}, we establish some definitions and results concerning $G$-manifolds with corners. 

In \autoref{pseudosection}, we describe the polar stabilization of pseudoisotopies. Athough the main focus of this paper is $h$-cobordisms, it is easiest to start with pseudoisotopies. Several technical lemmas in \autoref{pseudosection} are used again in \autoref{hcobsection}.

In \autoref{hcobsection}, we define a space of equivariant $h$-cobordisms on a smooth $G$-manifold with corners $M$. We also consider variants of this space in which the cobordisms are equipped with collars or mirror structures, and we show that these have the same homotopy type. 

In \autoref{sec:stab}, we define the polar stabilization of smooth $h$-cobordisms. The construction of the smooth structure on this stabilization, and the proof that successive stabilizations give compatible smooth structures, form the technical core of the paper. As mentioned in the introduction, this requires using cobordisms that are equipped with ``mirror structures'', and using tubular neighborhoods that are equipped with ``round bundle'' structures.

In \autoref{infinitysec}, we recall the notion of a left fibration of Segal spaces, following \cite{pedro,nima}.
We then construct the map of simplicial categories \eqref{intro_left_fibration} and take the associated left fibration, whose fibers are the $h$-cobordism spaces. 
\numberwithin{equation}{section}
\numberwithin{figure}{section}

Lastly, in \autoref{hcobspacesec}, we stabilize the equivariant $h$-cobordism space with respect to all $G$-representations, and we extend the resulting functor from smooth $G$-manifolds to $G$-CW spaces.

\subsection{Acknowledgments}
We thank Wolfgang L{\"u}ck for encouraging us to embark on this project of carefully working out the functoriality of the smooth $h$-cobordism space, and for making us aware of this gap in the literature. It is a pleasure to acknowledge contributions  to this project arising from conversations with Mohammed Abouzaid, Julie Bergner, David Carchedi,  Sander Kupers, Wolfgang L{\"u}ck, 
Nima Rasekh, Emily Riehl, Hiro Lee Tanaka and Shmuel Weinberger. We especially thank  David Carchedi and Nima Rasekh  for directing us to the  higher categorical results used in \autoref{infinitysec}. We are very grateful to Dennis DeTurck, Herman Gluck, Ziqi Fang, Robert Kusner, and Leando Lichtenfelz for helping us think about \autoref{frechet}. We also thank Malte Pieper for sharing insights from his thesis with us. 

Igusa was partially supported by Simons Foundation Grant 686616, Malkiewich was partially supported by NSF DMS-2005524 and DMS-2052923, and Merling was partially supported by NSF grants DMS-1943925 and DMS-2052988. This material is in part based on work supported by the National Science Foundation under DMS-1928930 while Merling was residence at the Simons Laufer Mathematical Sciences Institute (previously known as MSRI) in Berkeley, California, during the Fall 2022 semester. Lastly, Malkiewich and Merling thank the Max Planck Institute in Bonn, where they were in residence in the fall of 2018, and where the origins of this project can be traced back to.

\section{Preliminaries on $G$-manifolds with corners}\label{prelims}

In this section we collect some needed technical results on $G$-manifolds with corners, where $G$ is a finite group. These include existence of tubular neighborhoods for smooth embeddings, approximation of continuous maps by smooth maps, existence of collars, and extension of smooth maps from the boundary, or a part of the boundary, to an open neighborhood.

We claim no originality here---the results are either well-known or straightforward adaptions of existing arguments. We include them because the specific versions that we need, in the presence of both corners and a $G$-action, can be difficult to find. We also highlight in \autoref{smooth_extension_counterexample} a surprising difficulty that arises in smoothly extending a map from the boundary to a neighborhood of the boundary.

\subsection{$G$-manifolds with corners} Throughout, a \ourdefn{manifold} of dimension $n$ is a smooth manifold with corners, i.e., a Hausdorff second-countable topological space $M$ with a maximal smooth atlas locally modeled on $[0,\infty)^n$, or equivalently on $[0,\infty)^k \times \R^{n-k}$ for $0 \leq k \leq n$. 

We define smoothness for maps from a manifold, or any subset of a manifold, to a manifold.

 \begin{defn}\label{smooth} \orange{Let $S$ be a subset of $\R^n$. A function $f:S\to \R^m$ is \ourdefn{smooth} if every point $x\in S$ has an open neighborhood $U\subset \R^n$ such that the restriction of $f$ to $S\cap U$ extends to a smooth map $U\to \R^m$. A function from $S$ to a subset $T \subseteq \R^m$ is smooth if it is smooth when regarded as a map from $S$ to $\R^m$. A map from a subset of an $n$-manifold to a subset of an $m$-manifold is smooth if, with respect to coordinate charts, it corresponds everywhere to a smooth map from a subset of $[0,\infty)^k \times \R^{n-k}$ to a subset of $[0,\infty)^\ell \times \R^{m-\ell}$.}
  \footnote{This is the most natural generalization of smoothness to manifolds with corners, as defined in \cite[Section 1.2.1]{cerf}. However in \cite{joyce} this notion is only called \ourdefn{weakly smooth}. In \cite{melrose}, a smooth function on an open subset $\Omega$ of $[0,\infty)^k \times \R^{n-k}$ is defined as a smooth function on the interior, all of those derivatives extend continuously to $\Omega$. This definition is equivalent to ours by \cite[Theorem 1.4.1]{melrose}.}
\end{defn}

The definition above will be useful when the domain of the map is an open subset of the manifold, and also when it is the boundary, or some part of the boundary, and also in some other cases.

\begin{defn}
The \ourdefn{depth} of a point $x$ in $M$ is the unique number $k$ such that there is a coordinate chart that identifies $x$ with the origin in an open subset of $[0,\infty)^k \times \R^{n-k}$. The set of all depth $0$ points is the \ourdefn{interior} of $M$. The set of all depth $\geq 1$ points is the \ourdefn{boundary subspace} $\partial M \subset M$. Depth $\geq 2$ points are called  \ourdefn{corner points}, and $\partial^c M \subset \partial M$ is the set of corner points.\end{defn}

Notice that although $\partial M$ is a topological manifold it does not qualify as a smooth manifold with corners. (The corner points of $M$ lie in its interior.)

\begin{exmp}
	A product of two or more manifolds with corners has the structure of a manifold with corners in a canonical way. For example, the $n$-dimensional cube $I^n$ and the polydisc $D^{n_1} \times \ldots \times D^{n_k}$ are manifolds with corners.
\end{exmp}	

\begin{exmp}	The standard $k$-simplex $\Delta^k$, with smooth structure inherited from $\R^{k+1}$, is a manifold with corners. The affine linear embedding $\Delta^k \to \R^k$ that takes one vertex to the origin and the remaining vertices to the standard basis vectors defines a diffeomorphism between the complement of one face in $\Delta^k$ and an open subset of $[0,\infty)^k$.
\end{exmp}

Because of \autoref{smooth}, the tangent space of a manifold with corners $M$ at a point $x \in M$ can be defined in the usual way, 
even if the point is not in the interior:

\orange{
\begin{defn}\label{tangent_bundle}
	For $x \in M$, the \ourdefn{tangent space} $T_x M$ is the set of equivalence classes of data $(U,\phi,v)$ where $U \subseteq M$ is an open set containing $x$, $\phi\colon U \to [0,\infty)^k \times \R^{n-k}$ is a coordinate chart, and $v \in \R^n$ is a vector (\emph{not} necessarily lying in $[0,\infty)^k \times \R^{n-k}$). The data $(U,\phi,v)$ and $(U',\phi',v')$ are equivalent if the derivative of the transition function $D_{\phi(x)}(\phi' \circ \phi^{-1})$ takes $v$ to $v'$. The \ourdefn{tangent bundle} $TM$ is the set of all tangent vectors at all points $x \in M$, made into a smooth manifold with corners by using the local charts of the form $[0,\infty)^k \times \R^{n-k} \times \R^n$.
\end{defn}
}

A \ourdefn{diffeomorphism} is a smooth map with smooth inverse, i.e. a homeomorphism that identifies the maximal atlases. The following is an easy consequence of the inverse function theorem.

\begin{lem}\label{inverse_fn_thm}
	A map $M \to N$ of smooth manifolds with corners is a diffeomorphism iff it is a smooth bijection and has invertible first derivative at every point of $M$ (including the boundary and corner points).
\end{lem}

Following \cite{joyce}, we define a manifold called the \ourdefn{smooth boundary} of $M$, whose interior may be identified with the set of depth $1$ points. This will be used later when discussing faces. \orange{For any point $x\in M$, we define the set of \ourdefn{local boundary components} near $x$ to be the inverse limit, over neighborhoods $U$ of $x$, of $\pi_0(U\cap (\partial M\backslash \partial^c M))$. Of course, there are arbitrarily small neighborhoods $U$ such that $\pi_0(U\cap (\partial M\backslash \partial^c M))$ coincides with this limit. Note that the depth of $x$ is the number of local boundary components near $x$.}

The smooth boundary $\tipartial M$ is the set of pairs $(x,b)$ where $x \in M$ and $b$ is a local boundary component at $x$. This inherits a smooth atlas from $M$ so that $\tipartial M$ becomes a smooth manifold with corners \cite[Definition 2.6.]{joyce}.\footnote{We warn the reader that we are using different notation than in \cite{joyce}, where the subspace boundary is our $i(\tipartial M)$, and where $\partial M$ is defined as the smooth boundary, which we call $\tipartial M$. For us the subspace boundary $\partial M$ will play a key role in some of the technicalities, even though it is not a smooth manifold.} The map $i\colon \tipartial M \to M$ that forgets the local boundary component is smooth, and its image is $\partial M$. It is not injective if $M$ has corner points.

\begin{figure}[h]
	\centering
	\def\svgwidth{2.0in}
	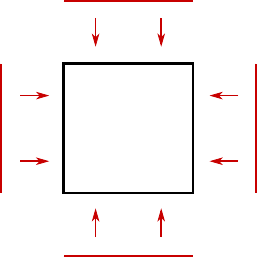
	\vspace{-.5em}
	\caption{The smooth boundary of a square consists of four line segments.}\label{fig:faces}
\end{figure}

\begin{defn}\label{Gcornerdef}
Suppose that $M$ is a manifold with corners, and that $G$ is a finite group acting smoothly on $M$. We say that the action is \orange{\ourdefn{trivial along corners}} if any of the following equivalent conditions hold.
	\begin{itemize}	
		\item For each boundary point $x \in \partial M$, the action of its stabilizer $G_x$ on the set of local boundary components near $x$ is trivial.
		\item $M$ is locally modeled by $G \times_H V \times [0,\infty)^k$, for varying $H \leq G$ and $H$-representations $V$.
		\item $M$ is locally modeled by finite products of smooth $G$-manifolds with boundary.
	\end{itemize}
\end{defn}

For example, a product $D(V_1)\times \ldots \times D(V_n)$, where each $D(V_i)$ is the disk in a $G$-representation, has $G$-trivial action on corners, while the product $I\times I$ \orange{with the $C_2$-action} $(x,y)\mapsto (y,x)$ does not. 

\begin{defn}
	Throughout this paper, a \ourdefn{$G$-manifold with corners} will be a manifold $M$ with corners, with a smooth $G$ action that is \orange{trivial along corners.}
\end{defn}

 \orange{
\begin{rem}
Many of the results here would hold more generally without assuming that the action is trivial along corners. We make this assumption for several reasons. First, it holds in the examples we encounter, so the extra generality is not needed. Second, it gives a simpler local model for $G$-manifolds---otherwise we would have to consider $G$-invariant corner subspaces of $G$-representations. Above all, it is needed in the sequel paper because it guarantees that for each subgroup $H \leq G$ the set of $H$-fixed points $M^H$  is a neat submanifold of $M$ in the sense of \autoref{neatdef}.
\end{rem}}



As usual, a continuous or smooth map between $G$-spaces or $G$-manifolds is said to be \ourdefn{equivariant} if it commutes with the $G$-action.

Smooth vector fields $\xi$ on a manifold with corners $M$ are defined in the obvious way, as smooth sections of the tangent bundle. In the presence of a group action we typically only consider $G$-invariant \orange{(i.e. equivariant) }vector fields.

\orange{
\subsection{Existence of smooth extensions}



In view of  \autoref{smooth}, we may speak of a smooth map $\partial M \to N$ even though $\partial M$ is not a smooth manifold. 

We need results guaranteeing the existence of a smooth extension to a neighborhood of $\partial M$ under certain circumstances. The following example shows that there can be a difficulty with this if a corner point of $M$ is mapped to a boundary point of $N$.


\begin{rem}\label{smooth_extension_counterexample}
	It may happen that a smooth map $f\colon \partial M \to N$ cannot be smoothly extended to any neighborhood of $\partial M$ in $M$ in such a way that the extended map still takes values in $N$. For example, let $M$ be $ [0,\infty)^3$, let $N$ be $ [0,\infty)$, and consider the quadratic form
	\[ q(x_1,x_2,x_3)=x_1^2+x_2^2+x_3^2-c(x_1x_2+x_1x_3+x_2x_3) \]
	where $1<c<2$. Since $c < 2$, $q$ maps $\partial M$ into $N$. On the other hand, any smooth extension of $q$ to a neighborhood must shoot past the boundary of $N$. In fact, any real-valued function $f(x_1,x_2,x_3)$ defined on a neighborhood of the origin in $\mathbb R^3$ and agreeing with $q$ at all points in $\partial M$ must agree with $q$ to second order at the origin and, since $c>1$, this implies that $f(t,t,t)$ is negative for sufficiently small positive values of $t$. Therefore, even though $f$ might extend to a smooth map $[0,\infty)^3 \to \R$, there can be no such extension, even locally, that takes values in $[0,\infty)$.
\end{rem}


For the moment we will confine attention to cases in which the map to be extended takes values in the interior of the target manifold.

\begin{lem}\label{smooth_on_arbitrary_G}
	If $S \subseteq M$ is a $G$-invariant subset and $f\colon S \to N$ is both smooth (in the sense of \autoref{smooth}) and equivariant, and takes values in $\textup{int}\ N$, then every point in $S$ has a $G$-invariant open neighborhood $U$ such that the restricted map $S\cap U\to \textup{int}\ N$ has an equivariant smooth extension to $U$.
\end{lem}


\begin{proof}
	Suppose that $x\in S$. Let $G_x$ and $G_{f(x)}$ be the stabilizers of $x$ and $f(x)$. Note that $G_x \leq G_{f(x)}$. The essential idea is to average a non-equivariant extension of $f$ over $G_x$.
	
	Note that in some coordinate chart of the form $G \times_{G_x} V \times [0,\infty)^k$ the point $x$ is given by the identity of $G$ and the origin of $V \times [0,\infty)^k$. The point $f(x)$ occurs in the same way in a coordinate chart of the form $G \times_{G_{f(x)}} W \times [0,\infty)^\ell$. Restrict $f$ to $V \times [0,\infty)^k$, where it gives a map
	\[ f\colon S' \to W \times [0,\infty)^\ell \]
	where $S' \subseteq V \times [0,\infty)^k$ is a subset containing the origin. Since $f$ is smooth at the origin, it admits a non-equivariant smooth extension $h$, defined on open subset $U \subseteq V \times \R^k$ containing the origin. Define a map $\tilde f$ on the $G_x$-invariant open set $U' = \bigcap_{g \in G_x} gU$ and taking values in the larger space $W \times \R^\ell$, by the formula
	\[ \tilde f = \frac{1}{|G_x|} \sum_{g \in G_x} g \circ h \circ g^{-1} (-). \]
	Then take the unique $G$-equivariant extension of $\tilde f$ from $U' \subseteq V \times \R^k$ to $G \times_{G_x} U' \subseteq G \times_{G_x} V \times \R^k$; it is easy to see that this is smooth as well.
\end{proof}


Next we use this local extension result to deduce a more global one:

\begin{lem}\label{smooth_local_to_global}
	If $M$ and $N$ are $G$-manifolds with corners, $S \subseteq M$ is an arbitrary $G$-invariant subset, and $f\colon S \to \textup{int }N$ is equivariant and smooth, then there exists a $G$-invariant open subset $U \subseteq M$ containing $S$ and an equivariant smooth extension $\tilde f\colon U \to \textup{int }N$.
\end{lem}

\begin{proof}
	For simplicity, we first describe the argument without the $G$-action. Cover $f(S)$ by open sets that are coordinate charts for $\textup{int }N$. As a subspace of a manifold, $f(S)$ is paracompact, so we may replace this cover of $f(S)$ by a locally finite refinement $V_i$. Since $N$ has a countable basis, this implies that the collection is countable.
	
	For each $i$, cover $f^{-1}(V_i) \subseteq S$ by open sets $U$ of $ M$ such that $f|_{f^{-1}(V_i)}$ has a smooth extension to $U$ whose image is contained in $V_i$. Let $U_{ij}$ be a locally finite refinement of this open cover. Choose a smooth partition of unity on $U_i=\bigcup_j U_{ij}$ subordinate to the cover $U_{ij}$, and use this partition of unity to form a weighted sum (with respect to the coordinates in $V_i$) of the given extensions of $f|_{f^{-1}(V_i)}$. This extends $f|_{f^{-1}(V_i)}$ to a smooth map $\tilde h_i\colon U_i \to V_i$ defined on an open neighborhood of $f^{-1}(V_i)$. Do this for each $i$.
	
	Then combine the maps $\tilde h_i$ as follows. Since there are countably many $V_i$, they may be arranged in a sequence. Write $W_n=\bigcup_{i=1}^n U_i$. Define a smooth extension $\tilde f_n\colon W_n \to \textup{int }N$ of $f|_{W_n\cap S}$ by induction on $n$. To go from $n-1$ to $n$, use a partition of unity subordinate to the cover $\{ W_{n-1}\cap U_n, U_n \}$ of $U_n$ to form a weighted sum (in the coordinates of $V_n$) of the function $\tilde f_{n-1}|_{W_{n-1}\cap U_n}$ and the function $\tilde h_n$. This map $U_n\to N$ is smooth, and when it is extended to $W_n$ by $\tilde f_{n-1}|_{W_{n-1}\backslash U_n}$ it remains smooth. This is $\tilde f_n:W_n\to N$.
	
	Finally, write $W=\bigcup_nW_n$. The maps $\tilde f_n$ yield an extension $\tilde f:W\to N$, because the cover of $f(S)$ by $V_i$ is locally finite; for each point $x \in S$ there is a neighborhood on which $\tilde f_n$ stops changing after finitely many steps. 
	
	To prove the statement with $G$-action, we have only to choose our coordinate charts, open sets, and partitions of unity to be $G$-invariant, and to construct the local extensions to be equivariant using \autoref{smooth_on_arbitrary_G}. This ensures that the extension $\tilde f$ is equivariant, and that its domain is a $G$-invariant open neighborhood of $S$.
\end{proof}

In the special case in which $S$ is the boundary of $M$, smoothness of a map can be verified by breaking the boundary up into its local components.

\begin{lem}\label{smooth_on_boundary}
	If $f\colon \partial M \to N$ is such that $f \circ i\colon \ti\partial M \to N$ is smooth, then $f$ is smooth. 
\end{lem}
}

\begin{proof}
	Since smoothness is defined locally, we may assume that $f$ is a map
	\[ \partial([0,\infty)^k) \times \R^{n-k} \to \R^m, \]
	 and we wish to extend $f$ smoothly to $[0,\infty)^k \times \R^{n-k}$. Let $V_1,\ldots,V_k$ be the subspaces of the domain obtained by restricting one of the first $k$ coordinates to 0. Then $f$ is defined on the union of the $V_i$, and is smooth on each $V_i$ separately.
	
	The question is unaffected if we subtract a function that admits a smooth extension to $[0,\infty)^k \times \R^{n-k}$. One such function is given by $f_1(x_1,\dots ,x_n)=f(0,x_2,\dots ,x_n)$. In other words, we project the first coordinate to 0 and then apply $f$. Replace $f$ by $f - f_1$. It now vanishes on $V_1$. Repeat with a second coordinate. Now the function vanishes on $V_2$ while still vanishing on $V_1$. Repeating with the remaining coordinates, $f$ now vanishes on every $V_i$, and therefore is zero, so it can be extended to the rest of $[0,\infty)^k \times \R^{n-k}$ by the zero function.
	
	\orange{To put it another way, the original function $f$ is the sum of the functions $f_i$ that we have inductively defined. Each $f_i$ smoothly extends to $[0,\infty)^k \times \R^{n-k}$, and so the same is true of $f$. Thus $f$ is smooth. }
	
\end{proof}

\orange{
We will need a slight generalization of the previous result.

\begin{defn}\label{tame_def}
	A subset $A \subseteq \partial M$ is a \ourdefn{partial boundary} if 
	for every point $x \in A$ of depth $k$, there is some number $1\le a\le k$ such that in some coordinate chart $U 		\subseteq [0,\infty)^k \times \R^{n-k}$ the set $A\cap U$ corresponds to 
	\[ U \cap (\partial([0,\infty)^a) \times [0,\infty)^{k-a} \times \R^{n-k}). \]
	Equivalently, near each point of $A$, $A$ is a union of local boundary components. See \autoref{fig:partial_boundary}.
\end{defn}

Note that the preimage $\ti A \subseteq \ti \partial M$ is an open subset of $\ti \partial M$. It need not be a closed subset, but if it is then $\ti A $ is the union of some set of connected components of $\ti \partial M$.

\begin{lem}\label{smooth_on_boundary_tame}
	\orange{Suppose that $A \subseteq \partial M$ is a partial boundary and $f\colon A \to N$ is a map. If $f \circ i\colon \ti A \to N$ is smooth, then $f$ is smooth.}
\end{lem}

\begin{proof}
	The proof is the same as in \autoref{smooth_on_boundary}, except that in the local model $f$ is only defined on $V_1,\ldots,V_a$, and not on $V_{a+1},\ldots,V_k$. However we may apply the same argument, projecting onto each of the subspaces where $f$ is defined, composing with $f$, and then subtracting off this composite function. As before, this reduces the problem to one in which $f$ is zero along each of the subspaces where it is already defined, and then the zero function is a valid extension to the rest of $[0,\infty)^k \times \R^{n-k}$.
\end{proof}

\begin{figure}[h]
	\centering
	\def\svgwidth{4.5in}
	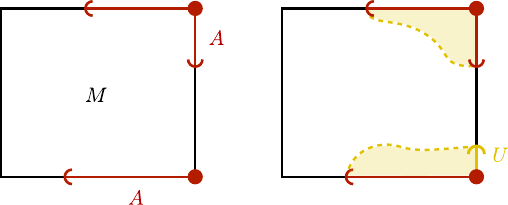
	\vspace{-.5em}
	\caption{A partial boundary $A \subseteq \partial M$ and an open neighborhood of $A$.}\label{fig:partial_boundary}
\end{figure}

Combining this with \autoref{smooth_local_to_global}, we get the main result of this subsection. When $A = \partial M$, it tells us that any equivariant map on $\partial M$ that is smooth on each ``face'' of the boundary extends to an equivariant smooth map on an open neighborhood of the boundary. As discussed before and demonstrated by \autoref{smooth_extension_counterexample}, the map must land in the interior of $N$ for this extension result to work.

\begin{prop}\label{smooth_extension}
	If $A \subseteq \partial M$ is a $G$-invariant partial boundary, $f\colon A \to \textup{int }N$ is an equivariant map, and $f \circ i\colon \ti A \to \textup{int }N$ is smooth, then $f$ extends to an equivariant smooth map $U \to \textup{int }N$ for some $G$-invariant open set $U$ containing $A$.
\end{prop}

We also record the variant of this result for vector fields, omitting the proof since it is similar to the above but significantly easier. We do need to assume the partial boundary $A$ is closed in order to extend the field to all of $M$ and not just to an open subset containing $A$. 

\begin{prop}\label{smooth_extension_vectors}
	If $A \subseteq \partial M$ is a closed $G$-invariant partial boundary, $\xi$ is a $G$-invariant vector field defined along $A$, and $\xi$ is smooth when restricted to each local boundary component near each point in $A$, then $\xi$ extends to a smooth, $G$-invariant vector field on all of $M$.
\end{prop}
}

\subsection{Embeddings, tubular neighborhoods, and submersions}
A \ourdefn{smooth embedding} $M \to N$ is a topological embedding that is a smooth map having full rank (the rank of the derivative is equal to the dimension of $M$) at every point. 
It is a \ourdefn{closed embedding} or \ourdefn{open embedding} if it is topologically a closed or open embedding, respectively. It is elementary that $i$ is a closed embedding if it is smooth, full rank, and injective, and if the source $M$ is compact. \orange{More generally, any map that is full rank and injective on a compact subset of $M$ extends to an embedding near that subset:

\begin{lem}\label{extend_embedding}
	If $f\colon M \to N$ is smooth, $K \subseteq M$ is any compact subset, and $f$ is injective on $K$ and has rank equal to $\dim M$ at each point of $K$, then $f|_U\colon U \to N$ is a smooth embedding for some open set $U$ containing $K$.
\end{lem}

\begin{proof} By a corollary of the inverse function theorem, if $f$ has full rank at $x$ then $f$ becomes an injection when restricted to some neighborhood of $x$. It follows that the following subset of $M\times M$ is open: all pairs $(x,y)$ such that either $f(x)\neq f(y)$ or else $x=y$ and the derivative at $x$ has full rank. 
\newline

That open set contains $K\times K$, so it contains $V\times V$ for some open set $V$ containing $K$. The restriction of $f$ to $V$ is injective and has full rank but is not necessarily a topological embedding. Choose $U\subset V$ to be a smaller open neighborhood whose closure in $V$ is compact. Then the restriction of $f$ to $U$ is a smooth embedding.

	
\end{proof}

}

\begin{figure}[h]
	\centering
	\def\svgwidth{3.5in}
	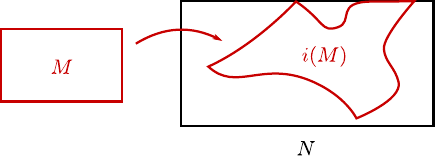
	\vspace{-.5em}
	\caption{A smooth codimension 0 embedding of manifolds with corners.}\label{fig:embedding}
\end{figure}

\begin{lem}\label{depthopenlemma}
	If $M$ and $N$ are (not necessarily compact) smooth manifolds with corners, both of the same dimension, then any smooth map $i\colon M \to N$ that has full rank and is injective and depth-preserving must be an open embedding.
\end{lem}

\begin{proof}
	We work locally at a point of depth $k$, so that without loss of generality $M$ and $N$ are neighborhoods of the origin in $[0,\infty)^k \times \R^{n-k}$ and $i$ preserves the origin. When $k = 0$, openness is a standard consequence of the inverse function theorem. For higher values of $k$, inductively we know that the restriction of $i$ to each face of $[0,\infty)^k \times \R^{n-k}$ is open, so that the restriction to the boundary $\partial([0,\infty)^k) \times \R^{n-k}$ contains a neighborhood of the origin. Therefore the image $i(\partial M)$ disconnects every sufficiently small neighborhood of the origin in $\R^n$.
	
	Since $i$ has full rank, it admits an extension to an open neighborhood of the origin in $\R^n$ that is a homeomorphism to an open neighborhood in $\R^n$. By restricting the size of this neighborhood on the exterior of $[0,\infty)^k \times \R^{n-k}$, we can ensure that this extended version of $i$ does not send exterior points to interior points. Therefore the restriction to $[0,\infty)^k \times \R^{n-k}$ is such that its image is the intersection of an open set in $\R^n$ and the subspace $[0,\infty)^k \times \R^{n-k}$, as desired.
\end{proof}

\orange{
%

\begin{lem}\label{embedding_is_diffeo}
	Suppose $M'$ is a manifold with corners, $M \subseteq M'$ is a compact connected codimension zero embedded manifold with corners, and $f\colon M \to M'$ is an embedding. If $f(\partial M)=\partial M$ and $f$ sends at least one point $x \in \textup{int }M$ into $\textup{int }M$, then $f(M)=M$ and the map $M \to M$ given by $f$ is a diffeomorphism.
\end{lem}

\begin{proof}
	By assumption $f$ is injective and has full rank, so it suffices to show that $f(M)=M$, that is, $f(\textup{int }M)=\textup{int }M$. For the inclusion $f(\textup{int }M)\subseteq \textup{int }M$, note that $f$ must map the connected set $\textup{int }M$ into the component of $M'\backslash \partial M$ that contains $f(x)$. This component is $\textup{int }M$. 
	
	For the reverse inclusion we use homology. Let $m$ be the dimension of $M$. The map $f$ induces horizontal maps as shown.
	\[ \xymatrix{
		H_{m}(M,\partial M;\Z/2) \ar[d]_-\cong \ar[r]^-{f_*} & H_{m}(M,\partial M;\Z/2) \ar[d]^-\cong \\
		H_{m}(M,M \setminus \{x\};\Z/2) \ar[r]^-{f_*} & H_{m}(M,M \setminus \{f(x)\};\Z/2)
	} \]
	Since $M$ is compact and connected, all of these groups are $\Z/2$. The bottom horizontal map is an isomorphism since $f$ is a codimension-zero embedding, and therefore the top horizontal map is an isomorphism as well.
	
	It follows that for each $y \in \textup{int }M$ the composed map 
	\[ H_{m}(M,\partial M;\Z/2) \xrightarrow{f_\ast} H_{m}(M,\partial M;\Z/2)\to H_{m}(M,M \setminus \{y\};\Z/2) \]
is an isomorphism. This can only be so if $y$ is in the image of $f$. 
\end{proof}
}

\begin{lem}\label{embed_into_V}
	If $M$ is a compact $G$-manifold with corners, then there is a closed embedding $M \to V$ into a sufficiently large orthogonal $G$-representation $V$.
\end{lem}

\begin{proof}
	A standard proof of the non-equivariant statement can be adapted as follows. \orange{Cover $M$ by a finite set of equivariant coordinate charts $\phi_i\colon U_i \to V_i$, where each $U_i$ is a $G$-invariant open subset of $M$, each $V_i$ is a $G$-representation, and $\phi$ is a diffeomorphism onto its image. Take a smooth partition of unity subordinate to this cover, $\lambda_i\colon M \to [0,1]$, and make it $G$-invariant by replacing each of its functions by the average $\frac{1}{|G|}\sum_{g \in G} \lambda_i \circ g$. The usual proof then constructs the embedding $M \to \bigoplus_i (V_i \oplus \R)$, whose coordinates are $(\lambda_i \cdot \phi_i, \lambda_i)$. As a result of our choices for $\phi_i$ and $\lambda_i$, this embedding is also $G$-equivariant.}
\end{proof}

\orange{
\begin{defn}\label{fiber_bundle}
	A \ourdefn{smooth $G$-equivariant fiber bundle} with fiber $F$ is a smooth $G$-equivariant map $p\colon E \to B$ such that for each $x \in B$ with isotropy group $G_x$ there is a $G_x$-invariant open neighborhood $U_x \subseteq B$, an action of $G_x$ on $F$, and a $G_x$-equivariant diffeomorphism of $p^{-1}(U_x)\cong F\times U_x$ preserving the product projection to $U_x$.
	
	Now suppose that a group $\Gamma$ is acting smoothly and faithfully on $F$. The bundle has \ourdefn{structure group} $\Gamma$ if there is a privileged class of trivializations as above, covering $B$, such that for each point $x$ the action of $G_x$ on $F$ is given by a homomorphism $G_x \to \Gamma$, and such that for any two of the trivializations the transition map $U_x\cap U_y\to \Diff(F)$ takes values in $\Gamma$.
	
	As a special case, a \ourdefn{smooth $G$-equivariant vector bundle} of rank $n$ is a smooth equivariant fiber bundle with fiber $\mathbb R^n$ and structure group $GL_n(\R)$. Equivalently, it is a smooth equivariant fiber bundle in which each fiber has the structure of an $n$-dimensional vector space, with $G_x$ acting linearly, and the $G_x$-equivariant smooth trivializations $p^{-1}(U_x) \cong F \times U_x$ can be chosen to be linear on each fiber.
\end{defn}

For example, the tangent bundle $TM \to M$ from \autoref{tangent_bundle} is a smooth equivariant vector bundle, with a trivialization for each $x \in M$ and $G_x$-equivariant coordinate chart containing $x$. We define} the \ourdefn{normal bundle} of a smooth embedding $i\colon M \to N$ as the quotient of tangent bundles $(i^*TN) / TM$. Notice that this \orange{does not require us to pick a metric on $TM$, and} makes sense even at boundary and corner points. If the embedding $i$ is equivariant then the normal bundle is a smooth equivariant vector bundle over $M$.

\orange{
\begin{defn}\label{tubular_nbhd}
For compact $G$-manifolds $M$ and $N$ and an equivariant embedding $i\colon M \to N$, a \ourdefn{tubular neighborhood} consists of a $G$-invariant inner product on the normal bundle $\nu$ of $i$ together with a codimension zero smooth equivariant embedding $\tilde i\colon D(\nu) \to N$ of the unit disk bundle, satisfying the following conditions. First, the restriction of $\tilde i$ to the zero section $M$ must be $i$. Second, the isomorphism from the vector bundle $\nu$ to itself obtained by differentiating $\tilde i$ at $M$ must be the identity. In other words, the first derivative of $\tilde i$ at a point $x\in M$ has the block form
\[ \begin{pmatrix} I & A \\ 0 & I \end{pmatrix} \]
when the tangent space to $D(\nu)$ at $x$ is split into tangent and normal directions using the inner product and the tangent space to $N$ at $x$ is split (somehow) into tangent and normal directions. 
\end{defn}

\begin{rem}\label{tubular_nbhd_alternate}
Given $i\colon M \to N$, any smooth $G$-vector bundle $\mu$ on $M$ with inner product and codimension zero smooth equivariant embedding $\tilde i\colon D(\mu) \to N$, coinciding with $i$ on the zero section, determines a tubular neighborhood. Indeed, the first derivative of $\tilde i$ determines an isomorphism $\mu\cong \nu$ of smooth $G$-vector bundles, and along this isomorphism $\nu$ receives an inner product and the induced embedding $D(\nu) \to N$ has the required condition on its first derivatives.
\end{rem}

\begin{figure}[h]
	\centering
	\def\svgwidth{2.5in}
	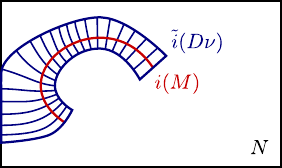
	\vspace{-.5em}
	\caption{A tubular neighborhood of manifolds with corners.}\label{fig:tubular_nbhd}
\end{figure}

We can also consider germs of tubular neighborhoods.
\begin{defn}\label{partial_tubular_nbhd}
If $M$ and $N$ are compact, a \ourdefn{partial tubular neighborhood} of the embedding $i:M\to N$ consists of a $G$-invariant open subset $U \subseteq \nu$ containing the zero section together with a codimension zero embedding $\tilde i:U\to N$ extending $i$ and inducing the identity map of $\nu$. A \ourdefn{tubular neighborhood germ} is an equivalence class of partial tubular neighborhoods, where two are equivalent if they agree in some neighborhood of the zero section. 
\end{defn}}

\begin{rem}\label{germ_to_full}
Of course, every tubular neighborhood determines a tubular neighborhood germ. Conversely, every tubular neighborhood germ is the germ of a tubular neighborhood. To see this, choose any inner product on $\nu$ and compose the given embedding $\tilde i:U\to N$ with an embedding $D(\nu)\to U\subseteq \nu$ that is the identity near the zero section. For example, we could construct such an embedding by applying a smooth embedding $[0,\infty) \to [0,\infty)$ to the radial coordinate, that is the identity near zero and that sends $[0,\infty)$ into $[0,\epsilon)$ for some $\epsilon $. \orange{(By assumption $M$ is compact.)}
\end{rem}

The usual proof of existence of tubular neighborhoods for embeddings into Euclidean space (e.g. \cite[Section 4.5]{hirsch}) applies in this equivariant setting:
\begin{lem}\label{weak_tubular_nbhd_thm}
	If $M$ is compact, every equivariant embedding into an orthogonal $G$-representation $V$ has a tubular neighborhood.
\end{lem}

\begin{proof}
	For this argument we identify the normal space of $M$ at $x$ (a quotient of tangent spaces) with the space of vectors in $V$ that are perpendicular to the tangent space of $M$ at $x$. \orange{So $\nu \subseteq M \times V$. Now define $\tilde i\colon \nu \to V$ by sending $(x,v) \in M \times V$} to $x+v \in V$. This map is clearly equivariant, and it is both full rank and injective along the zero section $M$. \orange{Therefore by \autoref{extend_embedding}} the map is full rank and injective in a \orange{neighborhood the zero section $U \subseteq \nu$. Furthermore its first derivative along the zero section induces the identity map on $\nu$,} and therefore it defines a tubular neighborhood germ. \orange{By \autoref{germ_to_full}, since a germ exists, a tubular neighborhood also exists.}
\end{proof}
	
Note that the above tubular neighborhood for $M \to V$ comes with a smooth retraction $p$ from a neighborhood \orange{of the interior of $M$ in $V$ back to $\textup{int }M$, defined by picking a $G$-invariant metric on $V$ and} sending every point to its nearest neighbor in $M$.
\orange{
This retraction can be used to prove smooth approximation in the compact case, as in \cite[Corollary 1.12]{wasserman}.
\begin{lem}\label{smooth_approximation}
	Any continuous equivariant map of compact $G$-manifolds $f\colon M \to \textup{int } N$ can be approximated by a smooth equivariant map. If $f$ is smooth on a neighborhood of a closed subset $C \subseteq M$ then the smooth map can be taken to agree with $f$ on $C$.
\end{lem}

\begin{proof}
	Fix an equivariant smooth embedding $N \to V$ in a representation and choose an equivariant smooth retraction $p:\Omega\to \textup{int }N$ as above. The domain $\Omega$ is a neighborhood of $\textup{int } N$. Choose $\epsilon > 0$ such that $\Omega$ contains the $\epsilon$-ball about each point in the closed set $f(M) \subseteq \textup{int }N$.
	
	Take any non-equivariant smooth approximation $\tilde f\colon M \to \textup{int } N$ rel $C$ such that $\|\tilde f(x) - f(x)\| < \epsilon$ for each $x$. Then the average of all the conjugates $g \circ \tilde f \circ g^{-1}$ is a map $M \to V$ that is also close to $f$:
	\begin{eqnarray*}
		\left\| \frac{1}{|G|} \sum_{g \in G} g\tilde fg^{-1}(x) - f(x)  \right\| 
		&\leq & \frac{1}{|G|} \sum_{g \in G} \left\| g\tilde fg^{-1}(x) - gfg^{-1}(x) \right\| \\
		&=& \frac{1}{|G|} \sum_{g \in G} \left \|\tilde f(g^{-1}x) - f(g^{-1}x) \right \| \\
		&<& \epsilon.
	\end{eqnarray*}
	Therefore this average lies in the open set $\Omega \subseteq V$, so we may compose it with the retraction $p$ to get the desired equivariant approximation to $f$. On each point of $C$, $\tilde f = f$, and so this average returns the function $f$ again.
\end{proof}
}

\orange{
We also record the same result for vector fields, whose proof is similar but easier, since we can just average the non-equivariant approximations and do not have to use the retraction $p$:

\begin{lem}\label{smooth_approximation_vectors}
	Any continuous $G$-invariant vector field $\xi$ on $M$ can be approximated by a smooth $G$-invariant vector field. If $\xi$ is smooth on a neighborhood of a closed subset $C \subseteq M$ then the smooth vector field can be taken to agree with $\xi$ on $C$.
\end{lem}

Next we consider a particularly nice kind of embedding for manifolds with corners.

\begin{defn}\label{neatdef} A closed embedding is \ourdefn{neat} if it is locally modeled on the inclusion
\[ [0,\infty)^k \times \R^{m-k} \times \{0\}^{n-m} \to [0,\infty)^k \times \R^{m-k} \times \R^{n-m} \]
with $m \leq n$, so in particular the depth of every point is preserved.
\end{defn}

\begin{figure}[h]
	\centering
	\def\svgwidth{3.5in}
	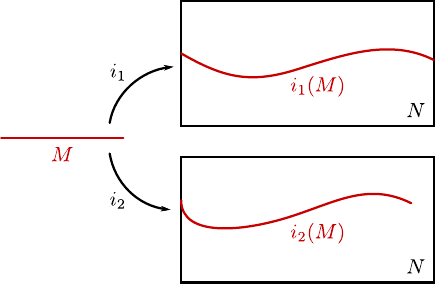
	\vspace{-.5em}
	\caption{A neat embedding $i_1$ and a non-neat embedding $i_2$.}\label{fig:neat_embedding}
\end{figure}

We will need the following extension of \autoref{smooth_on_boundary_tame}.
\begin{lem}\label{smooth_on_boundary_with_neat}
	If $S \subseteq M$ is the union of a partial boundary $A \subseteq \partial M$ and a neat submanifold $P \subseteq M$, a map $f\colon S \to N$ is smooth iff both its restriction to $P$ and its restriction to each component of $\ti A$ are smooth.
\end{lem}

\begin{proof}
	This is by a small modification of the proof of \autoref{smooth_on_boundary_tame}. Again the problem is local, so we consider a map to $\R^m$ defined on an open neighborhood of the origin in the space
	\[ [0,\infty)^a \times [0,\infty)^b \times \R^c \times \R^d, \]
	but only defined on the subspaces $V_1$, $\ldots$, $V_a$ where one of the coordinates in $[0,\infty)^a$ is zero, and on the subspace $W$ where all of the coordinates in $\R^d$ are zero. Our goal is to extend this $f$ to a neighborhood of the origin in the above product space.
	
	As in \autoref{smooth_on_boundary}, we show the problem is unaffected if we project away one of the coordinates in $[0,\infty)^a$, or all of the coordinates in $\R^d$, then apply $f$, and subtract the result from the original map. But after subtracting off all these projections, the map to be extended is identically zero, and so it can be extended by zero.
\end{proof}

Combining this with \autoref{smooth_local_to_global}, we get the following extension of \autoref{smooth_extension}, a smooth extension result for partial boundaries and neat submanifolds.

\begin{cor}\label{smooth_extension_with_neat}
	If $S \subseteq M$ is the union of a $G$-invariant partial boundary $A \subseteq \partial M$ and a $G$-invariant neat submanifold $P \subseteq M$, any equivariant map $f\colon S \to \textup{int }N$ that is smooth along $P$ and along each component of $\ti A$, must extend to a smooth equivariant map $U \to \textup{int }N$ for some $G$-invariant open set $U$ containing $S$.
\end{cor}
}

\orange{We finish this subsection by discussing submersions. For manifolds without boundary, submersions may be defined as smooth maps whose derivatives are surjective maps of tangent spaces, and this is equivalent to saying that the map is locally modeled on the projection $\mathbb R^\ell\times \mathbb R^m\to \mathbb R^\ell\times \{0\}.$ For manifolds with corners, we use the following definition. }

\begin{defn}\label{submersion}
	A \ourdefn{submersion of manifolds with corners} is a smooth map $E \to B$ locally modeled on the projection
	\[ [0,\infty)^k \times [0,\infty)^{j} \times \R^\ell \times \R^{\ell'} \to [0,\infty)^k \times \{0\} \times \R^\ell \times \{0\}. \]
	\orange{An \ourdefn{equivariant submersion} is an equivariant map that is also a submersion. We will only use this term in cases when the action of $G$ on $B$ is trivial.}
\end{defn}

For submersions in this sense the derivative map of tangent spaces is surjective at each point, and additionally the derivative is surjective on the induced maps between (constant-depth) strata. By \cite[5.1]{joyce}, these conditions are equivalent to being a submersion in our sense. 

The next result generalizes the Ehresmann fibration theorem to manifolds with corners. 

\orange{Recall from \autoref{fiber_bundle} the definition of an equivariant smooth fiber bundle. When $G$ acts trivially on the base, the definition simplifies and it becomes a map that is locally equivariantly diffeomorphic to $F \times U \to U$, where $F$ has $G$-action and $U$ has trivial action.}

\begin{lem}\label{ehresmann_with_corners}
	Let $p\colon E \to M$ be a \orange{proper} equivariant submersion of manifolds with corners, \orange{where $G$ acts trivially on $M$.} Then $p$ is an equivariant smooth fiber bundle.
	Furthermore, if the base is $\Delta^k$ then the bundle is equivariantly diffeomorphic to a trivial bundle $\Delta^k \times F \to \Delta^k$, \orange{and any given trivialization on a neighborhood of a closed set $\Delta^k \times C$ can be extended rel $\Delta^k \times C$ to a trivialization of the entire bundle.}
\end{lem}

\begin{proof}
	This follows from the usual proof of Ehresmann's theorem. In more detail, let $\mc V$ be the class of \orange{$G$-invariant} vector fields $\xi$ on $E$ with the property that, in any chart on $E$ in which $p$ is a product as in \autoref{submersion}, the component of $\xi$ in the $[0,\infty)^j$ direction is, at each point, tangent to that point's stratum in $[0,\infty)^j$. Note that $\mc V$ is convex and locally nonempty, and therefore globally nonempty using a partition of unity.
	
	Near any point $x \in M$ locally modeled by $[0,\infty)^k \times \R^\ell$, we fix vector fields near $x$ that point in the coordinate directions. Using the product neighborhoods from \autoref{submersion}, we pick local lifts of each of these fields that lie in $\mc V$, \orange{average them to make them $G$-invariant}, and patch them together by a partition of unity to get a lift in $\mc V$ defined on an open subset containing $p^{-1}(x)$. Flowing along the resulting vector fields gives the desired local trivialization of $E$ near $x$. \orange{If a trivialization already exists on a given open set containing $\Delta^k \times C$, we use that trivialization to pick the lifts of the vector fields, and define the remaining lifts on open sets disjoint from $\Delta^k \times C$. After patching these together by a partition of unity, the resulting trivialization will agree with the given one on $C$.}
\end{proof}

\orange{A consequence of this lemma is that a \orange{proper} equivariant submersion from a manifold to a $G$-trivial manifold always has local models of the form
	\[ [0,\infty)^k \times [0,\infty)^{j} \times \R^\ell \times (G \times_H V) \to [0,\infty)^k \times \{0\} \times \R^\ell \times \{0\}. \]}
	

\orange{
\subsection{Families of maps and families of tubular neighborhoods}

We introduce some notation that we will use frequently in the paper when defining simplicial sets of smooth maps and embeddings. Throughout this subsection we assume that $M$ is compact. Let $P$ be any manifold with corners, not necessarily compact; the main examples are $\Delta^k$ or an open neighborhood of the boundary in $\Delta^k$.

\begin{defn}\label{family}
	A $P$-family of smooth maps $M \to N$ is a smooth map $\phi\colon P \times M \to N$. It is a constant family if the map $\phi_t\colon M \to N$ is independent of the point $t \in P$.
	We say that $\phi$ is a family of embeddings if the track of $\phi$
	\[ \Phi\colon P \times M \to P \times N \]
	\[ (t,x) \mapsto (t,\phi_t(x)) \]
	is a smooth embedding, and we say that $\phi$ is a family of diffeomorphisms if the track $\Phi$ is a diffeomorphism.
\end{defn}

\begin{lem}
	 $\phi$ is a family of diffeomorphisms if and only if for each $t$ the map $\phi_t\colon M \to N$ is a diffeomorphism. If $\phi$ is a family of embeddings then each $\phi_t$ is an embedding. The converse holds if $M$ is compact.
	\end{lem}

\begin{proof}
	It is straightforward to see that $\Phi$ is an immersion iff each $\phi_t$ is an immersion, and that $\Phi$ is bijective iff each $\phi_t$ is bijective. The statement about diffeomorphisms follows. For embeddings, certainly $\Phi$ is injective iff each $\phi_t$ is injective, and if $\Phi$ is a homeomorphism onto its image then each $\phi_t$ is a homeomorphism onto its image. For the converse, suppose that each $\phi_t$ is an embedding and that $M$ is compact. Then the continuous injection $\Phi:P\times M\to P\times N$ is a homeomorphism to its image because it is closed, being a proper map.
	
\end{proof}

In the above, $P$ was a smooth manifold with corners, but it is useful to consider also the case when $P$ is $\partial \Delta^k$. By a $\partial \Delta^k$-family of maps $M \to N$ we mean a map $\partial \Delta^k\times M\to N$ that consists of a smooth $\Delta^{k-1}$-family for each face of $\Delta^k$. By \autoref{smooth_extension} any such family taking values in $\textup{int }N$ extends to a $T$-family of maps for some open set $T \subseteq \Delta^k$ containing $\partial\Delta^k$. If we start with a family of embeddings, then for $T$ sufficiently small the extension is also a family of embeddings:
\begin{lem}\label{extend_embedding_family}
	Suppose $M$ is compact, $T \subseteq \Delta^k$ is an open subset containing $\partial \Delta^k$, $\phi\colon T \times M \to N$ is a family of smooth maps, and $\phi_t$ is an embedding for each $t \in \partial\Delta^k$. Then there is an open set $T_1 \subseteq T$ containing $\partial\Delta^k$ such that $\phi_t$ is an embedding for each $t \in T_1$, for some open set $T_1 \subseteq T$ containing $\partial\Delta^k$.
\end{lem}

\begin{proof}
	Apply \autoref{extend_embedding} to the track $T \times M \to T \times N$. This track is an embedding along $\partial\Delta^k \times M$, and is therefore an embedding on some open neighborhood of $\partial\Delta^k \times M$. By the tube lemma this contains a open neighborhood of the form $T_1 \times M$.
\end{proof}
}

\orange{
Given a $\Delta^k$-family of equivariant embeddings $i\colon \Delta^k \times M \to N$, we let $\Delta^k \times \nu$ denote the normal bundle of the track of the embedding, trivialized so that it is a product of $\Delta^k$ and a vector bundle $\nu \to M$. Strictly speaking, $\nu$ is no longer the normal bundle of each of the embeddings $i_t\colon M \to N$, but they are isomorphic, and the notation $\Delta^k \times \nu$ will be more convenient for the following.

\begin{defn}\label{family_of_tubular_nbhd}
	A \ourdefn{$\Delta^k$-family of tubular neighborhoods} of $i$ consists of a choice of $G$-invariant inner product on the vector bundle $\Delta^k \times \nu$ and a choice of $G$-equivariant embedding $D(\Delta^k \times \nu) \to N$, such that for each point of $\Delta^k$ the resulting embedding $D(\nu) \to N$ is an equivariant tubular neighborhood of $i$ (\autoref{tubular_nbhd}).\footnote{By abuse of notation we also call the map $D(\Delta^k \times \nu) \to N$ a $\Delta^k$-family of embeddings $D(\nu) \to N$, since $\nu$ may be trivialized as a vector bundle with inner product, giving $D(\Delta^k \times \nu) \cong \Delta^k \times D(\nu) \to N$. We still write $D(\Delta^k \times \nu)$ to emphasize that we are allowing the inner product on $\nu$ to vary over $\Delta^k$.}

	
	Similarly a \ourdefn{$\Delta^k$-family of partial tubular neighborhoods} is an open set $U \subseteq \nu$ containing the zero section and a family of equivariant embeddings $\Delta^k \times U \to N$, such that for each point of $\Delta^k$ the resulting embedding $U \to N$ is an equivariant partial tubular neighborhood (\autoref{partial_tubular_nbhd}). Note that $U$ is constant throughout the family. A \ourdefn{$\Delta^k$-family of tubular neighborhood germs} is an equivalence class of such, where we identify two families if they agree when $U$ is replaced by some smaller open set.
\end{defn}
}

\begin{defn}\label{tubular_nbhd_space}
For an equivariant embedding $i\colon M\to \textup{int }N$ \orange{with $M$ compact,} we define a space $\Tub_{\sbt}(M)$ of tubular neighborhoods. \orange{A $k$-simplex is a $\Delta^k$-family of tubular neighborhoods of the constant family of embeddings $i\colon M \to \textup{int }N$, with all of the tubular neighborhoods landing in $\textup{int }N$. Similarly, the space of tubular neighborhood germs $\Tub_{\sbt}^g(M)$ has a $k$-simplex for each $\Delta^k$-family of tubular neighborhood germs of the constant family, all landing in $\textup{int }N$.}
\end{defn}



\begin{thm}[Tubular Neighborhood Theorem, vector bundle version]\label{tubular_nbhd_thm} Assume $M$ is compact. \orange{For any $\Delta^k$-family of embeddings $\Delta^k \times M \to \textup{int }N$, any $\partial \Delta^k$-family of tubular neighborhoods $D(\partial \Delta^k \times \nu) \to \textup{int }N$ can be extended to a $\Delta^k$-family of tubular neighborhoods $D(\Delta^k \times \nu) \to \textup{int }N$.}
\end{thm}	
\orange{
In particular, this says that for every fixed equivariant embedding $i\colon M \to \textup{int }N$, the corresponding space $\Tub_{\sbt}(M)$ is a contractible Kan complex. 
}

\orange{
The proof of \autoref{tubular_nbhd_thm} follows from the next three results.
\begin{lem}\label{weak_tnthm_2}
	For any family of embeddings $i\colon \Delta^k \times M \to \textup{int }N$, there exists at least one family of partial tubular neighborhoods $\tilde i\colon \Delta^k \times U \to \textup{int }N$. 
\end{lem}

\begin{proof}
We follow the usual proof as in \cite[Section 4.5]{hirsch}, which goes in two stages. For the first stage, we embed $N$ into an orthogonal $G$-representation $V$. By \autoref{weak_tubular_nbhd_thm} the embedding $N \to V$ has a tubular neighborhood germ. Note that there is a smooth retraction $p$ of some open neighborhood $\Omega$ of $\textup{int }N$ in $V$ back to $\textup{int }N$, sending every point in $\Omega$ to the closest point in $N$.
	
	For the second stage, we identify the normal bundle of $\Delta^k \times \nu$ with the set consisting of points $(t,x,v)\in \Delta^k \times M \times V$ such that the vector $v$ is tangent to $N$ at $i_t(x)$ and orthogonal to the tangent space of $i_t(M)$. Then we attempt to define $\tilde i\colon \Delta^k \times \nu \to \textup{int }N$ by $\tilde i_t(x,v) = p(x+v)$. This formula does not make sense on all of $\Delta^k \times \nu$, only on those points where $x+v \in \Omega$. This is an open set of $\Delta^k \times \nu$ containing the zero section $\Delta^k \times M$. The first derivative of $\tilde i$ also induces an isomorphism on the normal bundle, as desired. By \autoref{extend_embedding}, the track of $\tilde i$ therefore defines an embedding on some open subset of $\Delta^k \times \nu$ containing the zero section $\Delta^k \times M$. By the tube lemma, this contains an open set of the form $\Delta^k \times U$. We have therefore defined a partial tubular neighborhood $\tilde i\colon \Delta^k \times U \to \textup{int }N$.
\end{proof}

The next step is to prove the tubular neighborhood theorem for the space of germs. Precisely, we show that the analogue of \autoref{tubular_nbhd_thm} holds for the space $\Tub_{\sbt}^g(M)$ of equivariant tubular neighborhood germs. 
\begin{prop}\label{tubular_nbhd_thm_germs}
For any $\Delta^k$-family of embeddings $\Delta^k \times M \to \textup{int }N$, any $\partial \Delta^k$-family of tubular neighborhood germs can be extended to a $\Delta^k$-family of tubular neighborhood germs.
\end{prop}	

Again, in particular this says that for a fixed embedding $i\colon M \to \textup{int }N$, the corresponding space of tubular neighborhood germs $\Tub_{\sbt}^g(M)$ is a  contractible Kan complex.


\begin{proof}
	For simplicity, we first describe the argument for a constant family of embeddings $i\colon M \to \textup{int }N$, in other words we show that $\Tub_{\sbt}^g(M)$ is a contractible Kan complex.
	
	Consider any map of simplicial sets $\partial \Delta^k \to \Tub_{\sbt}^g(M)$. This gives a family of equivariant tubular neighborhood germs $\phi_0\colon \partial \Delta^k \times U \to \textup{int }N$, where $U \subseteq D(\nu)$ is an open subset containing the zero section, and along the zero section $\partial \Delta^k \times M$ the map is the given embedding $i\colon M \to N$. Note that in the manifold $\Delta^k \times \nu$, the subset $\partial \Delta^k \times U$ is a partial boundary, and $\Delta^k \times M$ is a neat submanifold. Therefore by \autoref{smooth_extension_with_neat}, we can extend $\phi_0$ to an equivariant smooth map $\phi_1$ on an open subset $O \subseteq \Delta^k \times \nu$ containing these two subspaces. Shrink $U$ to a compact neighborhood of the zero section, apply the tube lemma, and then shrink $U$ further to be open again. This guarantees that $O$ contains a smaller open set of the form $T \times U$, where $T \subseteq \Delta^k$ is an open subset containing the boundary $\partial \Delta^k$.
	
	The map $\phi_1\colon O \to \textup{int N}$ is smooth, but it is not necessarily an embedding for each point of $\Delta^k$. That is, its track $O \to \Delta^k \times \textup{int }N$ is not necessarily an embedding. However, it is an embedding when restricted to $\partial \Delta^k \times U$. Shrinking $U$ to be compact, applying \autoref{extend_embedding_family}, and then shrinking $U$ again to be open, we conclude that $\phi_2$ gives an embedding for each point of $T$ when $T$ is sufficiently small. It also agrees with the embedding $i$ along $T \times M$. Since $\phi_1$ is full rank and induces the identity in the $M$ direction, in the normal direction to $M$ it induces a smooth isomorphism of vector bundles, that is the identity along $\partial \Delta^k$. Composing with the inverse of this isomorphism, we get a new $T$-family of tubular neighborhood germs $\phi_2\colon T \times U \to \textup{int }N$.

\begin{figure}[h]
	\centering
	\def\svgwidth{4.7in}
	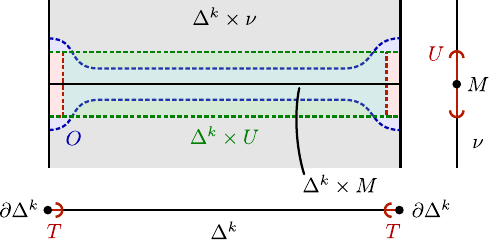
	\vspace{-.5em}
	\caption{Proving the tubular neighborhood theorem for germs.}\label{germtubfig}
\end{figure}

	Now take the single tubular neighborhood from \autoref{weak_tnthm_2}. By shrinking $U$, we may assume it is also defined on the same open set $U$. We use it to make a constant family of tubular neighborhoods $\psi\colon \Delta^k \times U \to \textup{int }N$. Use a smooth partition of unity on $\Delta^k$ subordinate to $\{T,\textup{int } \Delta^k\}$ to interpolate between this constant family $\psi$ and $\phi_2$, as maps $\Delta^k \times U \to V$. The interpolated map $\phi_3\colon \Delta^k \times U \to V$ sends the two subspaces $T \times U$ and $\Delta^k \times M$ into $\textup{int }N$, and therefore sends a neighborhood of these subspaces into $\Omega$, the open neighborhood of $\textup{int }N$ in $V$. Shrinking $U$, we may therefore assume that $\phi_3(\Delta^k \times U) \subseteq \Omega$.
	
	Finally, compose with the retraction $p\colon \Omega \to \textup{int }N$ to get an interpolation between $\phi_2$ and $\psi$ that lands in $\textup{int }N$, and call this $\phi_4\colon \Delta^k \times U \to \textup{int }N$. By construction, every map in this family agrees along the zero section in $U$ with $i\colon M \to N$, and the condition on first derivatives is satisfied throughout. By \autoref{extend_embedding}, it is also a family of embeddings if we allow ourselves to shrink $U$ one more time. This finishes the construction of a family of equivariant tubular neighborhood germs $\phi_4\colon \Delta^k \times U\to \textup{int }N$ extending the original family of germs $\phi_0\colon\Delta^k \times U \to \textup{int }N$.
	
	The argument works with the same formulas even if we allow the embedding $M \to N$ to change over $\Delta^k$, since the embedding $N \to \R^n$ is fixed throughout. \orange{In this case the family $\psi$ isn't constant, but it is still defined by the formula $p(x+v)$.}
\end{proof}
}

\orange{
Finally, we show how to go from families of germs back to families of tubular neighborhoods. 
\begin{prop}\label{germs_to_full}
	 For any family of equivariant embeddings $\Delta^k \times M \to \textup{int }N$, any family of equivariant tubular neighborhoods on $\partial \Delta^k$ whose germs are extended to $\Delta^k$, can be extended as a family of tubular neighborhoods to $\Delta^k$.
\end{prop}

In particular, in the case of a fixed embedding $i\colon  M \to \textup{int }N$ this says that the forgetful map $\Tub_{\sbt}(M) \to \Tub_{\sbt}^g(M)$ of the corresponding spaces of tubular neighborhoods and tubular neighborhood germs for $i$, which forgets the inner product and restricts the tubular neighborhood from the closed disc bundle $D(\nu)$ to its open disc bundle, is an acyclic Kan fibration.


\begin{proof}
	The proof is somewhat similar to the previous one. In this case, we are given a family of tubular neighborhoods $D(\partial \Delta^k \times \nu) \to \textup{int }N$ and an extension of the underlying family of germs to $\Delta^k \times U \to \textup{int }N$, for some open subset $U \subseteq \nu$ containing the zero section. We summarize this data by saying we have a single smooth equivariant map
	\[ \phi\colon D(\partial \Delta^k \times \nu) \cup (\Delta^k \times U) \to \textup{int N}, \]
	and for each $t \in \Delta^k$ the resulting map from either $D(\nu)$ or $U$ to $\textup{int }N$ is a tubular neighborhood or tubular neighborhood germ, respectively.
	
	We begin by picking any inner product on $\Delta^k \times \nu$ extending the given inner products on each face $\Delta^{k-1} \times \nu$. To see this is possible locally, we assume $\nu$ is a trivial bundle and let $\textup{Inn}(\R^n)$ be the space of positive-definite symmetric matrices, an open subset of $\R^{\binom{n}{2}}$, and therefore an open manifold. Then the given inner products on $\partial \Delta^k$ are given by smooth maps $\Delta^{k-1} \to \textup{Inn}(\R^n)$ that agree along the common faces. These can be extended to $\Delta^k \to \textup{Inn}(\R^n)$ by \autoref{smooth_on_boundary}. This shows that the inner products admit extensions to any sufficiently small open subset of $\Delta^k \times \nu$, but we may patch the local extensions together with a smooth partition of unity, just as we do for vector fields, and then make the result $G$-invariant by averaging. Once this is done, our inner product on $\Delta^k \times \nu$ gives us a subspace $D(\Delta^k \times \nu) \subseteq \Delta^k \times \nu$, and our goal is to extend the above map $\phi$ to this subspace while preserving its value along both $D(\partial \Delta^k \times \nu)$ and along $\Delta^k \times U'$ for a smaller open set $U' \subseteq U$.
	
	As in the proof of \autoref{tubular_nbhd_thm_germs}, we first extend $\phi$ from $D(\partial \Delta^k \times \nu)$ to a smooth map $\phi_1$ defined on an open neighborhood $O$ inside the product $D(\Delta^k \times \nu)$. It won't be necessary to specify the value of $\phi_1$ on the zero section $\Delta^k \times M$. Since the disc bundles $D(\nu)$ are compact, $O$ contains a set of the form $D(T \times \nu)$, where $T \subseteq \Delta^k$ is an open subset containing the boundary $\partial \Delta^k$. So we get a smooth equivariant map $\phi_1\colon D(T \times \nu) \to \textup{int }N$.
	
	
\begin{figure}[h]
	\centering
	\def\svgwidth{5.2in}
	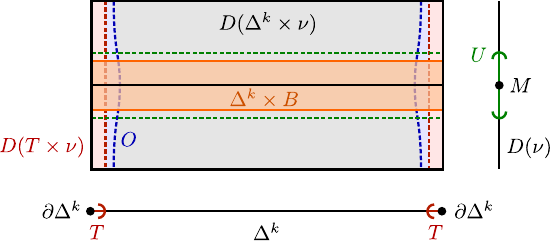
	\vspace{-.5em}
	\caption{Proving the tubular neighborhood theorem for for full tubular neighborhoods.}\label{tubfig1}
\end{figure}

	The next step is different because we want to recover the original family $\phi$ along $\Delta^k \times U$ for $U$ sufficiently small. Let $B \subseteq U$ be any fixed closed neighborhood of the zero section. Take a smooth $G$-invariant partition of unity on $\nu$ subordinate to $\{U,B^c\}$. Use this to interpolate between $\phi\colon \Delta^k \times U \to \textup{int }N$ and $\phi_1\colon T \times B^c \to \textup{int N}$, giving a smooth equivariant map
	\[ \phi_2\colon D(T \times \nu) \cup (\Delta^k \times B) \to V, \]
	which agrees with $\phi$ along both $\Delta^k \times B$ and $D(\partial \Delta^k \times \nu)$. These points land in $\textup{int N}$, so if we shrink $T$ more, then all of $D(T \times \nu)$ lands in $\Omega$. Therefore after composing with the retraction $p$ we get a smooth equivariant map $\phi_3\colon D(T \times \nu) \cup (\Delta^k \times B) \to \textup{int }N$. By applying \autoref{extend_embedding_family} to this family, if we shrink $T$ even more, we can ensure that each of the maps $D(\nu) \to \textup{int N}$ is a smooth embedding.
	
	The map $\phi_3$ agrees with $\phi$ on $D(\partial\Delta^k \times \nu) \cup (\Delta^k \times B)$, as required, so it remains to make it into a family of honest tubular neighborhoods and not just tubular neighborhood germs. Inside the disc bundle $D(\Delta^k \times \nu)$, the subspace $B$ contains all discs of radius $\epsilon$ for some small value of $\epsilon > 0$. Pick a smooth embedding $e\colon I \to I$ sending $I$ into $[0,\epsilon]$ and that is the identity near 0. This gives a constant family of smooth embeddings $(\textup{int }\Delta^k) \times I \to I$. Interpolate between this family and the constant family of identity embeddings $T \times I \to I$, using a smooth partition of unity of $\Delta^k$ subordinate to $\{T,\textup{int }\Delta^k\}$. This gives a family of embeddings $\Delta^k \times I \to I$ that are the identity of $I$ near $\partial \Delta^k$, and that agree with $e$ on $\Delta^k \setminus T$. Applying these embeddings to the radial coordinate of $D(\nu)$ gives a smooth embedding over $\Delta^k$
	\[ \tilde e\colon D(\Delta^k \times \nu) \to D(T \times \nu) \cup (\Delta^k \times B) \]
	that is the identity near the zero section $\Delta^k \times M$ and near the sides $D(\partial\Delta^k \times \nu)$. Composing $\phi_3$ with this map gives the desired family of equivariant tubular neighborhoods $\phi_4\colon D(\Delta^k \times \nu) \to \textup{int }N$.
\end{proof}

\begin{figure}[h]
	\centering
	\def\svgwidth{4.8in}
	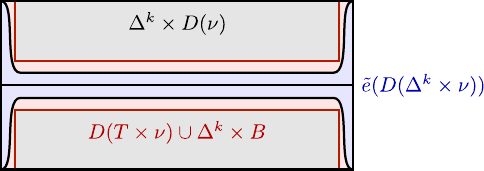
	\vspace{-.5em}
	\caption{The last step of the proof, turning the germs into full tubular neighborhoods.}\label{tubfig2}
\end{figure}

Together \autoref{tubular_nbhd_thm_germs} and \autoref{germs_to_full} give the proof of \autoref{tubular_nbhd_thm}.
}

\begin{rem}
	Neat embeddings can have ``neat'' tubular neighborhood \orange{germs}, i.e. ones in which the map $\tilde i$ is an open embedding. Furthermore the space of neat tubular neighborhood germs is contractible. However we will not need to prove such a statement in this paper, but this is treated in detail in \cite{GIMM2}.
\end{rem}

\subsection{Trimmings, faces, and collars}\label{sec:trimmings_collars}
Let $M$ be a compact $G$-manifold with corners. A smooth ($G$-invariant) vector field on an open subset of $M$ containing $\partial M$ is \ourdefn{inward pointing} if for each point $x \in \partial M$, in one (therefore in all) charts the vector at $x$ points to the interior of $[0,\infty)^k \times \R^{n-k}$. \orange{Equivalently, thinking of the vector at $x$ as a point $(x,\xi)\in [0,\infty)^k \times \R^{n-k}$, whenever one of the coordinates of $x$ is zero, the corresponding coordinate of $\xi$ is positive. By \autoref{smooth_extension_vectors}}, without loss of generality the vector field is defined on all of $M$. (It can be zero far away from $\partial M$.)

An embedded manifold with boundary $M' \subseteq M$ is a \ourdefn{trimming} if there is a $G$-invariant inward-pointing vector field on $M$ that is nonvanishing on $M \setminus \textup{int}\ M'$, transverse to $\partial M'$, and such that the integral curves give a homeomorphism $\partial M' \cong \partial M$. In particular, this implies that $M' \to M$ is continuously homotopic to a homeomorphism. 

\begin{figure}[h]
	\centering
	\def\svgwidth{3in}
	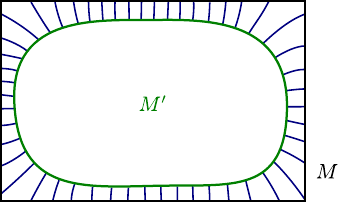
	\vspace{-.5em}
	\caption{A trimming $M'$ of a manifold with corners $M$.}\label{fig:trimming}
\end{figure}

\begin{lem}\label{trimmings_exist}
	Every compact $G$-manifold with corners has a trimming, \orange{which may be taken with $M \setminus M'$ contained in an arbitrarily small neighborhood $U$ of $\partial M$.}
\end{lem}

\begin{proof}
	Choose an inward-pointing vector field $\xi$ \orange{that is supported on $U$}; this can be done by gluing together local inward-pointing fields using a smooth partition of unity on $M$. By averaging, we can assume that $\xi$ is $G$-invariant. As in \cite[Section 6]{waldhausen_manifold}, $\xi$ gives $\partial M$ a smooth structure in which the charts are obtained by taking discs transverse to $\xi$ and flowing to reach $\partial M$. We use this smooth structure on $\partial M$ throughout this proof.
	
	\orange{Each inward-pointing vector field has unique integral curves defined starting from $\partial M$, using for instance \cite[Cor 1.13.1]{melrose}, and these curves are defined for all positive times $t \geq 0$ since $M$ is compact.} This gives us a continuous map $\phi\colon \partial M \times [0,\infty) \to U$. Although $\phi$ is not smooth with respect to the smooth structure on $\partial M$ described in the previous paragraph, it is an open topological embedding.
	
	Let $V$ be the image $ \phi(\partial M \times (0,1))$, with smooth structure coming from the fact that it is an open subset of $M$. Let $p\colon V \to \partial M$ be the projection back to $\partial M$. Although $\phi$ is not smooth, the projection $p$ is smooth, by construction. In fact, it is a smooth submersion whose fibers are open intervals. It therefore has a smooth equivariant section, \orange{defined by for instance taking local sections, averaging them to make them $G$-invariant, and then patching them together using a smooth equivariant partition of unity.} This defines the boundary of the desired submanifold $M' \subseteq M$.
\end{proof}

\orange{We can also use inward-pointing vector fields to extend each manifold with corners to a larger manifold with only interior points.}

\begin{prop}\label{deform_embedding_to_interior}
	If $M$ is a compact smooth $G$-manifold with corners, there is an isotopy of equivariant embeddings from $\id_M$ to an embedding $M \to M$ sending $M$ into the interior. \orange{This isotopy may be chosen to be supported in an arbitrarily small neighborhood $U$ of $\partial M$.}
\end{prop}

\begin{proof}
As before, choose an inward-pointing $G$-invariant vector field $\xi$, \orange{supported on $U$. Then the flow along $\xi$ defines a smooth map $\phi\colon M \times [0,\infty) \to M$. Restricting $\phi$ to $M \times [0,1]$ gives the desired isotopy, since $\phi(-,0)$ is the identity, $\phi(-,1)$ is an embedding of $M$ into its interior, and $\phi(x,t) = x$ whenever $x \not\in U$.}
\end{proof}

\begin{cor}\label{extend_to_open_manifold}
	Every compact $G$-manifold with corners can be smoothly equivariantly embedded into the interior of another $G$-manifold with corners.
\end{cor}

\orange{
The following is an important consequence of \autoref{smooth_extension} and \autoref{extend_to_open_manifold}. It is a smooth extension result that does not assume that $f$ lands in the interior of $N$.
\begin{cor}\label{smooth_on_partial_delta} 
	If $A \subseteq \partial M$ is a partial boundary and $N$ is compact, then every equivariant map $A \to N$ that is smooth as a map $\ti A \to N$ extends to an equivariant smooth map $U \to N'$, where $U$ is a $G$-invariant open neighborhood of $A$ in $M$, and $N'$ is an open manifold containing $N$ in its interior.
\end{cor}
}

By the counterexample in \autoref{smooth_extension_counterexample}, this is the best we can do in general.

\orange{
\begin{defn}\label{partial_inward_pointing}
	Suppose $A \subseteq \partial M$ is a closed partial boundary, in other words the image of a union of connected components of $\ti\partial M$. Then near each point there is a coordinate chart in which $A$ is identified with a subspace of the form
	\[ \partial([0,\infty)^a) \times [0,\infty)^b \times \R^{n-a-b}. \]
	We say that a vector field $\xi$ on $M$ is \ourdefn{inward pointing from $A$} if for each of the first $a$ coordinates, when $x$ is zero in that coordinate then $\xi(x)$ is positive in that coordinate, and for the next $b$ coordinates, if $x$ is zero in that coordinate then $\xi(x)$ is also zero in that coordinate. It is straightforward to see this is independent of the choice of chart.
\end{defn}

These are precisely the vector fields such that flowing along the field brings us away from $A$, but keeps us within any local boundary component that is not in $A$. It is easy to construct such a vector field near each point of $A$. The assumption that $A$ is closed allows us to also construct such a vector field near each point in $\partial M \setminus A$.

\begin{lem}\label{flow_from_partial_boundary}
	If $A \subseteq \partial M$ is a closed $G$-invariant partial boundary, then the space of $G$-invariant inward-pointing vector fields along $A$ is nonempty and convex. If $M$ is compact, flowing along such a field gives an equivariant isotopy from $\id_M$ to a map sending $M$ into $M \setminus A$. The vector field may be chosen so that this isotopy is supported in an arbitrarily small neighborhood of $A$.
\end{lem}

\begin{proof}
	 The proof is the same as in \autoref{trimmings_exist} and \autoref{deform_embedding_to_interior}, except that we prove the existence of such a vector field nonequivariantly, and then average the result over $G$ to make the field $G$-invariant. For this to work we need to know that each $g \in G$ preserves (nonequivariant) inward-pointing vector fields along $A$, but this is easy to see, as in the local model $G \times_H V \times [0,\infty)^k$, the $G$-action does not change the last $k$ coordinates.
\end{proof}

\begin{lem}\label{extend_inward_pointing_from_partial_boundary}
	If $A,B \subseteq \partial M$ are two closed $G$-invariant partial boundaries, any $G$-invariant vector field defined along $A$ that is inward-pointing along $B$ extends to a $G$-invariant vector field on all of $M$, that is inward-pointing along $B$.
\end{lem}

\begin{proof}
	It suffices to do this nonequivariantly, since a nonequivariant extension can be averaged over $G$, as in the previous proof. Since the space of inward-pointing vector fields is convex, it suffices to do the extension locally near each point of $x \in M$. If $x \not\in A$ then it is easy to define an inward-pointing vector field along $B$ near $x$, so assume that $x \in A$. We model $M$ near $x$ as $[0,\infty)^k \times \R^{n-k}$. The partial boundary $A$ is specified by a subset $S_A \subseteq \{1,\ldots,k\}$, and $B$ is specified by a subset $S_B \subseteq \{1,\ldots,k\}$. We perform the smooth extension of $\xi$ from the subset where one of the coordinates in $S_A$ is zero to all of $[0,\infty)^k \times \R^{n-k}$ by the same method as \autoref{smooth_extension_vectors}, see \autoref{smooth_on_boundary} and \autoref{smooth_on_boundary_tame} for the explicit formula. For each $y$ in this neighborhood, when $y_i = 0$ and $i \not\in S_B$, the $i$th coordinate of $\xi(y)$ is a linear combination of zero functions and is therefore zero. When $y_i = 0$ and $i \in S_B$, we need the $i$th coordinate to be positive. We can't guarantee this on our entire neighborhood, but it is true on some open set containing $x$, because the $i$th coordinate being positive is an open condition. Intersecting these open sets for each $i \in S_B$ gives a smaller open set on which the desired smooth extension exists and is inward-pointing along $B$.
\end{proof}
}

\begin{defn}\label{face} If $M$ is a $G$-manifold with corners, a \ourdefn{face} $F$ of $M$ is a $G$-invariant subspace of the smooth boundary $\tipartial M$ such that
\begin{itemize}
	\item $F$ is a union of components of $\tipartial M$ \orange{(possibly empty),} and
	\item the map $i\colon \tipartial M \to M$ is injective when restricted to $F$.
\end{itemize}
\end{defn}

\orange{Since $i$ is injective on $F$, we sometimes blur the distinction between $F$ being a subset of $\ti\partial M$ or a subset of $\partial M$. Note that faces are a special case of partial boundaries in the sense of \autoref{tame_def}.}

\begin{exmp}[Faces] We give some examples and nonexamples of faces.
	\begin{enumerate}
\item Each side of the square $I\times I$ is a face, as is the union of two opposite sides. But the union of two adjacent faces is not, since the inclusion back to $I \times I$ is not injective.
\item The boundary of the $n$-simplex $\Delta^n$  consists of  $n+1$ faces, each diffeomorphic to $\Delta^{n-1}$ as a manifold with corners. 
\item In the teardrop-shaped 2-manifold, there are no faces. There is only one component in the smooth boundary, a closed interval whose endpoints both map under $i$ to the top of the teardrop, so $i$ is not injective on this component. 
	\end{enumerate}
\end{exmp}

Next we define collars on faces of a manifold with corners, which will play an important role in our definition of $h$-cobordisms. 

\begin{defn}
	Let $F$ be a face of a $G$-manifold with corners $M$. A \ourdefn{collar} on $F$ is an extension of $F \to M$ to an equivariant embedding $c\colon F \times I \to M$ that preserves depth on $F \times [0,1)$ (and so is an open embedding on that subset). A collar is \ourdefn{neat} if it decreases depth by exactly 1 on $F \times \{1\}$.
\end{defn}

\begin{figure}[h]
	\centering
	\def\svgwidth{1.5in}
	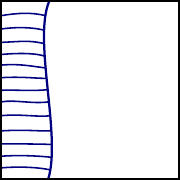
	\vspace{-.5em}
	\caption{A neat collar for the left-hand face of a square.}\label{fig:neat_collar}
\end{figure}

\begin{lem}\label{collars_unique}\label{collar_isotopy_extension}
	Every face $F$ in a smooth compact $G$-manifold with corners $M$ has a neat collar $F \times I \to M$. Any two collars are isotopic through collars, and any two neat collars are related by an ambient isotopy of $M$.
\end{lem}

As a result, we could re-define a face of $M$ to be a compact $G$-invariant subset $F \subseteq \partial M$ that has an open neighborhood diffeomorphic to $F \times [0,1)$.

\begin{proof}
	We can deform any collar to a neat collar by pre-composing with an isotopy of embeddings $I \to I$, so we focus on neat collars. \orange{Consider the space of $G$-invariant inward-pointing vector fields along $F$ in the sense of \autoref{partial_inward_pointing}. Spelling out what they look like in this case:}
	\begin{itemize}
		\item for every point $x \in M \setminus F$ of depth $k$, the vector at $x$ lies in the tangent space of the stratum of depth $k$ points 
		\item for every point $x \in F$ of depth $k$ in $F$ (so depth $k+1$ in $M$), \orange{we may identify the inclusion $F \to M$ with the inclusion $F \times \{0\} \to F \times I$, and along this identification,} the vector $x$ lies in the tangent space of the manifold-with-boundary
		\[ \textup{(depth $k$ points of $F$)}\times I, \]
		and is inward pointing (positive in the $I$ direction) at that point.
	\end{itemize}
	\orange{As in \autoref{partial_inward_pointing}, the collection of inward-pointing fields along $F$ is convex and nonempty. As in \autoref{trimmings_exist},} each such vector field has unique integral curves defined starting from $F$, defined for all positive times $t \geq 0$. By the depth-preservation conditions, flowing along one such vector field until time $t = 1$ provides a neat collar for the face $F$.
	
	Conversely, any neat collar $F \times [0,1] \to M$ defines such a vector field on its image. We can then extend this vector field to the rest of $M$ using an equivariant smooth partition of unity and adding the given vector field to the zero vector field in all the other charts. Given two neat collars, we can linearly interpolate between these fields and get a one-parameter family of fields with the same condition. Flowing along these fields from $F$ defines a one-parameter family \orange{(in the sense of \autoref{family})} of neat collars.
	
	To extend this to an ambient isotopy, we take the corresponding time-dependent vector field $F \times [0,1] \times [0,1] \to TM$ defined on the image of the isotopy, as in \cite[\S 8.1]{hirsch}. We extend this field to a time-dependent vector field $M \times [0,1] \to TM$ that is zero far away from this image, again using a smooth partition of unity. Flowing along this time-dependent field gives the desired ambient isotopy (\cite[8.1.2]{hirsch}).
\end{proof}

Two collars are said to have the same germ if they agree on some $G$-invariant open neighborhood of the bottom and sides
\[ (F \times \{0\}) \cup (\partial F \times I) \]
inside $F \times I$. Just as with tubular neighborhoods, every germ can be used to produce a collar:

\begin{lem}\label{partial_collar}
	If $M$ is compact, $F \subseteq M$ is a face, $L$ is an open subset of $F \times I$ containing the bottom and sides
	\[ (F \times \{0\}) \cup (\partial F \times I) \subseteq L \subseteq F \times I, \]
	 and $\tilde c\colon L \to M$ is a partially-defined collar, then there is a collar $c\colon F \times I \to M$ whose germ agrees with that of $\tilde c$.
\end{lem}

\begin{proof}
	By compactness there is an $\epsilon > 0$ such that $L$ contains $F \times [0,\epsilon]$. \orange{As in the proof of \autoref{germs_to_full},} pick a smooth embedding $I \to I$ sending $I$ into $[0,\epsilon]$ and that is the identity near 0. Composing $\tilde c$ with this embedding gives an embedding $F \times I \to M$ that agrees with $\tilde c$ on a neighborhood of the bottom, though not the sides.
	
	Now pick two $G$-invariant nested open neighborhoods $U_0 \subseteq U_1 \subseteq F$ containing $\partial F$ such that $U_1 \times I \subseteq L$, and let $V_1$ be the complement of the closure of $U_0$. Pick a $G$-invariant partition of unity subordinate to the cover $\{U_1,V_1\}$ and use it to add together the embedding $I \to I$ constructed above (on $V_1$) and the identity of $I$ (on $U_1$). The resulting equivariant embedding $F \times I \to F \times I$ agrees with the previous embedding outside of $U_1$, and is the identity outside of $V_1$. Therefore it is entirely contained in $L$. By construction it is also the identity near the bottom and sides. Therefore, composing with $\tilde c$ gives an equivariant collar $F \times I \to M$ whose germ agrees with that of $c$.
\end{proof}

This result will be useful---we will use germs of collars more than collars themselves.

\subsection{Smooth simplices of diffeomorphisms}
Suppose $W$ is a compact $G$-manifold with corners, $F \subseteq \partial W$ is a face \orange{(possibly empty)}, and and $C \subseteq \partial W$ is the closure of the complement of $F$ in $\partial W$, \orange{also possibly empty}. In particular, $C$ contains the boundary of $F$, and all of the corner points of $W$.

A \ourdefn{diffeomorphism of $(W,F)$} is a diffeomorphism of $W$ that is the identity on some neighborhood of $C$. Thus it restricts to give a diffeomorphism of $F$ (which is the identity on a neighborhood of $\partial F$).

The simplicial group of diffeomorphisms of $(W,F)$ is defined using families of diffeomorphisms parametrized by $\Delta^k$ that correspond to smooth maps from $\Delta^k \times W$ to $W$:

\begin{defn}\label{smooth_diffeo_space}
Let $\mc D_{\sbt}(W,F)$ be the simplicial set whose $k$-simplices are families of equivariant diffeomorphisms $\Delta^k \times W \cong \Delta^k \times W$ over $\Delta^k$ \orange{(as in \autoref{family})}, that are the identity on $\Delta^k \times U$ for some open set $U$ containing $C$.
\end{defn}

\orange{
The following characterization of homotopies into $\mc D_{\sbt}(W,F)$ is often useful.
\begin{lem}\label{smooth_simplicial_homotopy}
	If $X_{\sbt}$ is a simplicial set and $\Delta_{\sbt}[1]$ is the standard 1-simplex, a simplicial homotopy $X_{\sbt} \times \Delta_{\sbt}[1] \to \mc D_{\sbt}(W,F)$ is given by assigning each point of $X_k$ to a smooth $\Delta^k \times \Delta^1$-family of equivariant diffeomorphisms of $(W,F)$, in a way that respects the face and degeneracy maps of $X_{\sbt}$.
\end{lem}

\begin{proof}
	Given such families $f_x$ for each $x \in X_k$, we define the map of simplicial sets $X_{\sbt} \times \Delta_{\sbt}[1] \to \mc D_{\sbt}(W,F)$ by assigning each point $(x,\alpha) \in X_n \times \Delta_n[1]$ to the composite
	\[ \xymatrix @C=4em{ \Delta^n \times W \ar[r]^-{(\id,\alpha) \times \id} & \Delta^n \times \Delta^1 \times W \ar[r]^-{f_x} & W, } \]
	where $\alpha\colon \Delta^n \to \Delta^1$ is the affine-linear map induced by the map of totally ordered sets $\alpha \colon [n] \to [1]$. This composite is a family of diffeomorphisms which gives a point in $\mc D_n(W,F)$. It is straightforward to see that this respects the face and degeneracy maps applied to the point in $X_n \times \Delta_n[1]$, and therefore defines the desired map of simplicial sets.
\end{proof}
}

\orange{The simplicial set $\mc D_{\sbt}(W,F)$ is a Kan complex since it is a simplicial group \cite[Lemma I.3.4]{goerss_jardine}.} This gives us the following extension of \autoref{ehresmann_with_corners}. Consider an equivariant submersion of compact manifolds with corners $p\colon E \to \Delta^k$, with a face $M \subseteq \partial E$ such that the restricted map $M \to \Delta^k$ is also a submersion. By \autoref{ehresmann_with_corners}, we know that such a family can be trivialized to $\Delta^k \times (W,F)$ for some fixed manifold $W$ and face $F \subseteq \partial W$.

\begin{cor}\label{extend_trivializations}
	Any trivialization of $(E,M)$ that is defined on a proper union of $(k-1)$-dimensional faces $\cup_i D_i \subsetneq \partial\Delta^k$ can be extended to all of $\Delta^k$.
\end{cor}

\begin{proof}
	Without loss of generality $(E,M) = (W,F) \times \Delta^k$, and so the given trivialization is a family of diffeomorphisms of $(W,F)$ defined on $\cup_i D_i$. \orange{Since $\mc D_{\sbt}(W,F)$ is a simplicial group, it} is a Kan complex, so this family of diffeomorphisms can be extended to $\Delta^k$, giving the desired trivialization.
\end{proof}

\orange{
To connect this paper to the previous literature, it is also necessary to prove that these smooth simplices of diffeomorphisms give a space that is equivalent to the usual space of diffeomorphisms of $W$.
}

\begin{thm}\label{smooth_simplices}
	$\mc D_{\sbt}(W,F)$ is equivalent to the space of equivariant diffeomorphisms of $(W,F)$ with the $C^\infty$ topology.
\end{thm}

\begin{proof}
	
	This proof is an adaptation of \cite[Prop 1]{lurie_937_lecture6}. Let $\Diff(W,F)$ be the space of equivariant diffeomorphisms of $(W,F)$ with the $C^\infty$ topology. \orange{Since $\mc D_{\sbt}(W,F)$ and the singular simplices of $\Diff(W,F)$ are both Kan complexes,} it suffices to take a diagram
	\[ \xymatrix{
		\partial \Delta[k] \ar[d] \ar[r] & \mc D_{\sbt}(W,F) \ar[d] \\
		\Delta[k] \ar@{-->}[ur] \ar[r] & \Sing_{\sbt}(\Diff(W,F))
	} \]
	and show that the bottom map can be changed by a simplicial homotopy rel $\partial \Delta[k]$ to a map for which a dotted lift exists.
	
	\orange{Without loss of generality, $W$ is connected.} The top horizontal map in the above diagram corresponds to a family of equivariant diffeomorphisms $\phi_0\colon \partial \Delta^k \times W \to W$ and the bottom map is an extension to an equivariant continuous map
	\[ f\colon \Delta^k \times W \to W, \]
	that at each point $t \in \Delta^k$ gives a diffeomorphism $f_t\colon W \to W$, and that on some neighborhood $U_t$ of $C$ is the identity of $W$. (The partial derivatives in the $W$ direction are also continuous along the product $\Delta^k \times W$.) Furthermore, the neighborhoods $U_t$ can be chosen uniformly on each of the faces in $\partial \Delta^k$, \orange{since by \autoref{smooth_diffeo_space} each simplex in $\mc D_{\sbt}(W,F)$ gives a family of diffeomorphisms that is the identity on a single choice of open subset $U \subseteq W$.} This implies that the $U_t$ can also be chosen uniformly over the entire boundary $\partial \Delta^k$. Call this uniform $G$-invariant neighborhood $U_0$.
	
	\orange{
	The first step is to change the family rel $\partial \Delta^k$ so that there is a uniform choice of open set $U$ containing $C$ on which all of the maps $f_t$ are the identity. Note that $C$ is a closed partial boundary of $W$, so by \autoref{flow_from_partial_boundary}, we may form an isotopy of embeddings $\psi_t\colon W \to W$ from the identity to a map $\psi_1\colon W \to W \setminus C$, that restricts to an isotopy of embeddings $F \to F$, and that is the constant isotopy on $W \setminus U_0$. Conjugating the family $f\colon \Delta^k \times W \to W$ by $\psi_t$ defines a family of diffeomorphisms
	\[ \xymatrix @C=4em{ \Delta^k \times \psi_t(W) \ar[r]^-{\id \times \psi_t^{-1}} & \Delta^k \times W \ar[r]^-f & \Delta^k \times W \ar[r]^-{\id \times \psi_t} & \Delta^k \times \psi_t(W). } \]
	Each of these diffeomorphisms is identity near $\psi_t(C)$, and that therefore extends by the identity to a diffeomorphism $W \to W$. Furthermore this extension is and its partial derivatives in the $W$ direction are all continuous. This deforms our family of diffeomorphisms to a new family with the property that on the open set $U = W \setminus \psi_1(W)$ every diffeomorphism is the identity. Note that $C \subseteq U \subseteq U_0$. Furthermore since $\psi_t$ is supported on $U_0$, any diffeomorphism that is the identity on $U_0$ is unchanged by this isotopy. In particular our family on the boundary $\phi_0\colon \partial \Delta^k \times W \to W$ is unchanged, but now there is a uniform choice of neighborhood $U \subseteq W$ containing $C$ on which every diffeomorphism is the identity.
	
	The second step is to construct a smooth equivariant map
	\[ f_0\colon (T_0 \times W) \cup (\Delta^k \times U) \to W, \]
	where $T_0 \subseteq \Delta^k$ is an open set containing the boundary, the map $f_0$ is a family of equivariant diffeomorphisms of $(W,F)$ on $T_0$ agreeing with $\phi_0$ on $\partial \Delta^k \times W$, and on $\Delta^k \times U$ the map $f_0$ is the constant family of inclusion maps $U \to W$.
	
	Clearly we already know how to define $f_0$ on $\Delta^k \times U$, so we do that first. Then we define $f_0$ along the top face $T_0 \times F$. We do this one component of $F$ at a time, so without loss of generality $F$ is connected. By \autoref{extend_to_open_manifold}, we may find an open manifold $F'$ containing $F$. The map we want is already defined on the partial boundary $\partial \Delta^k \times F$ and along the open set $T_0 \times (U \cap F)$ containing $T_0 \times \partial F$. It is also smooth at every point in these subspaces, so by \autoref{smooth_local_to_global} it extends to some smooth equivariant map $O \to F'$ for some open set $O \subseteq \Delta^k \times F$ containing both $\partial \Delta^k \times F$ and $T_0 \times (U \cap F)$. Shrinking $T_0$, we may assume that $O$ contains all of $T_0 \times F$. Now we have a smooth equivariant map $f_0\colon T_0 \times F \to F'$, that along $T_0 \times (U \cap F)$ is the constant family of inclusions $(U \cap F) \to F$. By \autoref{extend_embedding_family}, if we shrink $T_0$ further then $f_0$ is a family of embeddings $F \to F'$. Since $U \cap F$ contains $\partial F$, these embeddings are all the identity on $\partial F$. Each one is also the identity on some point $x \in U \cap \textup{int }F$, so by \autoref{embedding_is_diffeo}, we therefore have a family of equivariant diffeomorphisms $f_0\colon T_0 \times F \to F$. (In the case that $U$ and therefore $\partial F$ are empty, $F'$ contains $F$ as a connected component, and along $\partial \Delta^k$ each of our embeddings $F \to F'$ lands in $F$. So, as long as we arrange $T_0$ to be path-connected, all of the embeddings are $F \to F$ and so  \autoref{embedding_is_diffeo} applies once more.) 
	
	Now we have defined $f_0$ on the three subspaces $\partial \Delta^k \times W$, $\Delta^k \times U$, and $T_0 \times F$ in a compatible way. We again apply \autoref{smooth_local_to_global} to extend this to a smooth map $O \to W'$, where $W'$ is an open manifold containing $W$ and $O \subseteq \Delta^k \times W$ is an open set containing $\partial \Delta^k \times W$, $\Delta^k \times U$, and $T_0 \times F$. By shrinking $T_0$ once more, we may assume that $O$ contains $T_0 \times W$, and that on $T_0 \times W$ this is a family of smooth embeddings $W \to W'$. But for each point of $T_0$, the resulting map $W \to W'$ is the identity on $U$ and a diffeomorphism on $F$, so it induces a homeomorphism $\partial W \to \partial W$. By \autoref{embedding_is_diffeo}, it is therefore a family of diffeomorphisms $W \cong W$. (If $U$ is empty, then one can only leave $W$ by passing through $\partial W = F$, and again as long as $T_0$ is path-connected this does not happen, so that \autoref{embedding_is_diffeo} applies. If $T_0$ were not path-connected one could flip from the given embeddings to embeddings that send $W$ \emph{outside} of $\textup{int }W \subseteq W'$!)
	
	We have therefore constructed the desired family of diffeomorphisms $f_0\colon T_0 \times W \to W$, that along $\Delta^k \times U$ is the constant family of inclusions $U \to W$, and that along $\partial \Delta^k \times W$ is the original family of diffeomorphisms $\phi_0$.
	}

	\orange{
	To extend $f_0$ to the rest of $\Delta^k \times W$, pick an equivariant smooth embedding $W \to V \times [0,\infty)$, where $V$ is a sufficiently large $G$-representation, that on a collar neighborhood of $F$ is given by an embedding $F \to V$ times the inclusion $[0,1] \subseteq [0,\infty)$, and away from this collar sends everything into $V \times (0,\infty)$. This is possible by the same technique as in \autoref{embed_into_V}. Away from $C$, this gives a neat embedding of manifolds with boundary
	\[ W \setminus C \to V \times [0,\infty). \]
	In particular, $F \setminus C$ goes to $V \times \{0\}$. By the construction in the proof of \autoref{weak_tubular_nbhd_thm}, there is a smooth equivariant retract $\pi\colon \Omega \to W$ where $\Omega$ is a tubular neighborhood of $W$ and therefore a neighborhood of $W \setminus C$, that sends $V \times \{0\}$ to $F$ and $V \times (0,\infty)$ to $W \setminus F$.
	
	For each finite open cover $\{T_i\}_{i=1}^n$ of $\Delta^k - T_0$ by sets contained in the interior of $\Delta^k$, extend the cover to $\Delta^k$ by including $T_0$, then pick a smooth partition of unity $\{\lambda_i\colon \Delta^k \to [0,1]\}$ subordinate to the resulting cover. Pick any point $t_i \in T_i$ and let $f_i = f_{t_i}$ be the diffeomorphism given by $f$ at $t_i$. Then we define a new family $f'\colon \Delta^k \times W \to W$ by the formula
	\[ f'(t,w) = \pi\left(\sum_{i=0}^n \lambda_i(t) f_i(w)\right), \]
	with the sum taken in $V \times [0,\infty)$. Of course, this formula only makes sense when the sum lies in $\Omega$. This is certainly true whenever $w \in U$ because every map $f_t$ is the identity on $U$.
	
	By definition, the values and first derivatives of the functions $f_t(-)$ vary continuously over $\Delta^k \times W$, and therefore are uniformly continuous because $\Delta^k \times W$ is compact. So by making the cover $\{T_i\}_{i=1}^n$ sufficiently fine, we may arrange that the $C^1$ distance from $f_t$ to any $f_i$ for which $\lambda_i(t) > 0$ to be less than $\epsilon$. This implies also that $C^1$ distance from $f_t$ to the sum $\sum_{i=0}^n \lambda_i(t) f_i$ is also less than $\epsilon$. Taking $\epsilon$ to be less than the distance from the compact set $f(\Delta^k \times (W \setminus U)) \subseteq (W \setminus U) \subseteq V \times [0,\infty)$ to the complement of $\Omega$, we therefore see that on $\Delta^k \times (W \setminus U)$, the sum stays within $\Omega$, so that the formula makes sense. Once it makes sense, it is clearly smooth and agrees with $f$ on $\partial \Delta^k$ by construction.
	
	Since $W$ is a compact subspace of an open manifold $W'$, the space of $C^1$ embeddings $W \to W'$ is an open subspace of $C^1(W,W')$. Therefore given $f\colon \Delta^k \to C^1(W,W')$, there is an $\epsilon$ such that any map $W \to W'$ whose $C^1$-distance to any $f_t$ is less than $\epsilon$ must be an embedding.\footnote{This is similar in spirit to \autoref{extend_embedding}, but different because we are talking about continuous families of smooth maps here.} Applying this to the above map $f'$, we conclude that for any sufficiently fine cover, all of the maps $f'(t,-)$ are $C^1$ embeddings $W \to W$. Of course we already know they are smooth as well. Each $f'(t,-)$ induces a map $F \to F$ that is the identity on $F \cap U$ and therefore on $\partial F$, so by \autoref{embedding_is_diffeo} it is a diffeomorphism $F \cong F$. Then $f'(t,-)$ induces a map $W \to W$ that is the identity on $C$ and a diffeomorphism on $F$, therefore a homeomorphism on $\partial W$, so by \autoref{embedding_is_diffeo} again it is a diffeomorphism $W \cong W$. (As above, when $U$ is empty we also have to point out that these embeddings are homotopic to ones sending $F$ to $F$ or $W$ to $W$, in order for \autoref{embedding_is_diffeo} to apply.)
	}
	
	Applying $\pi$ to a straight-line homotopy gives a deformation of this smooth family back to the original continuous family of diffeomorphisms. \orange{It is well-defined because the sum $\sum_{i=0}^n \lambda_i(t) f_i$ is within $\epsilon$ of $f_t$, and so any interpolation between the two is also within $\epsilon$ of $f_t$. At each stage of this homotopy we get a continuous family of diffeomorphisms} since they are embeddings and respect both $U$ and $F$. This concludes the proof.
\end{proof}

\orange{
\begin{rem}
	This definition of the simplicial group of diffeomorphisms $\mc D_{\sbt}(W,F)$ 
	is different, and simpler, than the standard definition in \cite[App A, Sec 3]{blr75}. The standard definition requires collar conditions on the simplices in order to guarantee that smooth extensions exist, see e.g. \cite[Def 2.2.1]{hllrw}, \cite[Sec 1.4]{krannich_homological}, and \cite[Def 2.8]{me_ww}, and as a result only obtains a semi-simplicial group, i.e. one without degeneracy maps. (See \cite[Sec 1.4]{krannich_homological} and \cite[Rem 2.2.2]{hllrw}.) The above proof, which combines the argument from \cite[Prop 1]{lurie_937_lecture6} with the smooth extension trick from \autoref{smooth_on_boundary}, avoids these issues and shows that this definition gives a simplicial group with the correct homotopy type.
\end{rem}
}

\section{Pseudoisotopies on manifolds with corners}\label{pseudosection}

Let $M$ be a compact smooth $G$-manifold with corners. In this section we will define a simplicial group  $\mc P_{\sbt}(M) $, the space of smooth pseudoisotopies on $M$. Inside $\mc P_{\sbt}(M) $ we will define a homotopy equivalent subspace  $\mc P^{mr}_{\sbt}(M)$, the space of \emph{mirror} pseudoisotopies. It is designed in such a way that the polar stabilization of a mirror pseudoisotopy is again a mirror pseudoisotopy. This prepares the way for a similar construction involving spaces of smooth $h$-cobordisms.

\subsection{Pseudoisotopies and mirror pseudoisotopies}
Recall that for $W$ a compact $G$-manifold with corners, and $F \subseteq \partial W$ a face, \autoref{smooth_diffeo_space} gives a space $\mc D_{\sbt}(W,F)$ of equivariant diffeomorphisms of $W$ that are the identity near the closure of $\partial W \setminus F$. The case of interest for us is when
\[ W = M \times [-1,0] \cong M \times I \]
with $M$ a compact $G$-manifold with corners, and $F = M \times \{0\}$ is the top face. We call this the space of equivariant pseudoisotopies on $M$:
\[ \mc P_{\sbt}(M) = \mc D_{\sbt}(M \times [-1,0],M \times \{0\}). \]

(We use $[-1,0]$ instead of $[0,1]$ because later we will be doubling $W$ along the top and we like $[-1,1]$ better than $[0,2]$.).

Let $r\colon M \times [-1,1] \to M \times [-1,1]$ be the reflection map $r(x,t) = (x,-t)$. Given a pseudoisotopy $f\colon M \times [-1,0]\to M \times [-1,0]$, the \ourdefn{double} of $f$ is the map
\[ \bar f\colon M \times [-1,1] \to M \times [-1,1] \]
that commutes with $r$ and agrees with $f$ on $M \times [-1,0]$. Note that if we write $\bar f=(\bar f_0,\bar f_1)$ with
\begin{align*}
	\bar f_0\colon & M \times [-1,1] \to M \\
	\bar f_1\colon & M \times [-1,1] \to [-1,1]
\end{align*}
then the requirement that $\bar f$ commutes with $r$ means that $\bar f_0(x,-t) = \bar f_0(x,t)$ and $\bar f_1(x,-t) = -\bar f_1(x,t)$ for all $(x,t )\in M \times [-1,1]$.

We call $f$ a \ourdefn{\mirror} pseudoisotopy if its double $\bar f$ is smooth. In \autoref{fig:pseudoisotopy}, the pseudoisotopy on the right is mirror, while the one on the left is not.

\begin{figure}[h]
	\centering
	\def\svgwidth{3.5in}
	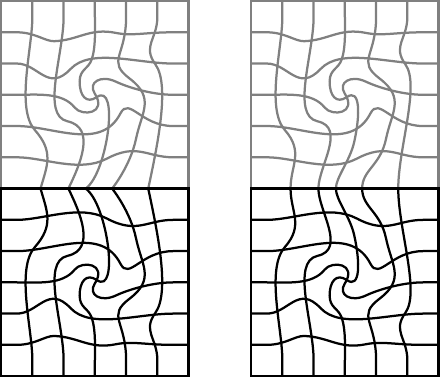
	\vspace{-.5em}
	\caption{Two pseudoisotopies on $I$ and their doubles. The pseudoisotopy on the right is mirror, while the one on the left is not.}\label{fig:pseudoisotopy}
\end{figure}
The doubles of the mirror pseudoisotopies are precisely those diffeomorphisms $\bar f$ of $M \times [-1,1]$ that are $(G \times C_2)$-equivariant ($C_2$ being the group of order $2$, acting by $r$) and such that $\bar f$ coincides with the identity on a neighborhood of the boundary. (Note that these conditions imply that the lower half of $M \times [-1,1]$ is sent to itself, so that $\bar f$ is in fact the double of some $f$.) 
When $f$ is mirror, we frequently drop the bar and just write $f$ for the double. We may even go further, extending $f$ to a diffeomorphism
	\[ f\colon M \times \R \to M \times \R \]
	\[ f(x,t) = (f_0(x,t),f_1(x,t)) \]
	such that $f_0$ is a $G$-equivariant even function to $M$, $f_1$ is a $G$-invariant odd function to $\R$, and $f$ is the identity on a neighborhood of the closure of the complement of $M \times [-1,1]$.

Let $\mc P_{\sbt}^{\mr}(M) \subseteq \mc P_{\sbt}(M)$ denote the subspace of mirror pseudoisotopies on $M$.

\begin{prop}\label{mirror_pseudo_equivalence}
	The inclusion $\mc P_{\sbt}^{\mr}(M) \subseteq \mc P_{\sbt}(M)$ is a weak equivalence.
\end{prop}

\begin{proof}
\orange{For this proof, call a pseudoisotopy \ourdefn{regular} if on some neighborhood $M \times (-\epsilon,0]$ of the top face of $M\times [-1,0]$ it coincides with the product of a diffeomorphism $ M \to M$ and the identity in the $I$ coordinate. Define a subspace $\mc P_{\sbt}^{\reg}(M)$ of $\mc P_{\sbt}(M)$, consisting of $\Delta^k$-families of regular pseudoisotopies on $M$ such that there is a uniform choice of $\epsilon$ valid for all members of the family. Regular implies mirror, so we have inclusions


$$
\mc P_{\sbt}^{\reg}(M) \subseteq \mc P_{\sbt}^{\mr}(M) \subseteq \mc P_{\sbt}(M).
$$

It will suffice if both of the inclusions $\mc P_{\sbt}^{\reg}(M) \subseteq \mc P_{\sbt}^{\mr}(M)$ and $\mc P_{\sbt}^{\reg}(M) \subseteq \mc P_{\sbt}(M)$ are equivalences.
		
	For the latter it suffices if for every commutative solid arrow diagram
	\[ \xymatrix{
		\partial \Delta[k] \ar[d] \ar[r] & \mc P_{\sbt}^{\reg}(M) \ar[d] \\
		\Delta[k] \ar@{-->}[ur] \ar[r] & \mc P_{\sbt}(M)
	} \]
	we can deform the bottom map by a simplicial homotopy rel $\partial \Delta[k]$ (preserving the commutativity) to one that has a lift as shown. The bottom map is a family of pseudoisotopies $f\colon \Delta^k \times M \times [-1,0] \to M \times [-1,0]$, and along $\partial\Delta^k$ the pseudoisotopies are regular. By assumption, every pseudoisotopy in $f$ is the identity along some neighborhood of the bottom and sides. Focusing on the sides and applying the tube lemma, we get a fixed neighborhood $U \subseteq M$ of the boundary $\partial M$ such that each pseudoisotopy in $f$ is the identity along $U \times [-1,0]$. By \autoref{smooth_simplicial_homotopy}, our goal is to give a $(\Delta^1 \times \Delta^k)$-family of pseudoisotopies from the $\Delta^k$-family $f$ to one in which the entire $\Delta^k$-family is regular, without changing the family on $\partial \Delta^k$ throughout the homotopy.
	
	We first describe the $\Delta^k$-family at the end of this homotopy, because then the rest of the homotopy will essentially be a straight line. Along the top $M \times \{0\}$, $f$ gives a $\Delta^k$-family of diffeomorphisms $f_0\colon \Delta^k \times M \to M$, that are all the identity on $U$. Multiplying this by the identity of $[-1,0]$ gives a family of diffeomorphisms $(f_0 \times \id)$, no longer a family of pseudoisotopies because they are not the identity along the bottom $M \times \{-1\}$, though they are still the identity and therefore agree with $f$ along $U \times [-1,0]$.
	
	As in the proof of \autoref{smooth_simplices}, pick an equivariant smooth embedding $M \to V$, where $V$ is a sufficiently large $G$-representation, giving a smooth embedding $M \times [-1,0] \to V \times [-1,0]$. By the construction in the proof of \autoref{weak_tubular_nbhd_thm}, there is a smooth equivariant retract $\pi \times \id\colon \Omega \times [-1,0] \to M \times [-1,0]$, where $\Omega$ is a tubular neighborhood of $M$ and therefore a neighborhood of $\textup{int }M$.
	
	For each $\epsilon > 0$ sufficiently small, take a smooth partition of unity $\{\lambda_0,\lambda_1\}$ on $[0,1]$ subordinate to $[-1,-\epsilon)$ and $(-2\epsilon,0]$, and define the family $f'\colon M \times [-1,0] \to M \times [-1,0]$ by the formula
	\[ f'(x,t) = (\pi \times \id)\left(\lambda_0(t)f(x,t) + \lambda_1(t)(f_0(x),t)\right), \]
	with the sum taken in $V \times [-1,0]$. (We have left the dependence on the point in $\Delta^k$ implicit.) As before, this formula only makes sense when the first coordinate of the sum lies in $\Omega$. This is certainly true whenever $x \in U$ because both $f$ and $f_0$ are the identity in that case. By making $\epsilon$ sufficiently small, we can make the $C^0$ distance from $f$ to $\lambda_0 f + \lambda_1 f_0$ as small as desired, and therefore make it smaller than the distance from $M \setminus U$ to the complement $V \setminus \Omega$. This ensures that the above sum lands inside $\Omega$ and therefore the formula is well-defined. By construction, on $M \times (-\epsilon,0]$ we have $f'(x,t) = (f_0(x),t)$, so that this is a family of regular pseudoisotopies. Furthermore, if $\epsilon$ is chosen sufficiently small then for every point in $\partial\Delta^k$ we have that $f = f_0 \times \id$ on $(-2\epsilon,0]$ and therefore $f'(x,t) = f(x,t)$ along the boundary $\partial\Delta^k$.
	
	Now apply $(\pi \times \id)$ to a straight-line homotopy to give a $(\Delta^1 \times \Delta^k)$-family of pseudoisotopies from $f$ to this new family of regular pseudoisotopies $f'$, that is a constant homotopy on $\partial \Delta^k$. This produces the desired simplicial homotopy, showing that $\mc P_{\sbt}^{\reg}(M) \subseteq \mc P_{\sbt}(M)$ is a weak equivalence.

	If the family $f$ is mirror then this entire procedure produces mirror pseudoisotopies, giving a lift
	\[ \xymatrix{
		\partial \Delta[k] \ar[d] \ar[r] & \mc P_{\sbt}^{\reg}(M) \ar[d] \\
		\Delta[k] \ar@{-->}[ur] \ar[r] & \mc P_{\sbt}^{\mr}(M)
	} \]
	and showing that $\mc P_{\sbt}^{\reg}(M) \subseteq \mc P_{\sbt}^{\mr}(M)$ is a weak equivalence. We conclude that $\mc P_{\sbt}^{\mr}(M) \subseteq \mc P_{\sbt}(M)$ is a weak equivalence as well.
	}
\end{proof}

\subsection{Smoothness properties of even and odd functions}
In order to understand stabilization of mirror pseudoisotopies, we need to recall some facts about even and odd functions. Let $M$ be any smooth manifold with corners. Let
\[ r\colon M \times \R \to M \times \R \]
be the reflection map $r(x,t) = (x,-t)$. We say that a function
\[ f\colon M \times \R \to \R \]
is \emph{odd} if $f \circ r = -f$, and that a function
\[ f\colon M \times \R \to X \]
to any set $X$ is \emph{even} if $f \circ r = f$. These definitions also apply when $f$ is only defined on some $r$-invariant subset of $M \times \R$, such as $M \times [-1,1]$.

\begin{lem}\label{odd_factor_out_t}
	If $f\colon M \times \R \to \R$ is smooth, and if $f(x,0)=0$ for all $x\in M$, then $f(x,t) = tg(x,t)$ for a smooth function $g\colon M \times \R \to \R$. In particular if $f$ is smooth and odd then $f(x,t) = tg(x,t)$ where $g$ is smooth and even.
\end{lem}

\begin{proof}
	The following well known argument is valid even when $M$ has corners. Write $f'(x,t)=\frac{\partial}{\partial t}f(x,t)$. Then
	\[ f(x,t)=f(x,t)-f(x,0)=\int_0^t f'(x,u)du=t \int_0^1 f'(x,st)\ ds. \]
	The last expression is a smooth function of $(x,t)$ by differentiation under the integral.
	\end{proof}

\begin{lem}
	If $f\colon M \times \R \to \R$ is smooth and even then $f(x,t) = h(x,t^2)$ for a smooth function $h$.
\end{lem}

\begin{proof}
\orange{For $u>0$ define $h(x,u)=f(x,\sqrt u)$. Since $f$ is even, $f(x,t)=h(x,t^2)$ for $t\neq 0$. The assertion is that $h$ extends smoothly from $u>0$ to $u\in \mathbb R$. By \cite[Theorem 1.4.1]{melrose} or \cite{whitney} it is enough if $h(x,u)$, $\frac{\partial}{\partial u}h(x,u)$, and more generally $\frac{\partial^r}{\partial^r u}h(x,u)$ all extend continuously to $u\ge 0$. 

Write $h'(x,u)=\frac{\partial}{\partial u}h(x,u)$. The key is this implication (which is then applied iteratively): if $h(x,u)$ is defined and smooth for $u>0$, and if $h(x,t^2)$ extends smoothly from $t\neq 0$ to $t\in \mathbb R$, then the same holds for $h'(x,u)$. 

To establish the implication, note that if $h(x,t^2)=f(x,t)$, where the right hand side is defined and smooth for all $t$, then $2th'(x,t^2)$ also extends smoothly to all $t$, so that by \autoref{odd_factor_out_t}, $h'(x,t^2)$ extends smoothly to all $t$.}

\end{proof}
	
\begin{cor}\label{even_function_of_t2}
		If $N$ is any smooth manifold with corners and $f\colon M \times \R \to N$ is smooth and even, then the function
	\[ g\colon M \times [0,\infty) \to N, \qquad g(x,t) = f(x,\sqrt{t}) \]
	is also smooth. 
\end{cor}
	
\subsection{Polar stabilization}
	Suppose that $f:M\times \R\to M\times \R$ is an equivariant mirror pseudoisotopy. 	
	Given an orthogonal $G$-representation $V$, define
	\[ \St^V f\colon M \times V \times \R \to M \times V \times \R \]
	to be the function that applies $f$ along $M\times L$ for every line $L$ through the origin in $V \times \R$. In formulas,
	\[ (\St^V f)(x,v,t) = \left(f_0(x,\|(v,t)\|),\frac{v}{\|(v,t)\|}f_1(x,\|(v,t)\|),\frac{t}{\|(v,t)\|}f_1(x,\|(v,t)\|)\right), \]
	\orange{with the last two coordinates set to zero if $(v,t) = (0,0)$.} The formula becomes simpler if we write $W = V \times \R$. Then for $(x,w)\in M\times W$ we have
	\[ (\St^V f)(x,w) = \left(f_0(x,\|w\|),\frac{w}{\|w\|}f_1(x,\|w\|)\right), \]
	\orange{with the last coordinate set to zero if $w = 0$.} Note that when $\|w\| \geq 1$ then
		\[ (\St^V f)(x,w)  = \left(x,\frac{w}{\|w\|} \cdot \|w\|\right) = (x,w). \]
	As a result we can regard $\St^V f$ as a pseudoisotopy of $M \times D(V)$ instead of $M \times V$.
	
	\begin{lem}\label{st_pseudo_smooth}
		$\St^V f$ is smooth, $G$-equivariant, and {\mirror}.
	\end{lem}
	
	\begin{proof}
		It is straightforward to see that $f_0(x,\|(v,t)\|)$ and $\frac{v}{\|(v,t)\|}f_1(x,\|(v,t)\|)$ are even, $\frac{t}{\|(v,t)\|}f_1(x,\|(v,t)\|)$ is odd, and all three are $G$-equivariant. Clearly $\St^V f$ is smooth away from $w = 0$. For smoothness at $w=0$, note that by \autoref{odd_factor_out_t} and \autoref{even_function_of_t2}
		\[ (f_0(x,t),f_1(x,t)) = (g_0(x,t^2),tg_1(x,t^2)) \]
		for some smooth equivariant functions $g_0$ and $g_1$ on $M \times [0,\infty)$, and therefore
		\[ (\St^V f)(x,w) = \left(g_0(x,\|w\|^2),w \cdot g_1(x,\|w\|^2)\right). \]
		
				The key is that $\|w\|^2$ is smooth (even though $\|w\|$ is not). Finally, $\St^V f$ is a diffeomorphism because its inverse is $\St^V(f^{-1})$.
	\end{proof}
	
	The same applies to a $k$-simplex of pseudoisotopies $\Delta^k \times M \times \R \to M \times \R$. The only difference is that now the functions $f_0$ and $f_1$ (and therefore $g_0$ and $g_1$ in the proof) depend smoothly on an extra argument $u \in \Delta^k$. The operation $\St^V$ clearly respects faces and degeneracies, and hence defines a map
\[ \St^V\colon \mc P_{\sbt}^{\mr}(M) \to \mc P_{\sbt}^{\mr}(M \times D(V)). \]
	
	The next lemma is functoriality for pseudoisotopies along inclusions of vector spaces.
	
	\begin{lem}
		For representations $V$ and $W$, $\St^{W \times V} f = \St^W\St^V f$.
	\end{lem}

	\begin{proof}
		One can check this directly from the formulas, but it perhaps clearer to say this: $\St^W\St^V f$ applies $\St^V f$ along every subspace of $W \times V \times \R$ that is $V$ times a line in $W \times \R$, and $\St^V f$ applies $f$ along every line in this subspace, so that $\St^W\St^V f$ applies $f$ along every line in $W \times V \times R$, and this matches the definition of $\St^{W \times V} f$.
	\end{proof}

As a result we get the commuting square
\[ \xymatrix @C=5em{
	\mc P_{\sbt}^{\mr}(M) \ar[d]^-{\St^{V \oplus W}} \ar[r]^-{\St^V} & \mc P_{\sbt}^{\mr}(M \times D(V)) \ar[d]^-{\St^W} \\
	\mc P_{\sbt}^{\mr}(M \times D(V \times W)) \ar[r]^-{\textup{extend by id}} & \mc P_{\sbt}^{\mr}(M \times D(V) \times D(W)).
} \]

\begin{rem}
	For the most part, the results in this section serve as a technical underpinning and a plausibility check for the more sophisticated stabilization we perform later on $h$-cobordisms. It should be possible to go further and establish functoriality of pseudoisotopies along embeddings of manifolds, as in \cite{pieper}, but we do not do so here.
\end{rem}

\section{$h$-cobordisms on manifolds with corners}\label{hcobsection}

In this section we define the space of $h$-cobordisms on a compact smooth $G$-manifold with corners. We also describe a homotopy-equivalent space of mirror $h$-cobordisms. As with pseudoisotopies, the additional mirror structure is used in creating ``polar'' stabilizations having good formal properties.

\subsection{Definitions}

Let $M$ be a compact smooth $n$-dimensional $G$-manifold with corners. An \ourdefn{equivariant cobordism} on $M$ is a compact $(n+1)$-manifold $W$ with corners, equipped with the following structure. There is a face of $W$ identified with $M$, called the bottom. There is a face $N$ of $W$, disjoint from $M$, called the top. The closure of the complement of $M\cup N$ in $\partial W$ is called the sides. \orange{We let $L \subseteq \partial W$ be the union of the bottom and sides.} There is an equivariant diffeomorphism between a neighborhood of the bottom and sides of $M\times [-1,0]$ and a neighborhood of \orange{$L$} in $W$, taking $M\times \{-1\}$ to the bottom of $W$ and taking $\partial M\times [-1,0]$ to the sides of $W$, and therefore taking a neighborhood of $\partial M\times \{0\}$ in $ M\times \{0\}$ to a neighborhood of $\partial N$ in $N$. We refer to this embedding of a neighborhood of $M\times \{-1\}\cup\partial M\times [-1,0]$ in $W$ as the \ourdefn{lower collar} and denote it by $c$. The germ of $c$ along $L$ is part of the structure of the equivariant cobordism.

(Again, we use $[-1,0]$ instead of $[0,1]$ because later we will be doubling $W$ along the top and we like $[-1,1]$ better than $[0,2]$.). 

\begin{defn}\label{hcob_def} An \ourdefn{equivariant $h$-cobordism} on $M$ is a cobordism as above such that
	the inclusions $M \to W$ and $N \to W$ are equivariant homotopy equivalences.
\end{defn}


\begin{figure}[h]
	\centering
	\def\svgwidth{3.6in}
	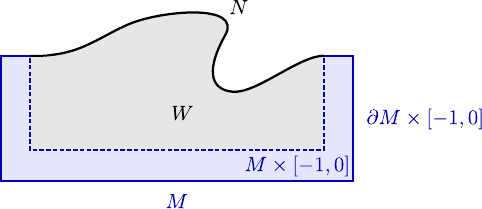
	\vspace{-.5em}
	\caption{An $h$-cobordism on a manifold with corners $M$.}\label{fig:h-cobordism}
\end{figure}

For definiteness, we fix a sufficiently large set $\mc U$ containing $M$ and assume that the underlying set of each cobordism $W$ is a subset of $\mc U$.

\orange{
We recall the notion of an NDR-pair and a DR-pair from e.g. \cite[\S 6.4]{concise}. We always assume that $A \subseteq X$ is a closed $G$-invariant subspace, and say that $(X,A)$ is an equivariant NDR-pair if the inclusion has the equivariant homotopy extension property, equivalently if $X \times I$ admits an equivariant continuous retraction to $(A \times I) \cup (X \times \{0\})$. It is a DR-pair if in addition $A \to X$ is an equivariant homotopy equivalence. For instance, if $A$ has a mapping cylinder neighborhood in $X$ then $(X,A)$ is an NDR-pair.

\begin{lem}\label{partial_boundary_is_ndr}
	If $M$ is any $G$-manifold with corners and $A \subseteq \partial M$ is any closed partial boundary, then $(M,A)$ is an equivariant NDR-pair.
\end{lem}

\begin{proof}
	By \autoref{flow_from_partial_boundary} we may form a neighborhood of $A$ in $M$ that is equivariantly homeomorphic to $A \times I$, in other words a mapping cylinder neighborhood.
\end{proof}

\begin{lem}\label{pushout_product}
	If $(X,A)$ and $(Y,B)$ are NDR-pairs then the pushout-product $(X \times Y,(X \times B) \cup (A \times Y))$ is an NDR-pair. Furthermore, if one of the two is a DR-pair then the pushout-product is a DR-pair. 
\end{lem}

\begin{proof}
	This is standard in the non-equivariant case. In the equivariant case, the spaces have $G$-action, and by inspection the formula giving the required retraction from \cite[\S 6.4]{concise} respects that $G$-action.
\end{proof}

\begin{lem}\label{l_to_w_is_dr}
	If $W$ is an equivariant $h$-cobordism then $(W,L)$ is an equivariant DR-pair.
\end{lem}

\begin{proof}
	By \autoref{partial_boundary_is_ndr} it is an equivariant NDR-pair. But $M \to L$ is an equivariant homotopy equivalence since we may deformation retract the sides $\partial M \times [-1,0]$ back onto $\partial M \times \{-1\}$. Since $M \to W$ is assumed to be an equivariant homotopy equivalence, we conclude that $L \to W$ is an equivariant homotopy equivalence.
\end{proof}
}

The \ourdefn{double} $\bar W$ of an $h$-cobordism $W$ is two copies of $W$ glued along their common top face $N$. We think of one as ``flipped over'' and use the interval $[0,1]$ in the place of $[-1,0]$, with $M$ located at 1 and $N$ located at 0. 

\begin{defn}\label{defn_encasing_function} We define some kinds of structures on $h$-cobordisms.
\begin{enumerate}
	\item A \ourdefn{\mirror} $h$-cobordism is an $h$-cobordism $W$ together with a $(G \times C_2)$-equivariant smooth structure on its double $\bar W$ that restricts to the given smooth structure of $W$, \orange{and that on the lower collar agrees with the standard smooth structure on $M \times [-1,1]$}.
	\item Given a {\mirror} $h$-cobordism, an \ourdefn{\rh} is a smooth $G$-equivariant map
	\[ \rho = (r,h)\colon W \to M \times [-1,0], \]
	having the following properties.
	\begin{itemize}
		\item The composition $\rho\circ c$ coincides with the identity map of $M\times [-1,0]$ in a neighborhood of the bottom and sides. In particular, $r:W\to M$ is a retraction to the bottom.
		\item The function $h$ satisfies \orange{$h^{-1}(-1) = M$} and $h^{-1}(0) = N$, and the derivative $Dh$ has rank 1 along $N$; we call $h$ the \emph{height function}.
		\item \orange{The retraction $r\colon W \to M$ sends $W$ minus the sides to the interior of $M$. }
		\item The double of $\rho$, in other words the map $\bar W \to M \times [-1,1]$ that extends $\rho$ and commutes with reflection, is smooth. (We may again denote this extended map by $\rho$, or by $(r,h)$.) Note that the double of $\rho$ maps a neighborhood of the boundary of $\bar W$ to a neighborhood of the boundary of $M\times [-1,1]$ by a diffeomorphism. 
	\end{itemize}
\end{enumerate}
An \ourdefn{encased} $h$-cobordism is a mirror $h$-cobordism equipped with \arhperiod This additional structure will be useful when we introduce stabilizations of $h$-cobordisms.

\end{defn}

\begin{figure}[h]
	\centering
	\def\svgwidth{5.5in}
	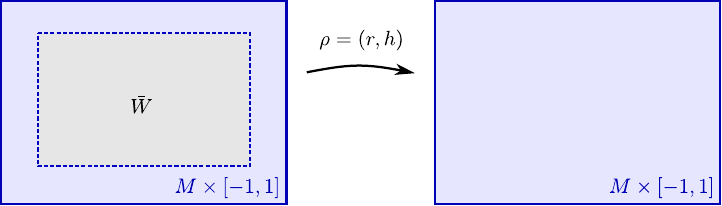
	\vspace{-.5em}
	\caption{An encased $h$-cobordism.}\label{fig:encased}
\end{figure}

Note that when making an encased $h$-cobordism it is redundant to specify the lower collar $c$, since it must coincide with the inverse of \therh  $\rho$ near the boundary.

For a fixed {\mirror} $h$-cobordism $W$, we can define a space of \rhs by defining a $k$-simplex to be a smooth map $\Delta^k \times \bar W \to M \times [-1,1]$ that for each point of $\Delta^k$ defines an encasing function, and such that there is a single open set $U \subseteq W$ containing the bottom and sides on which the \rhs agree with $c^{-1}$, throughout the entire simplex $\Delta^k$.

\orange{
\begin{prop}\label{rh_contractible}
	The space of equivariant \rhs on an equivariant {\mirror} $h$-cobordism is a contractible Kan complex.
\end{prop}

\begin{proof}
	The space of \rhs is a product of two spaces, a space of retractions and a space of height functions, so we consider the two factors separately. Without loss of generality $W$ is connected.
	
	A $k$-simplex of the space of height functions is a $(G\times C_2)$-invariant map $h:\Delta^k\times \bar W\to \R$ satisfying the following conditions. In a neighborhood of each point of $\Delta^k\times \partial\bar W$, $h$ is given by $c^{-1}$. For $u\in \Delta^k$ and $w$ in the interior of the lower copy of $W$, $h(u,w)\in (-1,0)$, and therefore for $u\in \Delta^k$ and $w$ in the interior of the upper copy of $W$, $h(u,w)\in (0,1)$. At each point $(u,w)\in \Delta^k\times N$, $h(u,w)=0$ and $Dh$ has rank $1$.

	We first show that the space of height functions is nonempty. Let $\bar U \subseteq \bar W$ be a (doubled) open neighborhood of the boundary, on which the height function $h$ must already be already defined by the lower collar. By \autoref{trimmings_exist}, there exists a trimming $N_0 \subseteq N$ whose complement is contained in $\bar U$. By the tubular neighborhood theorem (\autoref{tubular_nbhd_thm}), the inclusion $N_0 \to W$ extends to a $G \times C_2$-equivariant tubular neighborhood, in other words an equivariant bicollar $N_0 \times [-1,1] \cong \bar V \subseteq \bar W$. The collar coordinate then gives a definition of $h$ on $\bar V$, and we can patch this together with the existing $h$ on $\bar U$ by a smooth partition of unity subordinate to $\{\bar U,\textup{int }\bar V\}$.
	
	Shrinking these domains slightly to a $G \times C_2$-invariant closed neighborhood $\bar C$ of $\partial \bar W \cup N$, we now have a $G$-equivariant height function $h\colon C \to [-1,0]$, such that $h^{-1}(0) = N$, in other words $h(C \setminus N) \subseteq [-1,0)$. By the Tietze extension theorem, this may be extended to a continuous function $W \setminus N \to [-1,0)$ while preserving its original value on $C \setminus N$. The extension can be made equivariant by averaging, and then we may use smooth approximation (\autoref{smooth_approximation}) to change $h$ rel a smaller closed neighborhood $C'$ of $\partial W \cup N$, to make it equivariant and smooth. This finishes the construction of one height function with the desired properties.
	
	To show that the height functions form a contractible Kan complex, we must extend any $\partial \Delta^k$ of height functions to $\Delta^k$ in a smooth way. The basic idea is convexity, but we have to be careful because extending a $\partial \Delta^k$-family of height functions to $\Delta^k$ by coning off to a single height function does not in general produce a smooth map, even away from the cone point.
	
	Assume we are given a $\partial \Delta^k$-family of height functions $h_0\colon \partial \Delta^k \times \bar W \to [-1,1]$. By definition, there is a uniform choice of neighborhood $U_0$ of the bottom and sides along which $h_0$ agrees with the lower collar of $W$. (There is one for each face in $\partial \Delta^k$, and we intersect them.) We use smooth extension, specifically the version in \autoref{smooth_extension_with_neat}, to extend the given height function
	\[ h_0\colon (\partial \Delta^k \times \bar W) \cup (\Delta^k \times N) \to [-1,1] \]
	to a smooth function $h_1\colon T \times \bar W \to \R$, where $T$ is an open neighborhood of $\partial \Delta^k$ inside $\Delta^k$. This extension $h_1$ doesn't have the correct behavior in $T \times \bar U_0$, so we let $U_1 \subseteq U_0$ be a smaller open neighborhood whose closure is contained in $U_0$, Use a partition of unity subordinate to $\{\bar U_0,\bar U_1^c\}$ to interpolate between the canonical height function on $\bar U_0$ and the function $h_1$ on $\bar U_1^c$, as functions to $\R$. This gives a new smooth function $h_2\colon T \times \bar W \to \R$. Shrinking $T$ if necessary, we can guarantee that the derivative $Dh_2$ has full rank along $T \times N$, because it had full rank along $\partial \Delta^k \times N$.
	
	
	This height function $h_2$ does not need to stay in $[-1,1]$, but on the open set $h_2^{-1}((-1,1)) \cup (T \times \bar U)$, $h_2$ either stays in $(-1,1)$ or attains the endpoints $[-1,1]$ where it is supposed to, inside $\bar U$. This open set in $T \times \bar W$ contains $\partial \Delta^k \times \bar W$ by the assumptions we made in \autoref{defn_encasing_function}, so by the tube lemma it contains $T \times \bar W$ for some possibly smaller value of $T$. So now $h_2\colon T \times \bar W \to [-1,1]$. Similarly, we can take an open set about $T \times N$ on which $h^{-1}(\{0\}) = T \times N$, and take the union with the open set $h_2^{-1}([-1,0) \cup (0,1])$, giving an open set in $T \times \bar W$ on which $h^{-1}(\{0\}) = T \times N$, that contains $\partial \Delta^k \times \bar W$. So by shrinking $T$ again we can guarantee that $h_2^{-1}(\{0\}) = T \times N$.
	
	Now we use a partition of unity of $\Delta^k$ subordinate to $\{T,\textup{int }\Delta^k\}$ to interpolate between this $T$-family of height functions $h_2$ and a constant family on the height function we constructed earlier in the proof. This gives a smooth map $h_3\colon \Delta^k \times \bar W \to [-1,1]$. The interpolated maps are still full-rank along $N$ since the set of maps $\bar W \to [-1,1]$ whose first derivative along $N$ is ``positive'' (sending the bottom half $W$ to $[-1,0]$) is convex, as is the set of maps that sends $W \setminus N$ to $[-1,0)$. Therefore $h_3$ is a $\Delta^k$-family of height functions extending $h_0$, as desired.
	
	Now we turn to retractions; it will be easier here to prove the result all at once. Suppose we have a $\partial \Delta^k$-family of retractions $r_0\colon \partial \Delta^k \times \bar W \to M$. There is some uniform choice of neighborhood $\bar U$ of the boundary $\partial \bar W$ on which these retractions all agree with the one given by the lower collar. By the tube lemma, $\bar U$ contains a product $S \times [-1,1]$ where $S \subseteq M$ is an open neighborhood of $\partial M$. Let $M_0 \subseteq M$ be a trimming with the closure of $M \setminus M_0$ contained in $S$. Let $\bar W_0 \subseteq \bar W$ be the corresponding submanifold in which we have removed the open set $(M \setminus M_0) \times [-1,1]$. Then $r_0$ sends $\partial \Delta^k \times \bar W_0$ into the interior of $M$, since by assumption the retractions only send the sides of $\bar W$ to $\partial M$.
	
	By \autoref{smooth_extension}, 
	$r_0$ extends to a smooth map $r_1\colon T \times \bar W_0 \to \textup{int }M$, for some open set $T \times \bar W_0$ containing $\partial \Delta^k \times \bar W_0$. As for the height functions, we then interpolate between $r_1$ on $T \times \bar W_0$ and the canonical projection on $\Delta^k \times \bar U$. To describe the open sets we use, use the fact that $W$ is normal to separate the closed set $U^c$ from the closure of $(M \setminus M_0) \times [-1,0]$ by two open sets; this gives an open set $U_1$ and a closed set $C_1$ such that
	\[ (M \setminus M_0) \times [-1,0] \subseteq U_1 \subseteq C_1 \subseteq U. \]
	Then we interpolate between $r_1$ and the canonical projection on $\Delta^k \times \bar U$ using a partition of unity subordinate to the doubled open sets $\{\bar C_1^c,\bar U\}$.
	
	$M$ is not a vector space, so to make sense of this interpolation we need to embed $M$ into a representation $V$, and get an open neighborhood $\Omega$ that smoothly retracts to $\textup{int }M$, as in the proof of \autoref{smooth_simplices} and \autoref{mirror_pseudo_equivalence}. We observe that the compact set $\partial \Delta^k \times \bar U_1^c$ is sent by the projection to $\textup{int }M \subseteq  \Omega$, and therefore after shrinking $T$ we get that $T \times \bar U_1^c$ is sent by the interpolated function into $\Omega$. Applying the projection $p\colon \Omega \to \textup{int M}$, we arrive at a smooth function $r_2\colon (T \times \bar W) \cup (\Delta^k \times \bar U) \to M$. It is a retract because it agrees with the canonical projection inside $\Delta^k \times C_1$.
	
	Now that we have smoothly extended to an open neighborhood of the boundary of the entire product manifold $\Delta^k \times \bar W$, we extend to the interior. Let $\bar W_1 \subseteq \Delta^k \times \bar W$ be an equivariant trimming that is supported in the neighborhood of the boundary on which $r_2$ is defined. Passing to the lower half of the double, this gives a submanifold with depth-two corners $W_1 \subseteq \Delta^k \times W$. Let $K \subseteq \Delta^k \times W$ be the closure of the complement of $W_1$. Recall that $L \subseteq \partial W$ refers to the bottom and sides of $W$.
	
	Since $K$ arises from a trimming, it admits an equivariant deformation retraction to a partial boundary of $\Delta^k \times W$, namely the union $(\Delta^k \times L) \cup (\partial \Delta^k \times W)$, in other words the entire boundary minus the interior of the face $\Delta^k \times N$. The inclusion of this partial boundary into $\Delta^k \times W$ is the pushout-product of the two maps $\partial \Delta^k \to \Delta^k$ and $L \to W$. The first is an equivariant NDR-pair by \autoref{partial_boundary_is_ndr}, and the second is an equivariant DR-pair by \autoref{l_to_w_is_dr}, so their pushout-product is an equivariant DR-pair by \autoref{pushout_product}. It follows that the inclusion $K \to W$ is an equivariant homotopy equivalence as well.
	
	Next, consider the face of $\partial W_1$ that lies inside $W$, and take an equivariant neat collar of this face in $W_1$ using \autoref{collar_isotopy_extension}. This collar 
	 gives us an equivariant mapping cylinder neighborhood of $K$ in $W$, allowing us to see that $(W,K)$ is an equivariant NDR-pair. Since it is also an equivariant homotopy equivalence, it is an equivariant DR-pair. We conclude that there is an equivariant continuous deformation retract of $W$ onto $K$.
	
	We may now define a continuous function $r_3\colon \Delta^k \times W \to M$, by first applying this retraction to land in $\Delta^k \times K$, then applying $r_2$. This new retraction $r_3$ is smooth on $\Delta^k \times (W \setminus W_1)$ and continuous everywhere, and sends everything other than the sides of $W$ to $\textup{int }M$. We double $r_3$ to get a map $\Delta^k \times \bar W \to M$, then apply equivariant smooth approximation (\autoref{smooth_approximation}). We get a function $r_4\colon \Delta^k \times \bar W \to M$ which agrees with $r_3$ near the boundary, and is therefore the canonical projection there, and is an equivariant smooth map to the interior of $M$ everywhere else.

%
%
\end{proof}
}

\subsection{The space of $h$-cobordisms}
\begin{defn}\label{encased_diffeo}
A \ourdefn{diffeomorphism of $h$-cobordisms} $W \cong W'$ over $M$ is an (equivariant) diffeomorphism of manifolds with corners whose composition with the germ of the lower collar of $W$ is the germ of the lower collar of $W'$.
A \ourdefn{{\mirror}  diffeomorphism} between {\mirror} $h$-cobordims is one whose extension to the doubles is also a diffeomorphism. An \ourdefn{encased diffeomorphism} is a mirror diffeomorphism that also commutes with the \rhsperiod
\end{defn}

Recall from \autoref{smooth_diffeo_space} that $\mc D_{\sbt}(W,N)$ is the space (simplicial set) of equivariant diffeomorphisms that coincide with the identity in a neighborhood of \orange{the bottom and sides of $W$. We take this as the definition of the space of diffeomorphisms of the $h$-cobordism $W$.}

For a fixed manifold $M$, the desired homotopy type for the space of $h$-cobordisms $\mc H(M)$ is the disjoint union, over all diffeomorphism classes of $W$, of the classifying spaces $B\mc D_{\sbt}(W,N)$. We could take this as the definition of the $h$-cobordism space, but the following definition using families is more canonical and more convenient.

\begin{defn}\label{hcob_space} A \ourdefn{$\Delta^k$-family of equivariant $h$-cobordisms} is a smooth fiber bundle $p\colon E \to \Delta^k$ with $G$-action, whose fibers are equivariant $h$-cobordisms over $M$. Specifically it has 
a face $\Delta^k \times M \subseteq E$ and a lower collar $c\colon \Delta^k \times M \times [-1,0] \to E$, 
satisfying the same conditions on $c$ as in \autoref{hcob_def}.
Again, it is only the germ of the collar along the bottom and sides that is considered to be part of the structure. \orange{However, for a given $k$-simplex of $h$-cobordisms, there must be a single open set of the form $\Delta^k \times U \subseteq \Delta^k \times M \times [-1,0]$ on which the lower collar $c$ is defined.}

For definiteness we assume that each family has as its underlying set a subset of $\Delta^k \times \mc U$. 

Families of mirror $h$-cobordisms are defined similarly, as are families of encased $h$-cobordisms. (In the latter case the retraction and height function that constitute the encasement are required to be smooth across the entire family.) We let $\bar E$ refer to the (fiberwise) double of the family $E$. \end{defn}

By \autoref{ehresmann_with_corners}, every $\Delta^k$-family of $h$-cobordisms can be trivialized, along with the lower collars. The same holds for mirror $h$-cobordisms, thinking of the double of the family $\bar E$ as a $(G \times C_2)$-equivariant fiber bundle over $\Delta^k$ containing a trivial bundle (the germ of the bottom, sides, and their reflections).

\begin{rem}
	For a family of encased $h$-cobordisms, it is not true that \therh can be trivialized as well. That is, we do not assume that the family of cobordisms is isomorphic to one of the form $\Delta^k \times W$ in such a way that the \rh on $W$ is the same for all points in $\Delta^k$. If we did assume this, it would make \autoref{encased_to_mirror} below false.
\end{rem}

\begin{defn}\label{hcob_space}
Let $M$ be a smooth compact $G$-manifold with corners. The \ourdefn{space of equivariant $h$-cobordisms over $M$}  is the simplicial set $\mc H_{\sbt}(M)$ whose $k$ simplices are families of equivariant $h$-cobordisms over $\Delta^k$. Similarly, we define the space of mirror $h$-cobordisms $\mc H_{\sbt}^m(M)$, and the space of encased mirror $h$-cobordisms $\mc H_{\sbt}^c(M)$. The face and degeneracy maps are clear. 

\end{defn}

\begin{lem}\label{h_cobordism_space_kan_complex}
	$\mc H_{\sbt}(M)$ is a Kan complex.
\end{lem}

\begin{proof}
	Given a horn $\Lambda^k_i \to \mc H_{\sbt}(M)$, each face $\Delta^{k-1} \to \mc H_{\sbt}(M)$ is a $\Delta^{k-1}$-family of $h$-cobordisms, which is isomorphic to a trivial family $W \times \Delta^{k-1}$. We identify the entire family with $W \times \Lambda^k_i$ by an induction on the faces, using \autoref{extend_trivializations}.
	
	Once the entire family has been identified with $W \times \Lambda^k_i$, we extend it to $W \times \Delta^k$. Then we take the underlying set of $W \times \Delta^k$ and apply a bijection to the subset $W \times \Lambda^k_i$ so that we get the underlying set of the original family (before trivialization). This produces a map $\Delta^k \to \mc H_{\sbt}(M)$ that along $\Lambda^k_i$ strictly agrees with our original map.
\end{proof}

\begin{lem}
	$\mc H_{\sbt}(M)$ is equivalent to the disjoint union, over all diffeomorphism classes of $h$-cobordisms $W$ over $M$, of the classifying spaces $B\mc D_{\sbt}(W,N)$.
\end{lem}

\begin{proof}
	These spaces clearly have the same components, so we restrict to a single component $\mc H_{\sbt}(M)_{[W]}$. Let $\mc E_{\sbt}(M)_{[W]}$ denote the same simplicial set except that each family $E \to \Delta^k$ is equipped with a choice of trivialization $E \cong W \times \Delta^k$, and the face and degeneracy maps respect this trivialization.
	
	Since every family in $\mc E_{\sbt}(M)_{[W]}$ has a chosen trivialization, the entire simplicial set deformation retracts onto a single fixed 0-simplex. Specifically, we extend each $\Delta^k$ family $E \cong W \times \Delta^k$ to the $\Delta^{k+1}$ family given by $W \times \Delta^{k+1}$, but with the $\Delta^k$-face changed by a bijection on the underlying set so that it agrees with $E$.
	
	Furthermore $\mc E_{\sbt}(M)_{[W]}$ has a free action by the simplicial group $\mc D_{\sbt}(W,N)$, changing the trivializations. The quotient by this action is $\mc H_{\sbt}(M)_{[W]}$ (because by \autoref{ehresmann_with_corners} every family can be trivialized). It follows  that $\mc H_{\sbt}(M)_{[W]}$ is a classifying space for $\mc D_{\sbt}(W,N)$.
\end{proof}

\begin{prop}\label{space_of_mirrorhcob}
The forgetful map from mirror to ordinary $h$-cobordisms
\[ \mc H_{\sbt}^m(M) \to \mc H_{\sbt}(M) \]
is a weak equivalence of Kan complexes.
\end{prop}

\begin{proof}
\orange{
The proof is conceptually similar to that of \autoref{mirror_pseudo_equivalence}.
}
The \orange{simplicial set $\mc H_{\sbt}^m(M)$} is a Kan complex by the same proof as in \autoref{h_cobordism_space_kan_complex} \orange{-- the analog of \autoref{extend_trivializations} holds here because the simplicial set of mirror diffeomorphisms is also a simplicial group, therefore a Kan complex.} We have a commuting diagram
\begin{equation}\label{compare_to_v}
	\xymatrix{
	&\mc H_G^{v}(M) \ar[rd] \ar[ld] &
	\\
	\mc H_G^m(M) \ar[rr] && \mc H_G(M)
	}
\end{equation}
where $\mc H_G^{v}(M)$ is the space of $h$-cobordisms equipped with the extra data of a \orange{$G$-invariant vector field that is inward-pointing along the top face $N$, and that on some neighborhood of the bottom and sides $L \subseteq W$ agrees with the constant vector field on $M \times [-1,0]$ taking the value $(0,-1) \in TM \times \R$ at every point. Such a vector field gives} the germ of a $G$-equivariant collar on $N$ \orange{that agrees with the canonical collar on $M \times \{0\} \subseteq M \times [-1,0]$ near the sides.} Such a collar gives a mirror structure in a canonical way, by doubling it to a bicollar on $\bar W$ and using it to define the smooth structure at $N$.

It suffices to show that each diagonal map of \eqref{compare_to_v} is a weak equivalence. This follows if we can produce lifts
	\begin{equation}\label{mirrorhcob_eq}
	\xymatrix{
		\partial \Delta^k \ar[d] \ar[r] & \mc H_{\sbt}^{v}(M) \ar[d] \\
		\Delta^k \ar[r] \ar@{-->}[ur] & \mc H_{\sbt}(M)
		}
	\qquad \qquad
	\xymatrix{
		\partial \Delta^k \ar[d] \ar[r] & \mc H_{\sbt}^{v}(M) \ar[d] \\
		\Delta^k \ar[r] \ar@{-->}[ur] & \mc H_{\sbt}^{m}(M).
		}
	\end{equation}
	

	For the first diagram, we first trivialize the family of cobordisms $E \cong \Delta^k \times W$, so that we can regard the vector fields as living on a fixed $h$-cobordism $W$. \orange{Then we want to construct a vector field $\xi$ on $\Delta^k \times W$ that is inward-pointing along $\Delta^k \times N$, has zero $\Delta^k$-coordinate everywhere, and is the standard downward-pointing vector field near $\Delta^k \times L$. We are given $\xi$ on $\partial \Delta^k \times W$, and it agrees with the standard vector field on $\partial \Delta^k \times U$ for some neighborhood $U$ of the bottom and sides $L \subseteq W$. We therefore know $\xi$ on the partial boundary
	\[ (\partial \Delta^k \times W) \cup (\Delta^k \times L). \]
	By \autoref{extend_inward_pointing_from_partial_boundary}, this extends to a vector field on all of $\Delta^k \times W$ that is inward-pointing along $\Delta^k \times N$.
	
	If we examine the proof of \autoref{extend_inward_pointing_from_partial_boundary}, this extension is constructed locally and then patched together with a partition of unity. Therefore we can insist on using the open set $\Delta^k \times U$ and the standard vector field on this open set, and ask for every other one of the open sets to be disjoint from some smaller neighborhood $\Delta^k \times U_1$ whose closure is contained in $\Delta^k \times U$. This guarantees that the extension is standard on $\Delta^k \times U_1$.
	
	Further examination shows that the condition of being zero in the $\Delta^k$ coordinate is also preserved under the extension algorithm from \autoref{smooth_on_boundary}, so in each of the charts intersecting $(\partial \Delta^k \times W) \cup (\Delta^k \times L)$, the vector field in that chart has this condition. The remaining charts are either in the interior or intersect only $\Delta^k \times N$, and in those cases it is easy to pick a vector field with this condition. After patching the vector fields together, we therefore get a vector field with the desired property on all of $\Delta^k \times W$. This produces the lift in the first diagram in \eqref{mirrorhcob_eq}.
	}
	
	
	For the second diagram in \eqref{mirrorhcob_eq} the proof is the same, except that we trivialize the family as a family of {\mirror} $h$-cobordisms, and we \orange{add the condition that the vector field must} respect the mirror structure on $W$. Specifically, we want the fields that are smooth when extended to $\bar W$ by applying the $C_2$ action and then negating the vectors. Note that if the mirror structure comes from a vector field then the vector field must have this property, so the given \orange{vector field on $\partial \Delta^k \times W$ has this property, as does the standard vector field on $\Delta^k \times U$, and it is easy to pick a vector field with this property in each of the remaining charts. Since the space of such fields is preserved by linear combinations, we can extend as before to $\Delta^k \times W$} and get the desired lift.
\end{proof}

\begin{prop}\label{encased_to_mirror}
The forgetful map from encased to mirror $h$-cobordisms
\[ \mc H_{\sbt}^c(M) \to \mc H_{\sbt}^m(M) \]
is a weak equivalence of Kan complexes.
\end{prop}

\begin{proof}
	It suffices to define lifts
	\begin{equation*}
	\xymatrix{
		\partial \Delta^k \ar[d] \ar[r] & \mc H_{\sbt}^{c}(M) \ar[d] \\
		\Delta^k \ar[r] \ar@{-->}[ur] & \mc H_{\sbt}^{m}(M).
		}
	\end{equation*}
	Given a $\Delta^k$-family of mirror cobordisms with encasement data on $\partial\Delta^k$, we trivialize the family as before and then use \autoref{rh_contractible} to extend the encasement over the rest of $\Delta^k$.
\end{proof}

\begin{defn}\label{codim0def}
Suppose that $e\colon M\hookrightarrow M'$ is a codimension 0 embedding. We define $W'$, a mirror $h$-cobordism on $M'$, by taking the double $\bar{W'}$ to be the extension of $\bar W$ from $M$ to $M'$ by the trivial cobordism. Here are the details. Topologically we take the pushout of
\[ \overline{(M' \setminus M)}\times [-1,1] \leftarrow \partial M\times [-1,1] \to \bar W. \]
To specify a smooth structure on this that restricts to the given smooth structures on $\overline{(M' \setminus M)}\times [-1,1]$ and $\bar W$, we use the lower collar of $W$. The latter gives a definite way of identifying a neighborhood of $\partial M\times [-1,1]$ in $W$ with a neighborhood of $\partial M\times [-1,1]$ in $M\times [-1,1]$ and therefore a way of identifying a neighborhood of $\overline{(M' \setminus M)}\times [-1,1]$ in $\bar{W'}$ with a neighborhood of $\overline{(M' \setminus M)}\times [-1,1]$ in $M' \times [-1,1]$.

The \rhs on $\bar W$ extend to $\bar{W'}$ in the obvious way, and are smooth.
\end{defn}

The same applies to families, so that a codimension zero embedding $M\to M'$ yields a map $\mc H_{\sbt}^c(M) \to \mc H_{\sbt}^c(M')$.

The following lemma establishes that the homotopy type of the $h$-cobordism space of a manifold with corners is not changed by rounding the corners. Recall the notion of \ourdefn{trimming} from \autoref{sec:trimmings_collars}.

\begin{lem}\label{equivalent_to_boundaries_defn}
	If $M'$ is a trimming of $M$ then the map $\mc H_{\sbt}^c(M') \to \mc H_{\sbt}^c(M)$ induced by the embedding of $M'$ in $M$ is a weak equivalence.
\end{lem}

\begin{proof}
	In light of \autoref{space_of_mirrorhcob} and \autoref{encased_to_mirror}, it suffices to prove the same for the ordinary $h$-cobordism spaces $\mc H_{\sbt}(M') \to \mc H_{\sbt}(M)$. We show that any diagram
	\begin{equation}\label{equiv2}
	\xymatrix{
		\partial \Delta^k \ar[d] \ar[r] & \mc H_{\sbt}(M') \ar[d] \\
		\Delta^k \ar[r] \ar@{-->}[ur] & \mc H_{\sbt}(M)
		}
	\end{equation}
	admits a lift after modifying the horizontal maps by a homotopy of commuting squares. (A strict lift may not exist, because the map is not a Kan fibration.)
	
	\orange{We observe that $\mc H_{\sbt}(M')$ is homotopy equivalent to the simplicial set of $h$-cobordisms on $M$ in which the germ of the lower collar is required to be defined on all of $\overline{(M \setminus M')} \times I$. Along this equivalence, the map to $\mc H_{\sbt}(M)$ is just the forgetful map that allows the collar to be defined on a smaller set. Therefore the above diagram gives} a trivial family of $h$-cobordisms $W \times \Delta^k$ over $M$ such that for every point in $\partial\Delta^k$, \orange{the domain of the lower collar contains $\overline{(M \setminus M')} \times I$}. In the rest of $\Delta^k$ the lower collars lack this condition, \orange{but their domains still contain $U_0 \times I$} for some fixed open set $U_0$ containing $\partial M$.
	
	Pick a \orange{vector field on $M$ that is inward-pointing along $\partial M$.} By flowing along this field, we can produce a homotopy of diffeomorphisms of $M$, from the identity diffeomorphism, to one that sends $\overline{(M \setminus M')}$ into $U_0$. Furthermore, throughout the homotopy, $\overline{(M \setminus M')}$ is always sent into itself.
	
	Take the lower collar germs for the entire family over $\Delta^k$ and pre-compose by this homotopy (times the identity of $I$). Those germs that were \orange{defined on $\overline{(M \setminus M')} \times I$} continue to be so, and all of the remaining germs become \orange{defined on $\overline{(M \setminus M')} \times I$} by the end. This produces the desired homotopy of the square \eqref{equiv2} to one in which the lift exists. (We do not change the cobordisms, only their lower collars!)
\end{proof}

\section{Stabilization of $h$-cobordisms}\label{sec:stab}


In this section we describe how a smooth embedding $M \to M'$ determines a map of $h$-cobordism spaces $\mathcal H^c(M)\to \mathcal H^c(M')$. In the case of a codimension zero embedding this was done in \autoref{codim0def}. In the general case it is a two-step process: first use a cobordism over $M$ to make a cobordism over the total space of the normal disc bundle $D(\nu)$ of $M$ in $M'$, and then extend along the codimension 0 embedding $D(\nu) \to M'$ as before.

Although $\nu$ is a vector bundle, we will need to weaken its structure to something called a ``round bundle'' first. This is not necessary for defining the stabilization, but it becomes essential when we stabilize multiple times and compare the results. The structure of a vector bundle is too rigid---composites of vector bundles are not naturally vector bundles, and this creates an issue when composing tubular neighborhoods of successive embeddings.

\subsection{Round diffeomorphisms}\label{sec:round}

A disc bundle over a disc bundle over $M$ is not a disc bundle over $M$ in a natural way. The difficulty is not just that the fiber is a product of discs instead of a disk; the structure group is also wrong. 

\begin{defn}
	Let $V$ be an inner product space. A diffeomorphism $\rho\colon D(V) \to D(V)$ is \ourdefn{round} if $|\rho(v)| = |v|$ for all $v \in D(V)$. The round diffeomorphisms form a topological group $R(V)$ with the $C^\infty$ topology.	\end{defn}

	
	\orange{
	\begin{defn}A  \ourdefn{round bundle} is a smooth fiber bundle with fiber $D(V)$ and structure group $R(V)$. A \emph{$G$-equivariant round bundle} is a smooth equivariant fiber bundle (\autoref{fiber_bundle}) with structure group $R(V)$. In particular, each $G_x$-equivariant fiber is $D(V)$ with an action of $G_x$ through round diffeomorphisms.
	\end{defn} 
	}
	

For a good example of an element of $R(V)$ that does not belong to $O(V)$, suppose that $V = V_1 \oplus V_2$, $\phi$ is an element of $O(V_1)$, and $\gamma_{v_1}$ is a family of elements of $O(V_2)$ depending smoothly on $v_1 \in V_1$. Then
\[ (v_1,v_2) \mapsto (\phi(v_1),\gamma_{v_1}(v_2)) \]
belongs to $R(V)$ but not in general to $O(V)$.
 
\begin{lem}\label{OequalsR}
	The inclusion $O(V) \to R(V)$ is a homotopy equivalence.
\end{lem}

\begin{proof}For any $\rho\in R(V)$ the derivative $D_0\rho$ of $\rho$ at the origin belongs to $O(V)$. To see this, observe that
	\[ \frac{|\rho(v) - D_0\rho(v)|}{|v|} \to 0 \quad \textup{ as } \quad v \to 0, \]
	which implies that
	\[  \frac{|\rho(v)|}{|v|} - \frac{|D_0\rho(v)|}{|v|} \to 0 \quad \textup{ as } \quad v \to 0. \]
	Since $\frac{|\rho(v)|}{|v|} = 1$ for all $v$, this means that $\frac{|D_0\rho(v)|}{|v|} \to 1$. But this last quantity is constant along rays, so it must be identically equal to $1$, so that $|D_0\rho(v)| = |v|$ for all $v$. A linear map that preserves distance to the origin must belong to $O(V)$.
	
	We can now make a deformation retraction from $R(V)$ to $O(V)$, using the homotopy 
	\[ (\rho,t) \mapsto \frac{1}{t}\cdot\rho(tv). \]
	At time $t=1$ this is equal to $\rho(v)$. The limit as $t\to  0$ is $(D_0\rho)(v)$.
\end{proof}

\begin{prop}\label{equivalent_to_vector_bundle}
	Every equivariant smooth round bundle on $M$ is smoothly isomorphic (as an equivariant round bundle) to the unit disc bundle of an equivariant Euclidean vector bundle. \orange{Furthermore, this isomorphism can be chosen to agree with a given isomorphism to an equivariant Euclidean vector bundle on any closed partial boundary $A \subseteq \partial M$.}
\end{prop}

\begin{proof}
We deduce this not from \autoref{OequalsR} but from its proof. (One issue is smooth versus topological isomorphism of bundles. The other is that when nontrivial $G$-action on the base of the bundle is allowed then equivariant bundles do not simply correspond to bundles with a certain structure group.) 

We first prove the statement for trivial $G$ \orange{and $A = \emptyset$}. Given a round bundle $E_1\to M$, the homotopy of \autoref{OequalsR} defines a smooth bundle $E \to M\times I$, which at one endpoint is the original bundle $E_1 \to M \times \{1\}$, and at the other end is a smooth round bundle $E_0 \to M \times \{0\}$ whose structure group is $O(V)$, so that it arises from a vector bundle.

To be more specific, we present $E_1$ by preferred local trivializations $U\times D(V)$ and clutching functions $\phi(-)\colon U \cap U' \to R(V)$. We then define $E$ by taking the spaces $U \times I \times D(V)$, and gluing them together along the clutching functions $(U \times I) \cap (U' \times I) \to R(V)$ defined by $\frac{1}{t}\cdot\phi(t -)$. \orange{These still satisfy the cocycle condition because $\frac{1}{t}\cdot \gamma\left(t \cdot \frac{1}{t}\cdot \phi(t -)\right) = \frac{1}{t}\cdot \gamma\phi(t -)$.} When $t = 0$, these clutching functions land in $O(V)$, as desired.

We need to ensure that $E$ can be trivialized in the $I$-direction as a round bundle, so that $E_0 \cong E_1$ as round bundles. We arrange that by choosing a lift to $E$ of the standard vector field in $M\times I$ pointing in the $I$-direction. The lift should be such that in any preferred local trivialization $U\times I \times D(V)$, the component of the vector field along $D(V)$ has no radial component. Note that this condition is preserved by the clutching functions, since the clutching functions take values in the round diffeomorphism group $R(V)$. We may therefore define such fields in each trivialization $U\times I \times D(V)$ separately, and patch them together by a partition of unity.

When flowing along such a vector field, each integral curve maintains a constant distance to the origin in $D(V)$. Thus the flow is through round diffeomorphisms. This proves that we can trivialize $E$ in the $I$-direction as a round bundle, so that $E_0 \cong E_1$ as round bundles.

In the presence of a $G$-action preserving the round structure, this proof can be carried out equivariantly. We define the $G$-action on $E$ the same way we define the clutching functions, by taking the existing action of each element $g \in G$ on $E_1$ and extending it to $E$ by the map $\frac{1}{t}\cdot g(t -)$. When $t = 0$, this is a linear action on each fiber, making $E_0$ into the disc bundle of an equivariant Euclidean vector bundle. The vector fields in $E$ with no radial component are preserved by this action, so we construct such a field non-equivariantly as before, then average it over $G$ to construct such a field that is $G$-invariant. Flowing along this $G$-invariant field, we get an isomorphism of equivariant round bundles $E_0 \cong E_1$.

\orange{Finally, if $(E_1)|_A$ is already equivariantly identified with another Euclidean vector bundle $Y \to A$, then since the operation $\frac{1}{t}\cdot\phi(t -)$ does not change linear maps, the bundle $E|_{A \times I}$ is identified with $Y \times I$. This gives a canonical choice of equivariant vector field along $E|_{A \times I}$. If we can choose our vector field in $E$ to extend this given one, then our equivariant smooth round bundle isomorphism $E_0 \cong E_1$ will agree with the given equivariant smooth vector bundle isomorphism already given on $(E_1)|_A$.

To show the extension exists locally, in each local trivialization, we are given a vector field on the space
\[ [0,\infty)^a \times [0,\infty)^{k-a} \times \R^{n-k} \times I \times D(V) \]
defined along a subset of the partial boundary
\[ O \cap (\partial([0,\infty)^a) \times [0,\infty)^{k-a} \times \R^{n-k} \times I \times D(V)) \]
where $O$ is an open set in the larger manifold. Moreover, this vector field is zero in every coordinate except for $I$, where it is 1, and $D(V)$, where it has no radial component. Examining the proof of \autoref{smooth_on_boundary_tame}, when we extend this vector field to all of $O$, the resulting vector field continues to have these properties, since they are preserved by taking linear combinations. This shows that extensions with the required property exist locally; by patching them together by a partition of unity and then averaging them, the required equivariant extension exists globally as well.} 
\end{proof}

\begin{cor}\label{conjugate}
For a finite group $G$, every homomorphism $G \to R(V)$ is conjugate to a homomorphism $G \to O(V)$.
\end{cor}

\begin{rem}\label{frechet}
In the non-equivariant case, \autoref{equivalent_to_vector_bundle} would follow directly from \autoref{OequalsR} using the main result of \cite{christoph}, if we knew that the structure group of round diffeomorphisms forms a Frechet Lie group. This is known for $\Diff(D^n)$, but seems more difficult to show for the subgroup of round diffeomorphisms. We leave this as an open question which may be interesting in its own right. 
\end{rem}

The point of round bundles is that they allow us to compose tubular neighborhoods of successive embeddings without selecting additional trivialization data. Given round bundles $E \to A $ and $A \to B$ with fibers $D(V_2)$ and $D(V_1)$ respectively, the composite bundle has fiber $D(V_1) \times D(V_2)$. We define the \ourdefn{round composite} of the bundles by starting with $E \to B$ and then restricting to the subset of $E$ whose fiber over $b \in B$ is the subset of $D(V_1) \times D(V_2)$  corresponding to $D(V_1 \times V_2)$. In other words we restrict attention to those points in $E$ such that $s^2+t^2\le 1$ where $s$ is the norm in the fiber of $E \to A$ and $t$ is the norm of the image in the fiber of $A \to B$. The round composite is again a round bundle with fiber $D(V_1 \times V_2)$. What makes this well-defined independent of local trivialization is that a twisted product of round diffeomorphisms
\[ \xymatrix @R=0.5em{
	D(V_1) \times D(V_2) \ar[r] & D(V_1) \times D(V_2) \\
	(v_1,v_2) \ar@{|->}[r] & (\phi(v_1),\gamma_{v_1}(v_2))
} \]
always restricts to a round diffeomorphism of $D(V_1 \times V_2)$.

We need a variant of \autoref{tubular_nbhd_thm} for round bundles. A \ourdefn{round tubular neighborhood} of $i\colon M \to N$ is a smooth $G$-invariant codimension 0 submanifold $D \subseteq N$ containing $M$, a smooth equivariant map $p\colon D \to M$ that is a left inverse of the inclusion of $M$, and the structure of an equivariant round bundle on $p$. These are considered up to the appropriate equivalence relation. In view of \autoref{equivalent_to_vector_bundle}, \orange{there is a bijection between round tubular neighborhoods and equivalence classes} of vector bundle tubular neighborhoods in the sense of \autoref{tubular_nbhd_space}, where two such are identified if there is a round isomorphism of the disc bundles commuting with the embedding into $N$.

\begin{thm}[Tubular Neighborhood Theorem, round bundle version]\label{round_tubular_nbhd_thm} Assume $M$ is compact. For every family of equivariant embeddings $i\colon M \times \Delta^k \to N$ landing in the interior of $N$, any family of round tubular neighborhoods on $\partial \Delta^k$ that land in the interior of $N$ can be extended to $\Delta^k$.
\end{thm}	

\begin{proof}
	By \autoref{equivalent_to_vector_bundle}, \orange{we can refine the given family of round tubular neighborhoods on $\partial \Delta^k$ to a family of vector tubular neighborhoods on each face of $\Delta^k$ separately. Working one face at a time, we can choose these refinements to agree with the refinement on the previous faces, since their intersection with the new face $\Delta^{k-1}$ is a closed partial boundary of $\Delta^{k-1}$. Then, by \autoref{tubular_nbhd_thm}, this family of vector tubular neighborhoods on $\partial \Delta^k$} can be extended over the interior. But that extension is, by neglect of structure, an extension as a family of round tubular neighborhoods.
\end{proof}

\subsection{The stabilization and its smooth structure}

Let $M$ be a compact $G$-manifold with corners, let $W$ be a {\mirror} $h$-cobordism on $M$, and suppose that $W$ is equipped with \arhperiod We follow the notation of the previous section, letting $N$ be the top of the cobordism and denoting the double by $\bar W$.

We will make an encased $h$-cobordism $St^\nu(W)$ on the manifold $D(\nu)$, where $\nu$ is the normal bundle of an embedding $e:M\to M'$ and $D(\nu)\to M$ is its closed unit disc bundle. We define $St^\nu(W)$ by defining its double $St^\nu(\bar{W})$.

For the next definition, 
$p\colon D(\nu) \to M$ may be any equivariant round bundle. Let $D(\nu \times \R) \to M$ be the round composite of $p$ with the trivial bundle
\[ D(\nu) \times [-1,1] = D(\nu) \times D(\R) \to D(\nu). \]
In addition to the projection $D(\nu \times \R) \to M$ this has a map to $[0,1]$, namely the norm in the fiber over $M$. On the other hand, $W$ also has a map to $M\times [0,1]$, namely the encasing function $\rho$ to $M \times [-1,0]$ flipped upside down. We can therefore form the fiber product
\[ W \times_{M \times [0,1]} D(\nu \times \R). \]

This space will be the essential part of $St^\nu ( \bar{W})$.

For an inner product bundle $\xi$, let $S(\xi)$ be the total space of the unit sphere bundle. Let $D_0(\xi)\subset D(\xi)$ be the zero section, homeomorphic to the base of the bundle. Let $D_0'(\xi)$ be the complement of $D_0(\xi)$ in $D(\xi)$.

Consider the map 
\[ W \times_M S(\nu \times \R) \to W \times_{M \times [0,1]} D(\nu \times \R)\]
 given by $(w,u)\mapsto (w,|h(w)|u)$. This maps the open subset $(W\backslash N) \times_M S(\nu \times \R)$ homeomorphically to the open subset $(W\backslash N) \times_{M \times (0,1]} D_0'(\nu \times \R)$. It maps the complementary closed subset $N \times_{M} S(\nu\times \R)$ onto the closed set $N\times_{M\times \{0\}} D_0(\nu\times \R)$, and when we identify this last set with $N$ the map becomes the projection $(n,u)\mapsto n$.

The fiber product $W \times_{M \times [0,1]} D(\nu \times \R)$ can therefore be rewritten as the colimit
\[ \xymatrix{
	N \times_M S(\nu \times \R) \ar[d] \ar[r] & N \\
	W \times_M S(\nu \times \R). &
	} \]

Let $C(\nu \times \R)$ denote the complement of the interior of $D(\nu\times \R)$ inside $D(\nu) \times [-1,1]$, so that $D(\nu) \times [-1,1]$ is the union of $C(\nu\times \R)$ and $D(\nu\times \R)$ along $S(\nu\times \R)$.

\begin{defn}\label{hcob_stab}
	As a topological space with $G$-action, $\St^\nu(\bar W)$ is the colimit of the diagram
\[ \xymatrix{
	& N \times_M S(\nu \times \R) \ar[d] \ar[r] & N \\
	S(\nu \times \R) \ar[d] \ar[r]^-{} & W \times_M S(\nu \times \R) & \\
	C(\nu \times \R), & &
	} \]
	or more simply the colimit of
	\[ \xymatrix{
	S(\nu \times \R) \ar[d] \ar[r]^-{} & W \times_{M \times [0,1]} D(\nu \times \R) \\
	C(\nu \times \R). &
	} \]
\noindent	The middle horizontal map is induced by the inclusion of $M\cong M \times \{-1\}$ into $W$. 

The group $C_2$ acts on $\St^\nu(\bar W)$ by flipping the $\R$ direction. The set of $C_2$-fixed points is the fiber product
\[ W \times_{M \times [0,1]} D(\nu), \]
which in \autoref{stabfig}  corresponds to the center line. The whole is the double of the lower half along its top, and this lower half is called $\St^\nu(W)$.
\end{defn}

\begin{figure}[h]
	\centering
	\def\svgwidth{5.8in}
	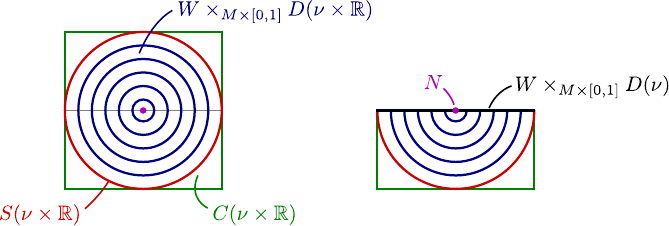
	\vspace{-.5em}
	\caption{The doubled stabilization $\St^\nu(\bar W)$ and its lower half $\St^\nu(W)$.}\label{stabfig}
\end{figure}

Before providing a smooth structure, we verify that we have made a topological equivariant $h$-cobordism.

\begin{lem}\label{is_equiv_hcob}
	If $W$ is $G$-equivariantly homotopy equivalent to its top and bottom, then so is $\St^\nu(W)$.
\end{lem}

\begin{proof}
	Pick a presentation of $\nu$ as a vector bundle, not just a round bundle. Then deformation retract $C(\nu \times \R)$ in the $\R$-direction back to $S(\nu \times \R)$, then shrink the rest to the center using the $G$-equivariant deformation retraction of $W$ to $N$. Throughout the homotopy we vary the corresponding point in $D(\nu \times \R)$ by scaling in the radial direction (for which we need the vector bundle structure). This shows that the inclusion of $N$ into $\St^\nu(W)$ is a $G$-equivariant homotopy equivalence.
	
	The same homotopy shows that $N$ is also equivalent to the top of $\St^\nu(W)$. Therefore $\St^\nu(W)$ is equivariantly homotopy equivalent to its top. The zig-zag of subspaces of $\St^\nu(W)$
	\[ \xymatrix{
		N \ar[r]^-\sim & W & \ar[l]_-\sim M \ar[r]^-\sim & D(\nu)
	} \]
	and equivariant homotopy equivalences now shows that the inclusion of the bottom $D(\nu)$ into $\St^\nu(W)$ is an equivalence as well.
\end{proof}
	
\begin{notn}\label{regions}
The following language will help us talk about $\St^\nu(\bar W)$ more clearly.
\begin{itemize}
	\item The \emph{trivial region} of $\St^\nu(\bar W)$, pictured in green in \autoref{stabfig}, is $C(\nu \times \R)$.
	\item The \emph{nontrivial region} of $\St^\nu(\bar W)$, pictured in blue in \autoref{stabfig}, is
	\[ W\times_{M\times [0,1]} D(\nu\times \R). \]
	\item The \emph{frontier} of $\St^\nu(\bar W)$, pictured in red in \autoref{stabfig}, is the intersection
\[ S(\nu \times \R)= M\times_{M\times \{1\}} S(\nu\times \R) \]
	of the trivial and nontrivial regions.
	\item The \emph{cone locus} of $\St^\nu(\bar W)$, pictured as the central point in \autoref{stabfig}, is the set
\[  N\times_{M\times \{0\}} D_0(\nu\times \R) \cong N \]
	inside the nontrivial region.
\end{itemize}
\end{notn}

\begin{defn}\label{smoothstr}
	We define the smooth structure on the double $\St^\nu(\bar W)$ as follows. 
	
The trivial region $C(\nu \times \R)$ has an obvious structure as a smooth manifold with corners near each point not on the frontier $S(\nu \times \R)$. Moreover, the map $D(\nu\times \R)\to M\times [0,1]$ is a submersion over the interior of $[0,1]$, giving the fiber product $W \times_{M \times [0,1]} D(\nu \times \R)$ an induced smooth structure away from 0 and 1. We still need to define the smooth structure near the frontier and near the cone locus.

To define a smooth structure near the frontier $S(\nu \times \R)$, we use the lower collar of $W$ to identify a neighborhood of the frontier in the nontrivial region with a neighborhood of $S(\nu\times \R)$ in $D(\nu\times \R)$, and thus to identify a neighborhood of the trivial region in $St^\nu(\bar W)$ with a neighborhood of $C(\nu \times \R)$ in the smooth manifold with corners $D(\nu) \times [-1,1]$.

It remains to define a smooth structure at the cone locus $N$. For each $n \in N$, pick a contractible neighborhood $U\to N$ of $n$ and a collar
\[ b: U \times (-\epsilon,0] \to W, \]
such that the double
\[ b: U \times (-\epsilon,\epsilon) \to \bar W \]
is smooth. Choose it such that $h\circ b$ is projection on the second factor; this is possible because the height function has full rank. It should also be equivariant with respect to the isotropy group $G_n\subset G$. 

Choose also a smooth trivialization of the round bundle $D(\nu) \to M$ near $r(n) \in M$, compatible with the action of $G_{r(w)}$. Pulling this back to $\bar W$ and adding on $\R$ gives a trivialization of the round bundle $W \times_M D(\nu \times \R) \to W$ over the contractible set $U \times (-\epsilon,0]$.

This gives a neighborhood of $(n,0) \in W \times_{M \times [0,1]} D(\nu \times \R)$ of the form
\[ (U\times [0,\epsilon)) \times_{M \times [0,1]} (M \times D(V\times \R)). \]
Since the height function is full rank along $U \times \{0\}$, it induces a local diffeomorphism $U \times [0,\epsilon) \to U \times [0,1]$ defined near 0, so up to homeomorphism this neighborhood is identified with an open subset of the product
\[ U \times D(V \times \R). \]
Therefore, we define the smooth chart on this neighborhood to be the projection map to $U \times D(V \times \R)$.

\end{defn}

\begin{prop}\label{stabsmooth}
The smooth structure given by each chart in \autoref{smoothstr} is independent of the choice of bicollar for $N$ and trivialization of $\nu$ as a round bundle. The charts therefore give a well-defined smooth structure on $\St^\nu(\bar W)$, \orange{making it into a manifold with corners.}
\end{prop}

In other words, the smooth structure on the stabilization comes from the {\mirror} structure of $W$ (the extension of its smooth structure to $\bar W$), not the particular choice of smooth bicollar used to present this {\mirror} structure. (In contrast, the smooth structure at the frontier depends on the choice of lower collar for the cobordism $W$, which is fixed in advance.)

\begin{proof} Suppose $b_1, b_2\colon U \times (-\epsilon,\epsilon) \to \bar W$ are two $(G \times C_2)$-equivariant bicollars of $N \to \bar W$ defined near $n_0 \in N$, both such that $h \circ b_i\colon U \times (-\epsilon,\epsilon) \to [-1,1]$ is the projection map on the second coordinate. Assume also that for each one we choose a trivialization of $\nu \to M$ as a round bundle near $r(n_0)$.
The transition between these two charts is a partially-defined map
\[ \xymatrix @R=0.5em{
	(U\times [0,\epsilon)) \times_{[0,1]} D(V\times \R) \ar[r] & (U\times [0,\epsilon)) \times_{[0,1]} D(V\times \R) \\
	(n,t,v) \ar@{|->}[r] & (n',t',v')
} \]
that is a homeomorphism on a neighborhood of $(n_0,0,0)$. By the definitions of the charts, we get $b_1(n,t)=b_2(n',t')$, and $v'=\ell(n,t)(v)$ for a round diffeomorphism $\ell(n,t)\in R(V \times \R)$ that is smooth and even as a function of $(n,t) \in U \times (-\epsilon, \epsilon)$. (It is a function of the projection to $M$, which is $r \circ b_1$, which is even.)

By the definition of the fiber products we have
\[ |v| = |h \circ b_1(n,t)| = |h \circ b_2(n',t')| = |v'|. \]
By the extra assumption that $h \circ b_1(n,t) = t$, we have $|t| = |v|$.

By our conventions, the formula $b_2^{-1} \circ b_1$ defines the germ of a diffeomorphism of $U \times (-\epsilon,\epsilon)$, whose $U$ coordinate is even. Applying \autoref{even_function_of_t2} to the first coordinates rewrites it as
\[ (b_2^{-1} \circ b_1)_1(n,t) = g_1(n,t^2), \]
for a smooth function $g_1\colon U \times [0,\epsilon) \to U$. Therefore $n' = g_1(n,t^2)= g_1(n,|v|^2)$.

Similarly, since $\ell$ is even in $t$, we have by \autoref{even_function_of_t2}
\[ \ell(n,t)=g_2(n, t^2)\]
for a smooth function $g_2\colon U \times [0,\epsilon) \to R(V \times \R)$. Therefore $v' = g_2(n,t^2)(v) = g_2(n,|v|^2)(v)$.

Putting this all together, the transition between the two charts is the map
\[ \xymatrix @R=0.5em{
	U\times D(V\times \R) \ar[r] & U\times D(V\times \R) \\
	(n,v) \ar@{|->}[r] & \left( \ g_1(n, |v|^2) \ , \ g_2(n, |v|^2)(v) \ \right),
} \]
which is smooth since $|v|^2$, $g_1$, and $g_2$ are smooth. Furthermore the map is full rank as $\ell(n,t) = g_2(n,|v|^2)$ is a diffeomorphism that to first order does not depend on $v$, whereas $(b_2^{-1} \circ b_1)_1(n,t)$ extends to a diffeomorphism of $U \times (-\epsilon,\epsilon)$, so it is full rank in the $n$ direction.
\end{proof}

\begin{defn}\label{st_rho}
	The encasing
\[ \St^\nu \rho\colon\  \St^\nu(\bar W) \to D(\nu) \times [-1,1] \]
is defined by projecting away from $W$ to the pushout
\[ \xymatrix{
	S(\nu \times \R) \ar[d] \ar[r]^-{} & D(\nu \times \R) \\
	C(\nu \times \R), &
} \]
which is $D(\nu) \times [-1,1]$.
\end{defn}

We record the following characterization of points in the stabilization. It follows immediately from the definition, and it will be very useful later. \orange{Recall that the nontrivial region is defined in \autoref{regions}.}

\begin{lem}\label{indexing_coordinates}
	Each point in $\St^\nu(\bar W)$ is determined uniquely by its image under $\St^\nu\rho$, together with its \orange{projection to} $W$ when in the nontrivial region.
\end{lem}

\begin{lem}\label{rhosmooth}
	The \rh $\St^\nu\rho$ is smooth.
\end{lem}

\begin{proof}
	On the trivial region and the chart that includes the frontier, this becomes the identity map of a neighborhood of the trivial region, so it is smooth. At the cone locus this becomes the projection
	\[ U \times D(V \times \R) \to M \times D(V \times \R) \]
	by applying $r$ to the first coordinate, which is smooth.
\end{proof}

In order to define a collar for $\St^\nu(\bar W)$, we first show that $\St^\nu\rho$ is a diffeomorphism onto its image, when restricted to some neighborhood of the boundary of $\St^\nu(\bar W)$.

\begin{lem}\label{st_rho_full_rank}
	$\St^\nu\rho$ is full rank on an open set containing the trivial region, and at any point in the nontrivial region for which $\rho$ is full rank at the associated $w \in W$.
\end{lem}

\begin{proof}
	For the trivial region this is obvious because the inclusion $C(\nu \times \R) \subseteq D(\nu) \times [-1,1]$ is full rank. For the nontrivial region away from the cone locus, this is also obvious because it is the pullback of the full rank map $\rho$. At the cone locus, it is full rank iff $r\colon N \to M$ is full rank, which is true iff $\rho\colon W \to M \times [-1,0]$ is full rank because we are already assuming that $h$ is full rank along $N$.
\end{proof}

Now suppose $\rho|_O$ is a diffeomorphism, from some open $(G \times C_2)$-invariant subset $O \subseteq \bar W$ containing $\partial \bar W$, to an open set $V \subseteq M \times [-1,1]$. Let $O'$ be the union of the trivial region and the points in the nontrivial region associated to $O$.

\begin{cor}\label{st_rho_diffeo}
	Under these assumptions, $(\St^\nu\rho)|_{O'}$ is a diffeomorphism onto its image $V' \subseteq D(\nu) \times [-1,1]$.
\end{cor}
 
\begin{proof}
	Follows immediately from \autoref{indexing_coordinates} and \autoref{st_rho_full_rank}.
\end{proof}

To give the collar for $\St^\nu(\bar W)$, we take $O \subseteq \bar W$ to be the image of a germ of a collar, an open subset containing the top, bottom, and sides of $\bar W$. By \autoref{st_rho_diffeo}, since $\rho$ is a diffeomorphism on $O$, $(\St^\nu\rho)|_{O'}$ is a diffeomorphism to an open subset $V' \subseteq D(\nu) \times [-1,1]$. It is easy to check that $V'$ contains the top, bottom, and sides (including the side $D(\nu)|_{\partial M} \times [-1,1]$), so we use this as our germ of a collar for $\St^\nu(\bar W)$.

\orange{
Finally, we check that the height function and retraction that we have constructed have the additional properties required by \autoref{defn_encasing_function}.
\begin{lem}\label{rhoallowed}
	The new height function satisfies
	\begin{itemize}
		\item $(\St^\nu h)^{-1}(-1)$ is the bottom $D(\nu)$,
		\item $(\St^\nu h)^{-1}(0) =  W \times_{M \times [0,1]} D(\nu \times \{0\})$,
		\item $D(\St^\nu h)$ is full-rank along $W \times_{M \times [0,1]} D(\nu \times \{0\})$,
	\end{itemize}
	 and the new retraction $\St^\nu r$ sends the complement of the sides of $\St^\nu(\bar W)$ into the interior of $D(\nu)$.
\end{lem}

\begin{proof}
	The height function arises from the projection of $\nu \times \R$ to the $\R$ coordinate, making the first two claims straightforward to check. To see it is full-rank along the preimage of $0 \in \R$, we know this on a neighborhood of the trivial region by \autoref{st_rho_full_rank}. At a point $(w,v,0)$ in the nontrivial region $W \times_{M \times [0,1]} D(\nu \times \R)$, the curve $t \mapsto (w,v\cos t,\sin t)$ projects along the height function to the curve $t \mapsto \sin t$, and therefore the height function is full rank. At the cone locus, the local chart for $\St^\nu(\bar W)$ is $U \times D(\nu \times \R)$, and the height function is the projection to $\R$, which is clearly full rank.
	
	The condition on the retraction is true on a neighborhood of the trivial region by \autoref{st_rho_full_rank}, and on the nontrivial region as well, when the point in $\bar W$ is in the chosen neighborhood $O$ of the boundary. The remaining points are of the form $(w,v,t) \in W \times_{M \times [0,1]} D(\nu \times \R)$, where $w \not\in O$ and $|(v,t)| < 1$, and the retract sends this to the second coordinate $v \in D(\nu)$. Since $w \not\in O$, $r(w) \not\in \partial M$, which implies by the fiber product condition that $v \not\in D(\nu)|_{\partial M}$. Furthermore the fact that $|(v,t)| < 1$ implies that $|v| < 1$, so that $v \not\in S(\nu)$. All together this proves that $(\St^\nu r)(w,v,t) = v$ is in the interior of $D(\nu)$, as required.
\end{proof}
}

This concludes the construction of $\St^\nu(\bar W)$ as an encased mirror $h$-cobordism. Combining this with codimension 0 embeddings allows us to stabilize along an arbitrary embedding $M \to M'$. 

\begin{defn}\label{fulldef}
Let $W$ be a mirror $h$-cobordism on $M$, $M\to M'$ an embedding with normal bundle $\nu\to M$, and $D(\nu) \to M'$ a tubular neighborhood. We define the double $\St^e(\bar W)$ of the stabilization $\St^e(W)$ by first forming $\St^\nu(\bar W)$ as in \autoref{hcob_stab} and then applying \autoref{codim0def} to the codimension 0 embedding $D(\nu)\to M'$.
\end{defn}

\begin{figure}[h]
	\centering
	\def\svgwidth{4.2in}
	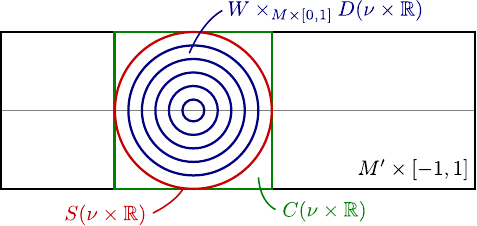
	\vspace{-.5em}
	\caption{The stabilization $\St^e(\bar W)$ for an embedding $e\colon M\to M'$.}\label{fig:polar_stab_2}
\end{figure}

Note that \autoref{indexing_coordinates} applies also to this extended stabilization. Let $C$ be the complement of the image of the interior of $D(\nu)$. The trivial region now means the union of $C \times [-1,1]$ and $C(\nu \times \R)$. The image of the nontrivial region under $\St^e \rho$ is still $D(\nu \times \R) \subseteq M' \times [-1,1]$, while the image of the trivial region under $\St^e \rho$ is the complement of the interior of this disc bundle.

We will also need the following functoriality property of stabilization with respect to isomorphisms, which is easy to check from the definitions.

\begin{lem}\label{isoonstab}
Each encased diffeomorphism $f\colon W\xrightarrow{\cong} W'$ of $h$-cobordisms over $M$ (\autoref{encased_diffeo}) induces an encased diffeomorphism $$\St^e(f)\colon \St^e(\bar W)\xrightarrow{\cong} \St^e (\bar W')$$
of $h$-cobordisms over $M'$.
\end{lem}

\subsection{Iterated stabilization}\label{composing_stabilizations}
Now assume that we have two composable embeddings of compact $G$-manifolds with corners
\[ M_0 \to M_1 \to M_2 \]
extending to codimension 0 embeddings
\[ e_{01}\colon D(\nu_{01}) \to M_1, \quad e_{12}\colon D(\nu_{12}) \to M_2, \quad e_{02}\colon D(\nu_{02}) \to M_2 \]
where the various $D(\nu_{ij})$ are smooth $G$-equivariant round bundles
\[ D(\nu_{01}) \to M_0, \quad D(\nu_{12}) \to M_1, \quad D(\nu_{02}) \to M_0 \]
and the embedding $e_{02}$ identifies $D(\nu_{02})$ with the round composite of
\[ \xymatrix{ D(\nu_{12})|_{D(\nu_{01})} \ar[r] & D(\nu_{01}) \ar[r] & M_0. } \]
In this setup we will define a bijection $\St^{e_{12}}\St^{e_{01}}(\bar W) \cong \St^{e_{02}}(\bar W)$ and prove it is a diffeomorphism.

\begin{prop}\label{canonical_homeo}
	There is a canonical natural homeomorphism \[ \St^{e_{12}}\St^{e_{01}}(\bar W) \cong \St^{e_{02}}(\bar W). \]
\end{prop}

\orange{We take this homeomorphism to be the unique map that commutes with the encasing function to $M_2 \times [-1,1]$, and that also commutes with the projection to $W$ when in the nontrivial region, see \autoref{regions}.}
By \autoref{indexing_coordinates}, there is at most one map with this property, so we are really just claiming that such a map exists and is a homeomorphism. The canonical map is also natural with respect to isomorphisms, meaning that for each encased diffeomorphism $f\colon W\xrightarrow{\cong} W'$ the following square commutes.
	\[\xymatrix{
	\St^{e_{12}}\St^{e_{01}}(\bar W) \ar[d]_-{\St^{e_{12}}\St^{e_{01}}(f)}^-\cong \ar[r]^-{\cong} &   \St^{e_{02}}(\bar W) \ar[d]_-{\cong}^-{\St^{e_{02}}(f)}\\
	\St^{e_{12}}\St^{e_{01}}(\bar W') \ar[r]^-{\cong} &   \St^{e_{02}}(\bar W')
	}\]
	
\begin{proof}
We begin by identifying the trivial regions. The trivial region of $\St^{e_{01}}(\bar W)$ consists of all points that go to the complement of the interior of $D(\nu_{01} \times \R)$ in $M_1 \times [-1,1]$. If we apply \autoref{st_rho_diffeo} to an open neighborhood of this first trivial region, we conclude that $\St^{e_{12}}\St^{e_{01}}\rho$ is a diffeomorphism on an ``extended trivial region'' in the double stabilization $\St^{e_{12}}\St^{e_{01}}(\bar W)$, consisting of both the complement of the interior of $D(\nu_{12} \times \R)$ in $M_2 \times [-1,1]$, and those points in
\[ \St^{e_{01}}(W) \times_{M_1 \times [0,1]} D(\nu_{12} \times \R) \]
where the first coordinate in $\St^{e_{01}}(W)$ is in the first trivial region. This hits all the points in $D(\nu_{12} \times \R)$ except those in the image of
\[ D(\nu_{01} \times [0,1]) \times_{D(\nu_{01}) \times [0,1]} D(\nu_{12} \times \R)|_{D(\nu_{01})}, \]
mapping to $D(\nu_{12} \times \R) \subseteq M_2 \times [-1,1]$ by projection. The map sends $(v,|u|,u)$ to $(v,u)$, so if we regard the restricted disc bundle as a product
\[ D(\nu_{12} \times \R)|_{D(\nu_{01})} \cong D(\nu_{01}) \times D(\nu_{12} \times \R) \]
then the map hits all pairs $(v,u)$ such that $|v|^2 + |u|^2 < 1$. In other words, the extended trivial region is identified with the complement of the interior of $D(\nu_{02} \times \R)$. However this is exactly the same set of points that are identified with the trivial region in $\St^{e_{02}}(\bar W)$. We conclude there is a unique identification between the extended trivial region in $\St^{e_{12}}\St^{e_{01}}(\bar W)$ and the trivial region in $\St^{e_{02}}(\bar W)$ that commutes with the \rhsperiod

For the nontrivial regions, we pick a trivialization of $D(\nu_{12})$ along the fibers of $D(\nu_{01})$, so that we can regard $D(\nu_{02})$ as the disc inside a direct sum round bundle $D(\nu_{01} \oplus \nu_{12})$. We take the map of fiber products
\begin{equation}\label{canonical_map_nontrivial_region} 
\xymatrix @R=1.5em {
	W \times_{M_0 \times [0,1]} D(\nu_{01} \times \R) \times_{D(\nu_{01}) \times [0,1]} D(\nu_{12} \times \R)|_{D(\nu_{01})}
	\ar[d] & (w,v,t,u) \ar@{|->}[d]\\
	W \times_{M_0 \times [0,1]} D(\nu_{01} \oplus \nu_{12} \times \R) & (w,v,u).
}
\end{equation}
Here $v \in \nu_{01}$, $t \in [0,1]$, and $u \in \nu_{12} \times \R$. This map satisfies \autoref{indexing_coordinates} because it preserves the associated point in $W$, and the \rhs take each side to
\[ (v,\ u) \in D(\nu_{01}) \times D(\nu_{12} \times \R) \subseteq M_2 \times [-1,1]. \]
It is also a bijection with inverse $(w,v,u) \mapsto (w,v,|u|,u)$.

All together this gives the canonical bijection $\St^{e_{12}}\St^{e_{01}}(\bar W) \cong \St^{e_{02}}(\bar W)$. It is clearly continuous, and since the source is compact it is also a homeomorphism.
\end{proof}

We note the following associative property of the canonical homeomorphism.

\begin{lem}\label{assoc_stab}
For any composite of three embeddings with tubular neighborhoods, 
the iterated stabilization maps fit into a commutative square
\[\xymatrix{
\St^{e_{03}}\bar W \ar[r]^-\cong \ar[d]_-\cong & \St^{e_{23}}\St^{e_{02}} \bar W \ar[d]^-\cong\\
\St^{e_{13}}\St^{e_{01}} \bar W  \ar[r]^-\cong & \St^{e_{23}} \St^{e_{12}} \St^{e_{01}} \bar W.
}\]
\end{lem}
Here the bottom horizontal map is the canonical isomorphism from \autoref{canonical_homeo}, applied to the cobordism $\St^{e_{01}} \bar W $, and the right vertical map is the isomorphism from \autoref{isoonstab} induced on the stabilization $ \St^{e_{23}}$ by the canonical isomorphism  $\St^{e_{02}} \bar W\cong \St^{e_{12}} \St^{e_{01}}\bar W$.

\begin{proof}
	This follows easily from the defining property of the canonical homeomorphism of \autoref{canonical_homeo}.
\end{proof}

Lastly, we prove that the iterated stabilization map is a diffeomorphism with respect to the smooth structures we defined on the stabilizations.

\begin{thm}
	The canonical homeomorphism
	\[ \St^{e_{12}}\St^{e_{01}}(\bar W) \cong \St^{e_{02}}(\bar W) \]
	is an encased diffeomorphism.
\end{thm}

\begin{proof}
	We already know this homeomorphism commutes with the \rhsperiod The fact that it is a diffeomorphism between the extended trivial region in $\St^{e_{12}}\St^{e_{01}}(\bar W)$ and the trivial region in $\St^{e_{02}}(\bar W)$ follows because the \rhs are diffeomorphisms on those regions (and so the canonical homeomorphism is a composite of two diffeomorphisms).
	
	It remains to show that the canonical homeomorphism is smooth and full rank inside the nontrivial region, where it is given by \eqref{canonical_map_nontrivial_region}. Since the check is local and the smooth charts are defined by trivializing the bundles, without loss of generality the bundles are trivial, so the map to check is
	\begin{equation}\label{canonical_map_simplified}
\xymatrix @R=0.5em{
	W \times_{[0,1]} D(V_1 \times [0,1]) \times_{[0,1]} D(V_2 \times \R)
	\ar[r] &
	W \times_{[0,1]} D(V_1 \times V_2 \times \R)
	\\
	(w,v,t,u) \ar@{|->}[r] & (w,v,u).
}
\end{equation}
	The points in the nontrivial region can be divided into three cases.
	
\noindent  \emph{Case 1: Points  not in the cone locus of $\St^{e_{12}}$.} These are the points $(w,v,t,u)$ for which $t \neq 0$, and therefore also $u \neq 0$ and $w \not\in N$. Both fiber products on the left-hand side of \eqref{canonical_map_simplified} are being taken over $(0,1)$ here, where the map on the right is a submersion and the smooth structure on the fiber product is being inherited from the product. Smoothness is therefore clear. The inverse has formula $(w,v,u) \mapsto (w,v,|u|,u)$, which is also smooth since $|u| > 0$. Therefore the map has full rank at these points as well.

\noindent	\emph{Case 2: Points in the cone locus of $\St^{e_{12}}$ but not also in the cone locus of $\St^{e_{02}}$}. These are the points $(w,v,t,u)$ for which $t = 0$ and therefore $u = 0$, but $w \not\in N$. On the left-hand side of \eqref{canonical_map_simplified} this means the first fiber product is over $(0,1)$ but the second is at 0.
	
	We adopt the shorthand
	\[ W_0 = W \setminus (M \amalg N), \qquad D_0 = D \setminus (\{0\} \cup \partial D), \]
	\[ N_1 = W_0 \times_{(0,1)} D_0(V_1), \qquad \bar W_1 = W_0 \times_{(0,1)} D_0(V_1 \times \R). \]
	So $N_1$ is the top of $\St^{\nu_{01}}(W)$ minus the cone locus and frontier, and $\bar W_1$ is a neighborhood of $N_1$ in $\St^{\nu_{01}}(\bar W)$. We choose the bicollar for $N_1$ in $\bar W_1$ by the formula
\begin{equation*}
	\xymatrix @R=0.5em{
	W_0 \times_{(0,1)} D_0(V_1) \times (-\epsilon,\epsilon) \ar[r] & W_0 \times_{(0,1)} D_0(V_1 \times \R) \\
	(w,y,t) \ar@{|->}[r] & (w,y,t).
	}
\end{equation*}
	As interpreted strictly, this is not defined on the entire domain, but we interpret it as defined near an arbitrary point in the domain for small enough $\epsilon$. It is equivariant and full rank and therefore defines an equivariant bicollar.

	The chart from \autoref{smoothstr} that gives the smooth structure on $\St^{\nu_{12}}$ is the composite of the first two maps below. The remaining map below is the canonical homeomorphism.
	\[ \resizebox{\textwidth}{!}{$
	\xymatrix @R=1.5em @C=.7em{
		N_1\times D(V_{2}\times \R) & W_0 \times_{(0,1)} D_0(V_1)\times D(V_{2}\times \R) & (w, y, v) \\
		\Big( N_1\times [0, \epsilon) \Big) \times_{[0,1]} D(V_{2}\times \R) \ar[u]_-\cong \ar[d]^-b & \Big( W_0 \times_{(0,1)} D_0(V_1)\times [0, \epsilon) \Big) \times_{[0,1]} D(V_{2}\times \R) \ar[u]_-\cong \ar[d]^-b & (w, y, |v|, v) \ar@{|->}[u] \ar@{|->}[d] \\
		W_1 \times_{[0,1]} D(V_{2} \times \R) \ar[d] & W_0 \times_{(0,1)} D_0(V_1 \times [0,1]) \times_{[0,1]} D(V_{2} \times \R) \ar[d] & (w, y, |v|, v) \ar@{|->}[d] \\
		W_0 \times_{[0,1]} D(V_1 \times V_2 \times \R) & W_0 \times_{[0,1]} D(V_1 \times V_2 \times \R) & (w, y, v) \\
	}$}
\]
	The composite is (a restriction of) an identity map, so it is smooth and full rank.
	
\noindent \emph{Case 3: Points in the cone locus of both $\St^{e_{01}}$ and $\St^{e_{02}}$}. These are the points $(w,v,t,u) = (n_0,0,0,0)$ with $n_0 \in N$. From \autoref{smoothstr}, the smooth structure on $\St^{e_{01}}(W)$ near one such point is defined by taking a neighborhood $U$ of $n_0$ in $N$ and $(G \times C_2)$-equivariant bicollar $b\colon U \times (-\epsilon,\epsilon) \to \bar W$. The smooth chart has the formula
\begin{equation*}
	\xymatrix @R=0.5em{
	\Big( U \times [0,\epsilon) \Big) \times_{[0,1]} D(V_1 \times [0,1]) \ar[r] & U \times D(V_1 \times [0,1]) \\
	(n,s,v,t) \ar@{|->}[r] & (n,v,t).
	}
\end{equation*}

Inside this smooth chart, the top of the cobordism is identified with the subset $U \times D(V_1)$. This top has a bicollar that is the inclusion (defined sufficiently close to the center of $D(V_1)$)
\[ c\colon U \times D(V_1) \times (-\epsilon,\epsilon) \subseteq U \times D(V_1 \times \R). \]

In the following diagram, the top horizontal is the canonical homeomorphism. The left-hand column is the smooth chart on the double stabilization $\St^{e_{12}}\St^{e_{01}}(\bar W)$, obtained by taking the smooth chart on the first stabilization defined just above using $b$, and then defining the smooth structure on the second stabilization inside that chart using $c$. The right-hand column is the smooth chart in the single stabilization $\St^{e_{02}}(\bar W)$, defined using the same bicollar $b$.
\[ 
\resizebox{\textwidth}{!}{$
\xymatrix @R=1.5em{
	W \times_{[0,1]} D(V_1 \times [0,1]) \times_{[0,1]} D(V_2 \times \R)
	\ar[r] &
	W \times_{[0,1]} D(V_1 \times V_2 \times \R)
	\\
	\Big( U \times [0,\epsilon) \Big) \times_{[0,1]} D(V_1 \times [0,1]) \times_{[0,1]} D(V_2 \times \R) \ar@{-->}[r] \ar[u]_-b \ar[d]^-\cong &
	\Big( U \times [0,\epsilon) \Big) \times_{[0,1]} D(V_1 \times V_2 \times \R) \ar[u]_-b \ar[d]^-\cong
	\\
	U \times D(V_1 \times [0,1]) \times_{[0,1]} D(V_2 \times \R) \ar@{-->}[r] &
	U \times D(V_1 \times V_2 \times \R)
	\\
	\Big( U \times D(V_1) \times [0,\epsilon) \Big) \times_{[0,1]} D(V_2 \times \R) \ar[u]_-c \ar[d]^-\cong
	\\
	U \times D(V_1) \times D(V_2 \times \R) &
}
$}
\]
The top horizontal has the formula $(w,v,t,u)$ to $(w,v,u)$ as in \eqref{canonical_map_simplified}. On the second line we can define the dashed map in the same way, replacing $w$ by $(n,s)$, which makes the top square commute. On the third line we similarly send $(n,y,t,u)$ to $(n,y,u)$, making the second square commute:
\[
	\xymatrix @R=1.5em{
	(n,s,v,t,u) \ar@{|->}[d] \ar@{|->}[r] & (n,s,v,u) \ar@{|->}[d] \\
	(n,v,t,u) \ar@{|->}[r] & (n,v,u)
	}
\]
Therefore the canonical homeomorphism, in these charts, is given by the lower route of the diagram:
\[ \xymatrix @R=1.5em{
	(n,y,|u|,u) \ar@{|->}[r] & (n,y,u) \\
	(n,y,|u|,u) \ar@{|->}[u] \ar@{|->}[d] \\
	(n,y,u)
} \]
This is just the identity map (restricted to a neighborhood of zero), which is obviously smooth and full rank.
\end{proof}

\subsection{Stabilizing a family}

If $E \to \Delta^k$ is a family of encased $h$-cobordisms over $M$, then we can regard it as a cobordism over $M \times \Delta^k$. Given a round bundle $D(\nu) \to M \times \Delta^k$, we can apply the procedure of \autoref{hcob_stab} to the cobordism $E$ to produce a cobordism $\St^\nu(E)$ over $D(\nu) \times \Delta^k$. Given a $k$-simplex of tubular neighborhoods $e\colon D(\nu) \to M' \times \Delta^k$, we then extend by zero to produce a cobordism $\St^e(E)$ over $M' \times \Delta^k$.

The one difference in this case is that $E$ is not trivial over $M \times \partial \Delta^k$, only over $\partial M \times \Delta^k$. However, we only need $\St^e(E)$ to be trivial on $\partial M' \times \Delta^k$, not $M' \times \partial\Delta^k$. So the collar of $\St^e(E)$ only has to be defined near the bottom $M' \times \Delta^k \times \{0\}$ and the part of the sides corresponding to $\partial M' \times \Delta^k \times I$, and indeed it is using \autoref{st_rho_full_rank}.

We have to establish the following lemma, so that by \autoref{ehresmann_with_corners} the stabilization $\St^e(\bar E)$ is a smooth fiber bundle over $\Delta^k$, as required by \autoref{hcob_space}. 
\begin{prop}
	If $E \to \Delta^k$ is a submersion then so is $\St^e(\bar E) \to \Delta^k$.
\end{prop}

\begin{proof}
	As this is a local statement, we can assume that $E = W \times \Delta^k$ and the bundle $\nu$ is trivial, but we cannot assume that \therh is constant. On an open neighborhood of the trivial region, the map is identified with a product projection
	\[ M' \times \Delta^k \times [-1,1] \to \Delta^k \]
	and is therefore a submersion.
	
	On the interior of the nontrivial region minus the cone locus, we can ignore the $C_2$ quotient, and since $S(V \times \Delta^k) \to M \times \Delta^k$ and $\bar W \times \Delta^k \to \Delta^k$ are submersions, so is the map
\[ ((\bar W \setminus N) \times \Delta^k) \times_{(M \times \Delta^k)} S(V \times \Delta^k) \to \Delta^k. \]
The fact that the retraction $r\colon \bar W \to M$ is not a submersion does not interfere with this.

Finally, on any chart at the cone locus
\[ U \times \Delta^k \times D(V \times \R), \]
in which we have chosen the bicollar $U \times \Delta^k \times (-\epsilon,\epsilon) \to \bar W \times \Delta^k$ to be fiberwise over $\Delta^k$, the map to $\Delta^k$ is just the projection, so it is a submersion here as well.
\end{proof}

For definiteness, each cobordism $W$ is a subset of some sufficiently large set $\mc U$, and for the stabilization procedure we fix a way of assigning the stabilized cobordism $\St^e(W)$ to another subset of $\mc U$. Each family $E \to \Delta^k$ is then considered as a subset of $\mc U \times \Delta^k$---this guarantees that restricting to a face or pulling back along a degeneracy strictly commutes with stabilization. As a result we get a map of spaces
\[ \St^e\colon \mc H_{\sbt}^c(M) \to \mc H_{\sbt}^c(M'). \]

For two stabilizations, the canonical homeomorphism $\St^{e_{01}}\St^{e_{12}}(\bar E) \cong \St^{e_{02}}(\bar E)$ has the same definition as before over each point of $\Delta^k$. As the formulas are built using functions that are smooth on all of $E$, the results are still smooth, and full rank whenever they are full rank over each point of $\Delta^k$ separately (because the derivative is the identity in the $\Delta^k$ direction). The canonical homeomorphism is therefore a diffeomorphism of manifolds over $\Delta^k$.

One can directly construct out of this a homotopy
\[ \St^{e_{01}}\St^{e_{12}} \sim \St^{e_{02}}\colon \mc H_{\sbt}^c(M_0) \times I \to \mc H_{\sbt}^c(M_2). \]
This makes $\mc H_{\sbt}^c$ (and therefore $\mc H_{\sbt}$) into a functor from the homotopy category of manifolds and smooth embeddings to the homotopy category of spaces. (Note the choice of tubular neighborhood $D(\nu) \to M'$ is contractible by \autoref{tubular_nbhd_thm}, making the induced maps well-defined up to homotopy.) In the next section, we do this in a more structured way and get a functor up to \emph{coherent} homotopy.

\section{The $h$-cobordism space as an $(\infty,1)$-functor}\label{infinitysec}

Now that we have most of the geometric preliminaries out of the way, we review a categorical framework to construct $(\infty,1)$-functors using Segal spaces, following \cite{pedro,nima}. We then build the smooth $h$-cobordism functor using this setup.

\subsection{Left fibrations of Segal spaces} We start by recalling the basic definitions. In this section any time we say ``space'' we mean simplicial set.

\begin{defn}\label{segal_cat}
A \ourdefn{Segal space} is a simplicial space $X_{\sbt}$ (i.e. a bisimplicial set) such that the derived Segal maps
\[\xymatrix{
X_n \ar[r] & X_1 \times^h_{X_0} X_1 \times^h_{X_0} \ldots \times^h_{X_0} X_1
} \]
are weak equivalences. An important special case is a \ourdefn{Segal category}, which is a Segal space with $X_0$ discrete.
\end{defn}

\orange{
We will take a simplicial category to mean a category internal to simplicial sets, rather than a category enriched in simplicial sets.
\begin{defn}\label{simplicial_cat}
A \ourdefn{simplicial category} is a category internal to simplicial sets, i.e., it has a simplicial set of objects $X_0$ and a simplicial set of morphisms $X_1$, and the source, target, and composition maps are maps of simplicial sets. In the special case when $X_0$ is discrete, this is a category enriched in simplicial sets.
\end{defn}

Note that every simplicial category $\mathcal C$ has a nerve, a bisimplicial set $N\mathcal C$ which at level $n$ is the $n$-fold fiber product $X_1 \times_{X_0} \ldots \times_{X_0} X_1$. If the source or target map $X_1 \to X_0$ is a fibration then this is a Segal space. In the special case that $\mathcal C$ is a category enriched in simplicial sets, i.e. $X_0$ is discrete, then $N\mathcal C$ is a Segal category.
}
In fact, up to an appropriate notion of equivalence, every Segal category comes from a simplicially enriched category, and therefore Segal categories are a model for $(\infty,1)$-categories \cite{bergner_thesis}.

\orange{
\begin{defn}\label{segal_equiv}
	A map $X\to Y$ of Segal spaces is a \ourdefn{levelwise equivalence} if $X_n \to Y_n$ is an equivalence of spaces for all $n$. A map $X\to Y$ of simplicial categories is called a levelwise equivalence if the maps $X_0\to Y_0$ and $X_1\to Y_1$ are equivalences of spaces. When $X_0$ and $Y_0$ are discrete we also call this a \ourdefn{pointwise equivalence} of simplicially enriched categories; it is a map inducing a bijection on objects and a weak equivalence on the mapping spaces. A \ourdefn{Dwyer-Kan equivalence} is a map of simplicially enriched categories that is a weak equivalence on each mapping space separately, and that is homotopically essentially surjective \cite[Def 2.7]{bergner_thesis}, \cite[Def 2.4]{dk3}.
\end{defn}
}

In contrast to much of the literature on Segal spaces, following \cite{pedro}, we do \emph{not} assume that our Segal spaces are Reedy fibrant. This will be convenient for the examples we aim to build.

\begin{defn}\label{left_fib}
Let $B$ be a Segal space. A  \ourdefn{left fibration} over $B$ is a map of simplicial spaces $X\to B$ such that $X$ is also a Segal space, and
	 the square
\[\xymatrix{
X_1\ar[r]^{d_0} \ar[d] &X_0\ar[d]\\
B_1\ar[r]^{d_0} & B_0
}\]
is homotopy cartesian. By \cite[1.7]{pedro}, in lieu of checking that $X$ is a Segal space, we could alternatively check that for each $n > 0$ the square
\[\xymatrix{
X_n\ar[r]^{d_0^n} \ar[d] &X_0\ar[d]\\
B_n\ar[r]^{d_0^n} & B_0
}\]
is homotopy cartesian.
\end{defn}

Let $\sF_B$ denote the category of left fibrations over the Segal space $B$. Inverting the \orange{levelwise} equivalences on the category of left fibrations over $B$, gives a homotopy category of left fibrations over $B$, which we denote $ho\sF_B$.

\begin{prop}\label{transport}
	Any \orange{levelwise} equivalence of Segal spaces $B' \to B$ induces an equivalence on homotopy categories of fibrations $ho\sF_{B'} \simeq ho\sF_B$.
\end{prop}

In fact, the homotopy category $ho\sF_B$ comes from a model structure on the category $ssSet_B$ of all simplicial spaces over $B$, described in \cite[Proposition 1.10]{pedro}, and any \orange{levelwise} equivalence of Segal spaces induces a Quillen equivalence. On the other hand, there is a simpler proof of \autoref{transport}:

A left fibration $X \to B$ is \ourdefn{fibrant} if the maps $X_n \to B_n$ are Kan fibrations. If $B' \to B$ is any \orange{levelwise} equivalence of Segal spaces, the pullback of any fibrant left fibration $X$ over $B$ to $B'$ is a left fibration, see \cite[1.11]{pedro}. We therefore get a functor $\sF_B \to \sF_{B'}$. This functor preserves \orange{levelwise} equivalences between fibrant left fibrations over $B$, giving a functor $ho\sF_B \to ho\sF_{B'}$. The left adjoint $\sF_{B'} \to \sF_B$ sends each left fibration $X' \to B'$ to the composite $X' \to B' \to B$. This is always a Segal space, and it is a left fibration as well under the assumption that $B' \to B$ is a \orange{levelwise} equivalence. It is easy to see that the derived functors of this adjunction give an equivalence of homotopy categories $ho\sF_{B'} \simeq ho\sF_B$.

We will rely on the following result from \cite{pedro, nima}, which we restate in terms of left instead of right fibrations.

\begin{thm}\label{mainpedro}
Let $\mc C$ be a simplicially enriched category. There exists a Quillen equivalence
$$\ssSet_{/N\mc C} \leftrightarrows \Fun(\mc C, \sSet),$$
where the category of bisimplicial sets $\ssSet$ over $N\mc C$ is endowed with the left fibration model structure, and the \orange{the category of simplicially enriched functors} is endowed with the projective model structure. 
\end{thm}

Along this equivalence, each left fibration $X \to N\mc C$ is sent to a diagram in which the objects and morphisms can be described explicitly up to homotopy. Let $X(c)$ be the pullback $X_0 \times_{(N\mc C)_0} \{c\}$, in other words the subspace of $X_0$ lying over the object $c$. Note that 
\[ X_0 = \coprod_c X(c), \]
and the left fibration condition on $X$ gives canonical weak equivalences
\begin{equation}\label{left_fibration_equivalent_to_bar_construction}
\coprod_{c,c_1,\ldots,c_n} X(c) \times \mc C(c,c_1) \times \ldots \times \mc C(c_{n-1},c_n)  \xleftarrow{\sim} X_n. 
\end{equation}
In particular, for $n = 1$ we get a canonical zig-zag
\begin{equation}\label{zigzag_action}
	 \coprod_{c,d} X(c) \times \mc C(c,d)  \xleftarrow{\sim} X_1 \xrightarrow{d_1}  \coprod_d X(d). 
\end{equation}
This provides a map $X(c)\times \mc C(c,d)\to X(d)$, up to homotopy.

\begin{lem}\label{associated_diagram_explicit}
	The equivalence of \autoref{mainpedro} sends each left fibration $X \to N\mc C$ to a diagram whose value at $c$ is equivalent to $X(c)$, and for which the action of $\mc C(c,d)$ is in the same homotopy class as the zig-zag \eqref{zigzag_action}.
\end{lem}

\begin{proof}
	The description we gave of the objects and the actions is invariant under \orange{levelwise} equivalence of left fibrations. Furthermore, each left fibration is equivalent to one coming from a diagram on $\mc C$. Therefore without loss of generality, $X$ arose from a diagram on $\mc C$ by applying the right adjoint functor from \autoref{mainpedro}.
	
	Explicitly, this right adjoint takes a diagram with spaces $X(c)$ and forms a simplicial space by the categorical bar construction. In other words, the maps \eqref{left_fibration_equivalent_to_bar_construction} are isomorphisms. Furthermore the action of $d_1$ in \eqref{zigzag_action} is by the action of $\mc C$ on the diagram $X(-)$. This verifies the claim in this special case, and therefore in general as well.
\end{proof}
%
%

Using this machinery, our strategy will be as follows: we construct a simplicial category $\Manst$ of manifolds and stabilization data, and show it is \orange{pointwise} equivalent to the simplicial category $\Man$ of manifolds and smooth embeddings. We then construct a left fibration $X \to N\Manst$, such that $X_0$ is the disjoint union of the $h$-cobordism spaces $\mc H_{\sbt}^c(M)$, and the actions coming from $X_1$ are the required stabilization maps, up to homotopy.  An application of \autoref{associated_diagram_explicit} produces the desired $h$-cobordism functor on $\mc \Man$.

\subsection{The category of smooth manifolds and tubular neighborhoods}

\begin{defn}\label{Man}
Let $\Mansm$ refer to the category of smooth compact $G$-manifolds with corners, and smooth equivariant maps. The space of morphisms from $M_0$ to $M_1$ is the simplicial set $\Sm(M_0,M_1)$ whose $p$-simplices are equivariant smooth maps $$M_0 \times \Delta^p \to M_1.$$ Let $\Man$ be the subcategory with the same objects, but where the morphism spaces are the subspaces
\[ \Emb(M_0,M_1) \subseteq \Sm(M_0,M_1) \]
of those smooth maps that are smooth embeddings.
\end{defn}

One can show that these mapping spaces are equivalent to the singular simplices of the spaces of smooth maps or of embeddings with the $C^\infty$ topology.

\begin{defn}\label{Mstab}
	The \emph{category of manifolds and stabilization data} $\Manst$ is a category enriched in simplicial sets, described as follows. The objects are smooth compact $G$-manifolds with corners. A map from $M_0$ to $M_1$ is given by a round bundle $p_0\colon D(\nu_{01})\to M_0$ and an embedding $e_{01}\colon D(\nu_{01})\hookrightarrow M_1$. The round bundles are considered up to isomorphism of the bundles commuting with the embeddings.

Given another morphism from $M_1$ to $M_2$ defined by a round bundle $p_{12}\colon D(\nu_{12})\to M_1$ and embedding $e_{12}\colon D(\nu_{12})\hookrightarrow M_2$, we define the composite map from $M_0$ to $M_2$ by taking the subspace $D(\nu_{02})$ of the pullback 
	\begin{equation}\label{compatibility_pullback}
\xymatrix{
	D(\nu_{02}) \ar@/_1em/[dddr]_-{p_{02}} \ar@/^1em/[rrrd]^-{e_{02}} \ar[dr]^{\subseteq} &&& \\
	& D(\nu_{01})\times_{M_0} D(\nu_{12}) \ar[d] \ar[r]& D(\nu_{12}) \ar[d]^-{p_{12}} \ar[r]_-{e_{12}} & M_2 \\
	& D(\nu_{01}) \ar[d]^-{p_{01}} \ar[r]_-{e_{01}} & M_1 & \\
	& M_0 &&
}
\end{equation}

	\noindent consisting of points $(s,t)$ such that $s^2+t^2\leq 1$. This is a round bundle $p_{02}\colon D(\nu_{02})\to M_0$, and it has an embedding $e_{02}\colon D(\nu_{02})\hookrightarrow M_2$. Note that this composition is associative.
	
We simplicially enrich $\Manst$ by defining a $p$-simplex of morphisms from $M_0$ to $M_1$ to be a round bundle $D(\nu_{01})\to M_0\times \Delta^p$ with a codimension 0 embedding $D(\nu_{01})\to M_1\times \Delta^p$ over $\Delta^p$. The composition map is then a map of simplicial sets.
\end{defn}

\begin{thm}\label{manst_equivalent}
The map that forgets the stabilization data is a \orange{pointwise equivalence of simplicially enriched categories} $\Manst \to \Man$.
\end{thm}

\begin{proof}
\orange{By \autoref{deform_embedding_to_interior}, there is an isotopy of embeddings $N \to N$ from the identity to an embedding sending $N$ into its own interior. By composing with this isotopy, we get an isotopy from any $\Delta^p$ family of embeddings $M\to N$ to a family of embeddings that all lie in the interior of $N$. Furthermore, if the embeddings along $\partial \Delta^p$ already lie in the interior of $N$, then we may take the isotopy of embeddings $N \to N$ to be supported on a neighborhood of $\partial N$ that is disjoint from the image of $M \times \partial \Delta^p$. Then this is an isotopy} rel $\partial \Delta^p$. Therefore, if we restrict the mapping spaces $\Man(M,N)$ to those embeddings that land in the interior of the target, we get an equivalent \orange{simplicial set that is a} Kan complex. Similarly, restricting $\Manst(M,N)$ to the data where $D(\nu) \to N$ lands in the interior of $N$, gives an equivalent simplicial set that is a Kan complex. Now the result follows from \autoref{round_tubular_nbhd_thm}, the tubular neighborhood theorem for round bundles.
\end{proof}

\begin{rem}
	\orange{The mapping spaces $\Manst(M,N)$ are not Kan complexes. They become Kan complexes if we make $D(\nu) \to N$ always land in the interior of $N$.} However the stabilization from \autoref{fulldef} works just fine even if $D\nu$ touches $\partial N$. \orange{More importantly, if we insisted that $D\nu$ always landed in the interior of $N$, we would lose the identity morphisms, so $\Manst$ would not be a category anymore.}
\end{rem}

\subsection{The homotopy coherent $h$-cobordism functor}
The next step is to define the desired left fibration over $N\Man$. By \autoref{transport} and \autoref{manst_equivalent}, it suffices to build the left fibration over $N\Manst$ instead.

\begin{defn}\label{Hstab}
	We define $\Hanst$ as the following category internal to simplicial sets. The simplicial set of objects is $\coprod_M \mc H_{\sbt}^c(M)$, the space of all encased $h$-cobordisms over all compact $G$-manifolds with corners $M$. 
A $0$-simplex of morphisms from an $h$-cobordism $W_0$ on $M_0$ to an $h$-cobordism $W_1$ on $M_1$ is a morphism in $\Manst$ from $M_0$ to $M_1$ together with an encased diffeomorphism $f_{01} \colon \St^{e_{01}}\bar W_0 \cong \bar W_1$.

Given another morphism from $W_1$ to an $h$-cobordism $W_2$ on $M_2$, and encased diffeomorphism $f_{12} \colon \St^{e_{12}}\bar W_1 \cong \bar W_2$, we define the composite morphism from $W_0$ to $W_2$ as follows. We compose the morphisms in $\Manst$, and define the encased diffeomorphism $f_{02}\colon \St^{e_{02}}\bar W_0\cong W_2$ as 
\begin{equation}\label{compositeiso}\St^{e_{02}}\bar W_0\xrightarrow[\cong]{} \St^{e_{12}}\St^{e_{01}} \bar W_0 \xrightarrow[\cong]{\St^{e_{12}}(f_{01})} \St^{e_{12}}\bar W_1\xrightarrow[\cong]{f_{12}}\bar W_2,\end{equation}
where the first unlabeled diffeomorphism is the canonical isomorphism of \autoref{canonical_homeo}. We show associativity in the next lemma.

A $p$-simplex of morphisms in $\Hanst$ is a $p$-simplex in $\Manst$ from $M_0$ to $M_1$ together with an encased diffeomorphism of families $f_{01} \colon \St^{e_{01}}\bar E_0 \cong \bar E_1$ over $\Delta^p$. The composition is defined in the same was as above. The source, target, and composition maps are all maps of simplicial sets.\end{defn}

\begin{lem}
The composition rule for $\Hanst$ defined in \autoref{Hstab} is associative.
\end{lem}

\begin{proof}
We consider the associativity of composition on 0-simplices of the morphism simplicial set, since on $p$-simplices the argument is similar.

Suppose we are given morphisms  $M_0\to M_1\to M_2\to M_3$ in $\Hanst$. First, this is the data of such morphisms in $\Manst$, which is a diagram of round bundle composites

\[\xymatrix{
         D(\nu_{03}) \ar@/^2.6pc/@[][rrr]^-{e_{03}} \ar[d] \ar[r]& D(\nu_{13}) \ar@/^1.8pc/@[][rr]_-{e_{13}} \ar[d] \ar[r] & D(\nu_{23})\ar[d] \ar[r]_-{e_{23}}  & M_3\\
	 D(\nu_{02}) \ar@/^2.0pc/@[][rr]^-(0.7){e_{02}} \ar[d] \ar[r]& D(\nu_{12}) \ar[d] \ar[r]_-{e_{12}} & M_2 & \\
	 D(\nu_{01}) \ar[d] \ar[r]_-{e_{01}} & M_1 && \\
	 M_0 &&&
}\]
We recall that this composition in $\Manst$ is strictly associative. Moreover, we are given the data of isomorphisms
$$\St^{e_{01}} \bar W_0 \xrightarrow[\cong]{f_{01}} \bar W_1,\ \ \ \ \St^{e_{12}} \bar W_1 \xrightarrow[\cong]{f_{12}} \bar W_2,\ \ \ \ \St^{e_{23}} \bar W_2 \xrightarrow[\cong]{f_{23}}\bar W_3.$$
The two ways of associating the three-fold composition are given by the top and bottom routes of the following diagram, from $\St^{e_{03}}\bar W_0$ to $W_3$. The left-hand square commutes by \autoref{assoc_stab} and the middle square commutes by \autoref{canonical_homeo}. This verifies associativity in $\Hanst$.
\[\resizebox{\textwidth}{!}{$
 \xymatrix @C=5em{
	\St^{e_{23}}\St^{e_{02}} \bar W_0 \ar[r]^-\cong &
	\St^{e_{23}} \St^{e_{12}} \St^{e_{01}} \bar W_0 \ar[r]_-\cong^-{\St^{e_{23}}\St^{e_{12}}(f_{01})} &
	\St^{e_{23}} \St^{e_{12}} \bar W_1 \ar[r]_-\cong^-{\St^{e_{23}}(f_{12})} &
	\St^{e_{23}}\bar W_2 \ar[d]_-\cong^-{f_{23}}
	\\
	\St^{e_{03}}\bar W_0 \ar[u]^-\cong \ar[r]^-\cong &
	\St^{e_{13}}\St^{e_{01}} \bar W_0 \ar[u]^-\cong \ar[r]_-\cong^-{\St^{e_{13}}(f_{01})} &
	\St^{e_{13}} \bar W_1 \ar[u]^-\cong &
	W_3
}$} \]
\end{proof}

Since $\Hanst$ is a category internal to simplicial sets, its nerve is a strict Segal space. It will be a consequence of \autoref{left_fibration} below that it is also a (homotopical) Segal space.

There is an evident forgetful map of simplicial categories $\Hanst \to \Manst$, and therefore a map of bisimplicial sets $N\Hanst \to N\Manst$, that forgets the $h$-cobordism but remembers the base manifold.

\begin{prop}\label{left_fibration}
	The map $N\Hanst \to N\Manst$ is a left fibration.
\end{prop}

\begin{proof}
	By \autoref{left_fib}, it suffices to show that the morphism spaces, object spaces, and source maps define a homotopy pullback square:
	\[ \xymatrix{
		\Hanst_n \ar[d] \ar[r]^-{\textup{source}} & \Hanst_0 \ar[d] \\
		\Manst_n \ar[r]^-{\textup{source}} & \Manst_0
	} \]
	In the right-hand column, the target $\Manst_0$ is discrete, the set of all compact $G$-manifolds with corners $M$. Each fiber of the map $\Hanst_0 \to \Manst_0$ is a space of encased equivariant $h$-cobordisms $\mc H_{\sbt}^c(M)$. By \autoref{encased_to_mirror}, these fiber spaces are Kan complexes, so the right-hand vertical is a Kan fibration. Therefore it suffices to show that the map from $\Hanst_n$ to the strict pullback is an acyclic Kan fibration.
	
	We prove this first when $n = 1$. The claim is that if we have a $p$-simplex of morphisms $e\colon M_0 \to M_1$ in $N\Manst$, a \orange{$\Delta^p$-family of $h$-cobordisms $E_0 \to M_0 \times \Delta^p$ that restricts to the $\partial\Delta^p$-family $\partial E_0 \to M_0 \times \partial \Delta^p$, another $\partial\Delta^p$-family of $h$-cobordisms $\partial E_1 \to M_1 \times \partial\Delta^p$, and a diffeomorphism $\partial E_1 \cong \St^e(\partial E_0)$, then this data can be extended to a $\Delta^p$-family of $h$-cobordisms $E_1 \to M_1 \times \Delta^p$} and a diffeomorphism $E_1 \cong \St^e(E_0)$. This is easy---we form $E_1$ by taking $\St^e(E_0)$, and applying a bijection to the underlying set along $M_1 \times \partial\Delta^p$, so that the points in that subset are replaced by the corresponding points in the original family $\partial E_1$.

	For $n > 1$ the proof is the same, except that we have a $p$-simplex of composable morphisms $M_0 \to M_1 \to \cdots \to M_n$, a \orange{$\Delta^p$-family of $h$-cobordisms $E_0 \to M_0 \times \Delta^p$, and the restricted family $\partial E_0 \to M_0 \times \partial \Delta^p$ is stabilized to each $M_i$ and identified with some family $\partial E_i \to M_i \times \partial \Delta^p$, for each $i$. Again, we extend this to a full $\Delta^p$-family} $E_i \to M_i \times \Delta^p$ by stabilizing all of $E_0$, and relabeling the points over $\partial \Delta^p$ to match the original given family $\partial E_i$.
\end{proof}

Applying \autoref{associated_diagram_explicit} to this left fibration gives the main theorem of the paper.

\begin{thm}\label{h_cobordisms_infinity_1_functor}
There is a simplicially enriched functor $$\mc H_{\Diff}(-)\colon \Man \to \sSet\, ,$$ sending each compact $G$-manifold with corners $M$ to a space equivalent to $\mc H_{\Diff}(M)$, and each homotopy class of smooth equivariant embeddings $M \to M'$ to the homotopy class of maps $\mc H_{\Diff}(M) \to \mc H_{\Diff}(M')$ given by the stabilization in \autoref{sec:stab}.
\end{thm}

\orange{
\subsection{Changing the group $G$}
Within this brief subsection we will use the subscript $\mc H_G$ to emphasize that $G$ is the group of equivariance, since it will vary. We describe how taking $H$-fixed points, or taking $G$-orbits when the $G$-action is free, defines natural transformations
\[ \mc H_{G}(M) \ra \mc H_{WH}(M^H), \qquad \mc H_{\textup{$G$}}(M) \ra \mc H_1(M/G), \]
where the second natural transformation is only defined on the subcategory of compact smooth manifolds with free $G$-action.

By \autoref{mainpedro}, natural transformations of functors out of a simplicially enriched category $\mc C$ correspond to maps of left fibrations over $N\mc C$, i.e., maps of simplicial spaces commuting with the projection to $N\mc C$. More generally, we have the following lemma.
\begin{lem}
	Any functor of simplicially enriched categories $F\colon \mc C \to \mc D$ and any commuting square of left fibrations
	\[ \xymatrix{
		X \ar[d] \ar[r] & Y \ar[d] \\
		N\mc C \ar[r]^-{NF} & N\mc D
	} \]
	gives a natural transformation of the associated functors $X \to Y \circ F$, where $X\colon \mc C \to \sSet$ and $Y\colon \mc D \to \sSet$ are the simplicially enriched functors that correspond to the left fibrations from the diagram by \autoref{mainpedro}.
\end{lem}
\begin{proof}
Left fibrations are preserved by pullbacks, so we consider the left fibration $N\mc C \times_{N\mc D} Y\to N\mc C$ obtained by pulling back the left fibration $Y\to N\mc D$ along $NF$. We claim this corresponds to the simplicially enriched functor $Y\circ F\colon \mc C\to \sSet$. As in the proof of \autoref{associated_diagram_explicit}, to prove this it suffices to assume that $Y$ is a categorical bar construction of some diagram, which by abuse of notation we also call $Y\colon \mc D \to \sSet$. Then
\begin{align*}\label{left_fibration_equivalent_to_bar_construction}
Y_n &= \coprod_{d,d_1,\ldots,d_n} Y(d) \times \mc D(d,d_1) \times \ldots \times \mc D(d_{n-1},d_n), \\
N_n \mc D &= \coprod_{d,d_1,\ldots,d_n} * \times \mc D(d,d_1) \times \ldots \times \mc D(d_{n-1},d_n), \\
N_n \mc C &= \coprod_{c,c_1,\ldots,c_n} * \times \mc C(c,c_1) \times \ldots \times \mc C(c_{n-1},c_n), \\
\end{align*}
and so their pullback is
\[ \coprod_{c,c_1,\ldots,c_n} Y(F(c)) \times \mc C(c,c_1) \times \ldots \times \mc C(c_{n-1},c_n) \]
which is exactly the categorical bar construction of the functor $Y \circ F\colon \mc C \to \sSet$.

Thus the map of left fibrations $X\to N\mc C \times_{N\mc D} Y$ over $N\mc C$ induced by the commutative square corresponds to a natural transformation of the associated simplicially enriched functors $X \to Y \circ F\colon \mc C \to \sSet$ via \autoref{mainpedro}.
\end{proof}

Therefore it suffices to define maps of left fibrations as shown, where ``$G$-free'' means that we restrict attention to those manifolds and $h$-cobordisms in which the base manifold (and therefore the $h$-cobordism) has free $G$-action.
	\[ \xymatrix{
		\Hanst_G \ar[d] \ar[r]^-{(-)^H} & \Hanst_{WH} \ar[d] & \Hanst_{\textup{$G$-free}} \ar[d] \ar[r]^-{(-)/G} & \Hanst_{1} \ar[d] \\
		\Manst_G \ar[r]^-{(-)^H} & \Manst_{WH} & \Manst_{\textup{$G$-free}} \ar[r]^-{(-)/G} & \Manst_{1}.
	} \]
\begin{prop}
	The $H$-fixed points and $G$-orbits define commuting squares as above.
\end{prop}
\begin{proof}
	We define the functor $\Manst_G \to \Manst_{WH}$ on objects by taking each manifold $M_0$ to its $H$-fixed points $M_0^H$, and on morphisms, by taking the round bundle $D(\nu_{01}) \to M_0$ and embedding $D(\nu_{01}) \to M_1$ to their $H$-fixed points $D(\nu_{01}^H) \to M_0^H$ and embedding $D(\nu_{01}^H) \to M_1^H$. The $H$-fixed points of a round bundle are still a round bundle in a canonical way, and this rule respects compositions, so that we get a map of categories enriched in simplicial sets.
	
	For the map $\Hanst_G \to \Hanst_{WH}$, for each embedding $e\colon D(\nu_{01}) \to M_1$ we observe that the $H$-fixed points commute with pushouts and fiber products up to homeomorphism, giving a homeomorphism
	\begin{align*}
		\St^e(\bar W)^H & = \left[ \big( W \times_{M \times [0,1]} D(\nu \times \R) \big) \cup_{S(\nu \times \R)} C(\nu \times \R) \cup (M_1 \setminus \im e) \right]^H \\
		&\cong \big( W \times_{M \times [0,1]} D(\nu \times \R) \big)^H \cup_{S(\nu \times \R)^H} C(\nu \times \R)^H \cup (M_1 \setminus \im e)^H \\
		&\cong \big( W^H \times_{M^H \times [0,1]} D(\nu^H \times \R) \big) \cup_{S(\nu^H \times \R)} C(\nu^H \times \R) \cup (M_1^H \setminus \im e^H) \\
		& = \St^{e^H}(\bar W^H)
	\end{align*}
	that preserves the smooth structure and is therefore a diffeomorphism. Using this canonical diffeomorphism, we define the map $\Hanst_G \to \Hanst_{WH}$ by taking the $H$-fixed points of all of the cobordisms, and sending each $p$-simplex of morphisms $f_{01} \colon \St^{e_{01}}\bar E_0 \cong \bar E_1$ to the corresponding diffeomorphism
	\[ f_{01}^H \colon \St^{e_{01}^H}(\bar E_0^H) \cong [\St^{e_{01}}(\bar E_0)]^H \cong \bar E_1^H. \]
	This respects composition and the simplicial structure, giving a map of simplicial categories that makes the square commute.
	
	For the orbits of free actions, the proof is the same, using the fact that when the $G$-action is free on the base, the orbits of a bundle are an invertible operation whose inverse is given by pullback along $M \to M/G$, and the pullback certainly preserves fiber products. This gives a homeomorphism
	\begin{align*}
		\St^e(\bar W)/G & = \left[ \big( W \times_{M \times [0,1]} D(\nu \times \R) \big) \cup_{S(\nu \times \R)} C(\nu \times \R) \cup (M_1 \setminus \im e) \right]/G \\
		&\cong \big( W \times_{M \times [0,1]} D(\nu \times \R) \big)/G \cup_{S(\nu \times \R)/G} C(\nu \times \R)/G \cup (M_1 \setminus \im e)/G \\
		&\cong \big( W/G \times_{M/G \times [0,1]} D(\nu/G \times \R) \big) \cup_{S(\nu/G \times \R)} C(\nu/G \times \R) \cup (M_1/G \setminus \im e/G) \\
		& = \St^{e/G}(\bar W/G)
	\end{align*}
	that preserves the smooth structure and is therefore a diffeomorphism.
\end{proof}

\begin{cor}
	The $H$-fixed points and $G$-orbits define natural transformations of the functors constructed in \autoref{h_cobordisms_infinity_1_functor}
\[ \mc H_{G}(M) \ra \mc H_{WH}(M^H), \qquad \mc H_{\textup{$G$}}(M) \ra \mc H_1(M/G), \]
	where the second one is only defined on the subcategory where the $G$-action is free. Furthermore, for each individual manifold $M$ this map is homotopic to the map of $h$-cobordism spaces induced by taking the fixed points or orbits, respectively.
\end{cor}

\begin{rem}
	In the subsequent paper \cite{GIMM2} we will show that the first of these natural transformations is a weak equivalence whenever every point of $M$ has stabilizer conjugate to $H$, and the second of these natural transformations is always a weak equivalence.
\end{rem}
}

\section{The stable $h$-cobordism space }\label{hcobspacesec}

Now that we have made the space of equivariant $h$-cobordisms into a functor on smooth manifolds and smooth embeddings, the last task is to stabilize with respect to representation discs, and extend the functor to all $G$-CW complexes and continuous maps. \orange{Our construction is very general and will work for any simplicial functor $F\colon \Man \to \sSet$, not just the $h$-cobordism functor $\mc H(-)$ produced by \autoref{h_cobordisms_infinity_1_functor}. Let $\mc U$ be a complete $G$-universe, so it is a direct sum of countably many copies of each irreducible representation. This is topologized as a colimit of finite-dimensional subspaces, so every compact subset lies in a finite-dimensional subspace.}

\begin{defn}\label{stable_h_cobordisms}
	\orange{For any compact $G$-manifold $M$ and any simplicial functor $F\colon \Man \to \sSet$,} let
	\[ F^{\mc U}(M) =  F(M \times \mc U) =  \underset{\textup{compact } K \subseteq M \times \mc U} \colim \ \ F(K) \]
	be the space (simplicial set) obtained as the colimit over inclusions of compact $G$-invariant submanifolds of $M \times \mc U$. When $G = 1$ \orange{and $F(-)$ is the $h$-cobordism functor $\mc H(-)$, we also refer to this stabilization as $\mc H^\infty(M)$. Note that this is a filtered colimit of simplicial sets and is therefore a homotopy colimit.}
\end{defn}

Since colimits are natural, $F^{\mc U}(-)$ is also a simplicial functor on the category $\Man$ of compact smooth $G$-manifolds and embeddings.

\begin{rem}
	By cofinality, this colimit can be evaluated by taking only the submanifolds of the form $M \times D_R(V)$ for finite-dimensional representations $V \subseteq \mc U$, and $R \geq 1$ a radius that goes to infinity. Since $M \times D(V) \to M \times D_R(V)$ is homotopic through embeddings to a diffeomorphism, it induces an equivalence on $F(-)$, and therefore we have an equivalence from the colimit over representation discs,
\[ \xymatrix{ \underset{V \subseteq \mc U} \colim F(M \times D(V)) \ar[r]^-\sim & F(M \times \mc U). } \]
\end{rem}

\orange{
\begin{rem}
	Also by cofinality, the colimit can be evaluated by taking only the submanifolds of $M \times \mc U$ that are framed, or stably framed (cf. \cite[p.152]{waldhausen_manifold}).
\end{rem}
}

Our next task is to extend $F^{\mc U}(-)$ from the category of embeddings $\Man$ to the larger category of smooth maps $\Mansm$ from \autoref{Man}. To do this, we define the space $\Emb(M_0,M_1 \times \mc U)$ as the colimit of the spaces of embeddings into compact submanifolds of $M_1 \times \mc U$. Including the origin into $\mc U$, and projecting away $\mc U$, induce maps
\begin{equation}\label{forgetful}
\xymatrix{
	\Emb(M_0,M_1) \ar[r] & \Emb(M_0,M_1 \times \mc U) \ar[r]^-\sim & \Sm(M_0,M_1)
}
\end{equation}
whose composite is the inclusion of the space of embeddings into the space of all smooth maps.

\begin{lem}\label{emb_u_smooth}
	The second map in \eqref{forgetful}, that projects away $\mc U$, is a weak equivalence.
\end{lem}

\begin{proof}
	The argument is similar to that of \autoref{tubular_nbhd_thm} and \autoref{rh_contractible}. Given a $\Delta^k$-family of equivariant smooth maps $f_t\colon M_0 \to M_1$ and $\partial\Delta^k$-family of equivariant smooth maps $g_t\colon M_0 \to \mc U$ such that the product maps $(f_t,g_t)\colon M_0 \to M_1 \times \mc U$ are embeddings for all $t \in \partial\Delta^k$, we need to extend the family $g_t$ to all of $\Delta^k$ so that the product maps are embeddings for all $t \in \Delta^k$.
	
	 We first \orange{observe that by compactness, the maps $g_t$ for $t \in \partial \Delta^k$ all land in a finite-dimensional subspace $V \subseteq \mc U$. By \autoref{smooth_extension}, we can extend $g\colon \partial\Delta^k  \times M_0 \to V$ to a map $T \times M_0 \to V$ that we also denote $g$, where $T \subseteq \Delta^k$ is an open subset containing $\partial\Delta^k$. By \autoref{extend_embedding_family},} the product maps $(f_t,g_t)$ are still embeddings provided $T$ is small enough. \orange{Take a partition of unity $\{\lambda_T,\lambda_{int}\}$ on $\Delta^k$ subordinate to $\{T,\textup{int }\Delta^k\}$ and let $C = \lambda_{int}^{-1}(0)$, a closed neighborhood of $\partial\Delta^k$ contained in $T$.}
	
	\orange{Next we extend $g$ from a $T$-family of equivariant smooth maps $T \times M_0 \to V$ to a $\Delta^k$-family of equivariant smooth maps $\Delta^k \times M_0 \to W$, where $W$ is a larger finite-dimensional subspace of $\mc U$, such that $g_t\colon M_0 \to W$ is an embedding when $t \not\in C$, in other words when $\lambda_{int}(t) > 0$. To do this, we fix a smooth equivariant embedding $h$ of $M_0$ into the orthogonal complement of $V$ in $\mc U$, and let $W$ be the span of $V$ and the image of $h$. Then we use $\{\lambda_T,\lambda_{int}\}$ to interpolate between $g$ and the constant family of embeddings $h\colon M_0 \to W$. This gives a new $\Delta^k$-family of smooth maps $g_t\colon M_0 \to W$, and when $\lambda_{int}(t) > 0$ the map $g_t$ is an embedding because some projection of it is the embedding $h$ times a positive real number. Now} the product maps $(f_t,g_t)$ are embeddings for all $t \in \Delta^k$, \orange{because when $t \in C$ these maps have not been changed, and when $t \not\in C$ the map $g_t$ itself is an embedding.}
%
\end{proof}

This motivates us to find an intermediate category $\Manth$ that is equivalent to $\Mansm$, but whose morphism spaces involve embeddings of $M_0$ into $M_1 \times \mc U$. The construction of this category is subtle---we can't take all embeddings $M_0 \times \mc U \to M_1 \times \mc U$, because then projecting away $\mc U$ doesn't respect composition. On the other hand, we can't take all maps $M_0 \to M_1 \times \mc U$ and compose them using isometries $\mc U \times \mc U \to \mc U$, because then the category lacks identity maps. The following definition circumvents both of these issues.

Let $\mc L(\mc U,\mc U)$ be the space of equivariant linear isometries $\mc U \to \mc U$, i.e. maps that are equivariant, linear, and metric-preserving, but not necessarily isomorphisms. This is a simplicial set in which a $p$-simplex of linear isometries $\mc U \times \Delta^p \to \mc U$ \orange{must be smooth in the following sense: note that for any finite-dimensional subspace $V \subseteq \mc U$, the restriction to $V \times \Delta^p$ lands in a finite dimensional subspace of $\mc U$ and we require this map between finite dimensional manifolds with corners to be smooth.}

\begin{defn}\label{Mthick}
	The \ourdefn{category of manifolds and $\mc U$-embeddings} $\Manth$ has as objects the smooth compact $G$-manifolds with corners. The morphism space from $M_0$ to $M_1$ is the subspace
	\[ \Manth(M_0,M_1) \subseteq \mc L(\mc U,\mc U) \times \Emb(M_0,M_1 \times \mc U) \]
	of all isometries $f\colon \mc U \to \mc U$ and embeddings $e = (e^{(1)},e^{(2)})\colon M_0 \to M_1 \times \mc U$ such that the second coordinate $e^{(2)}$ lands in the orthogonal complement $f(\mc U)^\perp$.
	
	The composition is by composing the isometries $f$, and composing the embeddings $e$ by the rule
	\[ \xymatrix{
		M_0 \ar[r]^-{e_1} & M_1 \times f_1(\mc U)^\perp \ar[r]^-{e_2 \times f_2} & M_2 \times f_2(\mc U)^\perp \times f_2(f_1(\mc U)^\perp) \ar[r]^-{\subseteq} & M_2 \times (f_2 \circ f_1)(\mc U)^\perp.
	} \]
	It is straightforward to check this is associative and preserves the simplicial structure.
\end{defn}

We define functors
\[ \xymatrix{ \Man \ar[r] & \Manth \ar[r] & \Mansm } \]
that are the identity on objects. The first functor sends each embedding $e\colon M_0 \to M_1$ to the identity $\id\colon \mc U \to \mc U$ and the embedding $(e,0)\colon M_0 \to M_1 \times \mc U$. The second functor sends each pair $f\colon \mc U \to \mc U$ and $e = (e^{(1)},e^{(2)})$ to the smooth map $e^{(1)}\colon M_0 \to M_1$. The composite of these functors sends $e\colon M_0 \to M_1$ to itself, forgetting that $e$ is an embedding.

\begin{lem}\label{first_extension}
	The stable $h$-cobordism space $F^{\mc U}(-)$ extends from $\Man$ to $\Manth$.
\end{lem}
\begin{proof}
\orange{We say a map $M_0 \times \mc U \to M_1 \times \mc U$ is an embedding if its restriction to any compact submanifold of any finite dimensional subspace, which then lands in a finite dimensional subspace, is an embedding.}

	From the definition of the functor $F^{\mc U}(-)$, it extends to the larger category whose objects are manifolds $M_0$, and whose morphisms are embeddings
\[ \Emb(M_0 \times \mc U,M_1 \times \mc U). \]
	However, this category contains $\Manth$ inside, by restricting the morphisms to the subset of embeddings of the form
\[ (x,v) \mapsto (e^{(1)}(x),f(v) + e^{(2)}(x)). \]
	This subcategory, in turn, contains $\Man$ inside, as those embeddings of the form
	\[ (x,v) \mapsto (e(x),v), \]
	and on this subcategory we get the original functor $F^{\mc U}(-)$.
\end{proof}

\orange{
\begin{rem}
We note that an embedding $M_0 \times \mc U \to M_1 \times \mc U$ in our sense might fail to be a topological embedding, but we use this term to refer to maps that are embeddings on each finite-dimensional subspaces, which is all we will need in the upcoming proofs.
\end{rem}
}

To extend $F^{\mc U}(-)$ to the category manifolds and smooth maps $\Mansm$, we need to prove that the forgetful map $\Manth \to \Mansm$ is an equivalence on morphism spaces.

\begin{lem}\label{l_contractible}
	The space $\mc L(\mc U,\mc U)$ is contractible.
\end{lem}

\begin{proof}
	We import a standard proof from the continuous case \cite{lms} to the smooth case. Fix two isometries $i_1,i_2\colon \mc U \to \mc U$ such that the sum map
	\[ \xymatrix{ (i_1,i_2) \colon \mc U \oplus \mc U \ar[r]^-\cong & \mc U } \]
	is an isomorphism. For instance, if we write $\mc U$ as a countable sum of regular representations $\bigoplus_{n=1}^\infty V_n$, we could take $i_1$ to be a shuffle map that sends the $n$th summand to the $(2n)$th summand by an identity map, and $i_2$ the shuffle map that takes the $n$th summand to the $(2n-1)$st summand.
	
	Suppose there a exists smooth homotopy of isometries $H_1$ from $\id$ to $i_1$. Then, post-composing with $H_1$ gives a homotopy from the identity of $\mc L(\mc U,\mc U)$ to a map landing in the subspace of isometries that factor through $i_1$. Specifically, for each isometry $f$, the homotopy $H_1(f(-),t)$ deforms from $f$ to $i_1 \circ f$. After this, we can apply the homotopy
	\[ H_2(v,t) = (\cos t) \cdot i_1(f(v)) + (\sin t) \cdot i_2(g(v)) \]
	for a fixed isometry $g$. Since the images of $i_1$ and $i_2$ are orthogonal, this is a homotopy through isometries, and deforms the map $i_1 \circ (-)$ on $\mc L(\mc U,\mc U)$ to the constant map taking everything to the fixed isometry $i_2 \circ g$. Each of these homotopies is smooth and so induces a simplicial homotopy on $\mc L(\mc U,\mc U)$. (Pasting them together makes something that is only piecewise smooth, which is why we have to carry out the two homotopies separately.) Therefore $\mc L(\mc U,\mc U)$ is contractible.
	
	It remains to construct the homotopy $H_1$---in other words, we have reduced to showing that $\mc L(\mc U,\mc U)$ is path-connected. As in \cite{lms}, if we take the above explicit choices for $i_1$ and $i_2$, then the straight-line homotopy
	\[ H_1(v,t) = (1-t)v + (t)i_1(v) \]
	is a smooth homotopy through equivariant linear injective maps $\mc U \to \mc U$. To make this into a homotopy through isometries, we  apply a version of Gram-Schmidt orthogonalization.
	
	We note that this homotopy is through maps of the form $f \otimes_\R V$, where $f$ is a non-equivariant isometry and $V$ is a regular representation. Therefore it suffices to apply Gram-Schmidt to a one-parameter family of non-equivariant linear injective maps, to make them into isometries. This is straightforward: we pick a well-ordered orthonormal basis and apply the algorithm to its image under $(1-t)v + (t)i_1(v)$ for each value of $t$ separately. The resulting formulas are clearly smooth in $t$, and the resulting vectors are orthonormal, hence they define a one-parameter family of linear isometries from the identity to $i_1$. Applying $(-) \otimes_R V$ provides the desired smooth homotopy of equivariant linear isometries from the identity to $i_1$, finishing the proof.
\end{proof}

\begin{prop}\label{emb_equiv}
	The forgetful map $\Manth(M_0,M_1) \to \Mansm(M_0,M_1)$ is a weak equivalence.
\end{prop}

\begin{proof}
	This is a modification of the proof of \autoref{l_contractible}. We take the same homotopy $H_1$ from the identity of $\mc U$ to the inclusion $i_1$, and compose both $f$ and $e^{(2)}$ with $H_1$ to deform them to maps that factor through $i_1$. Then we use the homotopy $H_2$ to deform the smooth map $i_1 \circ e^{(2)}$ to $i_2 \circ g$ for a fixed embedding $g\colon M_0 \to \mc U$. Since $g$ is an embedding, throughout this homotopy the resulting product map $M_0 \to M_1 \times \mc U$ is an embedding. Also, since both $i_2(\mc U)$ and $i_1(f(\mc U)^\perp)$ are orthogonal to $i_1(f(\mc U))$, this homotopy is through embeddings that are in the orthogonal complement of $i_1(f(\mc U))$. Finally, we deform the isometry $i_1 \circ f$ to $i_1$ by pre-composing $i_1$ by the homotopy $H_1$. Throughout this homotopy, the isometry lands in $i_1(\mc U)$, so the embedding in $i_2(\mc U)$ is always in its orthogonal complement.
	
	Together, these homotopies deform the identity map of $\Manth(M_0,M_1)$ to the map that sends $(f,(e^{(1)},e^{(2)}))$ to $(i_1,(e^{(1)},i_2 \circ g))$. So we have made everything fixed, except the smooth map $e^{(1)} \colon M_0 \to M_1$.
	
	To show that the forgetful map is a weak equivalence, we now define its homotopy inverse to be the map sending $e$ to $(i_1,(e,i_2 \circ g))$. The composite of these two in one direction gives the identity of $\Mansm(M_0,M_1)$, while in the other direction we get the map that we have shown is homotopic to the identity of $\Manth(M_0,M_1)$.
\end{proof}

\begin{cor}\label{second_extension}
	Up to equivalence, the stabilized functor $F^{\mc U}(-)$ extends from $\Man$ to $\Mansm$.
\end{cor}

\begin{proof}
	We already know by \autoref{first_extension} that $F^{\mc U}(-)$ extends from $\Man$ to $\Manth$. The map $\Manth \to \Mansm$ is a bijection on objects, and by \autoref{emb_equiv}, it is a weak equivalence on mapping spaces. It follows that any functor on $\Manth$ extends up to equivalence to $\Mansm$, using either \autoref{transport} and \autoref{mainpedro} on the associated Segal spaces, or by taking a homotopy left Kan extension (see e.g. \cite{dhks}).
\end{proof}

For the penultimate step, we include $\Mansm$ into the simplicial category of $\mc F$ of finite $G$-CW complexes and continuous maps. 

\begin{prop}\label{dk1}
	The inclusion $\Mansm \to \mc F$ is a \orange{pointwise equivalence of simplicial categories.}
\end{prop}

\begin{proof}
		\orange{As in the proof of \autoref{manst_equivalent}, we can restrict the spaces $\Mansm(M_0,M_1)$ and $\mc F(M_0,M_1)$ to the equivalent subspaces of those maps that land in the interior of $M_1$. Then it suffices to show after doing this that the forgetful map from smooth to continuous is an acyclic Kan fibration.
		
		Suppose we are given a $\partial \Delta^k$-family of smooth equivariant maps $\phi\colon M_0 \to \textup{int }M_1$ that extends to a $\Delta^k$-family of continuous equivariant maps $f\colon M_0 \to \textup{int }M_1$. Since $\partial \Delta^k \times M_0$ is a partial boundary of $\Delta^k \times M_0$, we can use \autoref{smooth_extension} to extend the smooth family to a $T$-family of smooth maps, $\phi\colon T \times M_0 \to \textup{int }M_1$, for an open neighborhood $T$ of $\partial\Delta^k$. As in the proof of \autoref{emb_u_smooth}, we choose a smooth partition of unity $\{\lambda_T,\lambda_{int}\}$ on $\Delta^k$ subordinate to $\{T,\textup{int }\Delta^k\}$ and use this to interpolate between $\phi$ and $f$, giving a new $\Delta^k$-family of continuous maps $\tilde f_0$ that is smooth in a neighborhood of $\partial\Delta^k$ and agrees with $\phi$ on $\partial\Delta^k$. Then we apply smooth approximation (\autoref{smooth_approximation}) to make a smooth family $\tilde f_1$ that agrees with $\tilde f_0$ and therefore with $\phi$ on $\partial\Delta^k$.
		
		To make this interpolation work, since the target $M_1$ is not a vector space, we embed $M_1$ into a representation $V$ and let $\Omega \subseteq V$ be an open neighborhood that \emph{continuously} retracts onto $M_1$ by the nearest-neighbor map $\pi\colon \Omega \to M_1$, and let $\epsilon > 0$ be such that the $\epsilon$-tube about the compact set $M_1$ is contained in $\Omega$. (Note the contrast with the earlier proofs of \autoref{smooth_simplices}, \autoref{mirror_pseudo_equivalence}, and \autoref{rh_contractible} where we had to use the neighborhood that smoothly retracts but only onto $(\textup{int }M_1)$, and so in order to get an $\epsilon$ we had to restrict to some compact subset of $(\textup{int }M_1)$.) Then we choose $T$ small enough that the $C^0$-distance from $\phi$ to $f$ on $T \times M_0$ is smaller than $\frac{\epsilon}{2}$, and when we do the smooth approximation of $\tilde f_0$ to make $\tilde f_1$, we make the $C^0$ distance less than $\frac{\epsilon}{2}$. This guarantees that $\tilde f_0$ is well-defined, and that the $C^0$ distance from $\tilde f_1$ to the original family $f$ is less than $\epsilon$.
		
		Finally, we apply $\pi$ to the straight-line homotopy between $\tilde f_1$ and $f$ in $V$, which is possible because their distance is smaller than the radius of $\Omega$. This gives a continuous homotopy of families from $\tilde f_1$ to $f$, rel $\partial \Delta^k$. This proves that the lift of $\Delta^k$ from continuous maps to smooth maps exists after deforming it rel $\partial\Delta^k$, which is enough to prove that the map of simplicial sets is a weak equivalence.}
\end{proof}

We also recall the following standard fact, which can be proven by embedding into a representation and taking a sufficiently nice neighborhood:
\begin{prop}\label{dk2}
	Every finite $G$-CW complex is homotopy equivalent to some compact smooth $G$-manifold with boundary.
\end{prop}

Together \autoref{dk1} and \autoref{dk2} prove that $\Mansm \to \mc F$ is a Dwyer-Kan equivalence \orange{as in \autoref{segal_equiv}.} In other words, they have different objects, but the same equivalence classes of objects, and the mapping spaces are equivalent. As in the proof of \autoref{second_extension}, this is enough to conclude that the functor $F^{\mc U}(-)$ extends up to equivalence to the larger simplicial category $\mc F$, for instance by taking a homotopy left Kan extension from $\Mansm$ to $\mc F$.

Finally, we can perform an additional homotopy left Kan extension to extend the functor $F^{\mc U}(-)$ from $\mc F$ to the simplicially enriched category of all $G$-CW complexes, not necessarily finite, and continuous maps. This does not change its homotopy type on the finite complexes. Therefore, on the smooth compact $G$-manifolds and embeddings, we still have the same $h$-cobordism space up to equivalence, and the same stabilization map up to homotopy. We conclude:

\begin{thm}\label{thm_stabilize_arbitrary_extend}
	For every simplicial functor $F\colon \Man \to \sSet$, the stabilized functor $F^{\mc U}\colon \Man \to \sSet$ is equivalent to a simplicial functor that extends the larger category of all $G$-CW complexes.
\end{thm}

\begin{cor}
	Up to equivalence, the stable $h$-cobordism functor $\mc H^{\mc U}(-)$ extends to a simplicial functor from all $G$-CW complexes to spaces.
\end{cor}

\orange{
Note that for a functor $F\colon \Man \to \sSet$, by \autoref{thm_stabilize_arbitrary_extend} the stabilized functor $F^{\mc U}\colon \Man \to \sSet$ is automatically a homotopy functor, in the sense that it sends every smooth embedding $M_0 \to M_1$ that is an equivariant homotopy equivalence to a weak equivalence of simplicial sets. This implies that if we stabilize it again then nothing happens:
\begin{lem}\label{stabilize_doesnt_change_homotopy_functor}
	If $F$ is a homotopy functor then the natural transformation $F \to F^{\mc U}$ is a weak equivalence. For any functor $F$, the stabilization of the map $F \to F^{\mc U}$ produces a weak equivalence $F^{\mc U} \to (F^{\mc U})^{\mc U}$.
\end{lem}

\begin{proof}
	The first claim follows because the colimit defining $F^{\mc U}$ can be simplified to a colimit along a sequence of manifolds of the form $M \times D(V_n)$, where $V_n$ is a sequence of representations that exhausts $\mc U$. Since $F$ is a homotopy functor, it sends the embeddings $M \times D(V_n) \to M \times D(V_{n+1})$ to weak equivalences. Since sequential colimits in simplicial sets are homotopy colimits, we conclude that $F(M) \to F^{\mc U}(M)$ is a weak equivalence.
	
	For the second claim, we write out the definition of the map $F^{\mc U}(M) \to (F^{\mc U})^{\mc U}(M)$ as a map of colimits
\[ \xymatrix{ \underset{m \to \infty} \colim F(M \times D(V_m)) \ar[r] & \underset{m \to \infty} \colim \underset{n \to \infty} \colim F(M \times D(V_m) \times D(V_n)). } \]
	This is isomorphic to the map from $F^{\mc U}$ to its own stabilization, and since $F^{\mc U}$ is a homotopy functor, this is a weak equivalence.
\end{proof}

It follows from a standard result in higher category theory \cite[Proposition 5.2.7.4(3)]{lurie_htt} that $F^{\mc U}$ is the universal approximation of $F$ by a homotopy functor. In other words, the $\infty$-category of all simplicial homotopy functors $\Man\to\sSet$ is a reflective localization of the $\infty$-category of all simplicial functors $\Man\to\sSet$, and the associated localization is exactly the stabilization procedure $F \mapsto F^{\mc U}$ of \autoref{stable_h_cobordisms}.

We can also characterize $F^{\mc U}$ as the universal extension of $F$ from smooth compact $G$-manifolds to all $G$-spaces. Recall that $\mc F$ is the simplicial category of finite $G$-CW complexes and continuous maps. 
\begin{lem}\label{man_to_finite_CW}
	There is an equivalence of homotopy categories (or $\infty$-categories) between the category of simplicial homotopy functors $\Man \to \sSet$ and the category of simplicial functors $\mc F \to \sSet$.
\end{lem}

\begin{proof}
	Given a simplicial functor $F\colon \Man \to \sSet$ that sends homotopy equivalences to weak equivalences, the extension $F^{\mc U}\colon \mc F \to \sSet$ is  equivalent to $F$ on the subcategory $\Man$ by \autoref{stabilize_doesnt_change_homotopy_functor}. So if we go from $\Man$ to $\mc F$ and back, we get the same functor up to equivalence. On the other hand, if we go from $\mc F$ to $\Man$ and back, we get the same functor up to equivalence on the subcategory $\Man$. By \autoref{dk2}, the functors are therefore equivalent on the entire category $\mc F$.
\end{proof}

We also have the standard result that restriction and homotopy left Kan extension give an adjunction of homotopy categories (or $\infty$-categories) between simplicial functors $\mc F \to \sSet$ and simplicial functors from all $G$-CW complexes to $\sSet$.
%
%
All together, this means that $F^{\mc U}$ is the universal approximation of $F$ by a functor that is defined on all $G$-spaces. 
\begin{cor}
	If $\mc D$ is the homotopy category (or $\infty$-category) of simplicial functors from $G$-CW complexes to $\sSet$, $\mc C$ is the homotopy category (or $\infty$-category) of simplicial functors from $\Man$ to $\sSet$, and $\mc D \to \mc C$ is the forgetful functor, then the left adjoint $\mc C \to \mc D$ is given explicitly by the stabilization $F \mapsto F^{\mc U}$ from \autoref{thm_stabilize_arbitrary_extend}.
\end{cor}

In particular, taking $F$ to be the smooth $h$-cobordism functor $\mc H(-)$, this implies that $\mc H^{\mc U}(-)$ is the universal extension of  $\mc H(-)$ to a functor on all $G$-CW complexes that sends equivariant homotopy equivalences to weak equivalences.}

\begin{rem}
	The constructions in this paper also apply to topological $h$-cobordisms $W$ over smooth manifolds $M$. The doubles, mirror structures, and height functions are unnecessary, and all functions only have to be continuous, not smooth. We get a forgetful natural transformation from the smooth functor to the topological one, both before and after stabilizing by $\mc U$. We therefore get a natural transformation of functors on all $G$-CW complexes
	\[ \mc H^{\mc U}_{\Diff}(-) \to \mc H^{\mc U}_{\Top}(-). \]
	In the non-equivariant case, these stable $h$-cobordism functors agree (as functors on the homotopy category) with earlier definitions found in the literature \cite{hatcher_concordance, waldhausen_manifold}. 
	It would be useful to upgrade this by showing that our definition agrees with earlier ones in the topological category as $(\infty,1)$ functors, as in \cite{pieper}. This would require further elaborations concerning PL $h$-cobordisms and simple maps that go beyond the scope of this paper.

\end{rem}

  \bibliographystyle{alpha}
  \bibliography{references}

\begin{thebibliography}{HLLRW21}

\bibitem[BdB18]{pedro}
Pedro Boavida~de Brito.
\newblock Segal objects and the {G}rothendieck construction.
\newblock In {\em An alpine bouquet of algebraic topology}, volume 708 of {\em
  Contemp. Math.}, pages 19--44. Amer. Math. Soc., [Providence], RI, [2018]
  \copyright 2018.

\bibitem[Ber07]{bergner_thesis}
Julia~E. Bergner.
\newblock Three models for the homotopy theory of homotopy theories.
\newblock {\em Topology}, 46(4):397--436, 2007.

\bibitem[BLR75]{blr75}
Dan Burghelea, Richard Lashof, and Melvin Rothenberg.
\newblock {\em Groups of automorphisms of manifolds}, volume Vol. 473 of {\em
  Lecture Notes in Mathematics}.
\newblock Springer-Verlag, Berlin-New York, 1975.
\newblock With an appendix (``The topological category'') by E. Pedersen.

\bibitem[Cer61]{cerf}
Jean Cerf.
\newblock Topologie de certains espaces de plongements.
\newblock {\em Bull. Soc. Math. France}, 89:227--380, 1961.

\bibitem[DHKS04]{dhks}
William~G. Dwyer, Philip~S. Hirschhorn, Daniel~M. Kan, and Jeffrey~H. Smith.
\newblock {\em Homotopy limit functors on model categories and homotopical
  categories}, volume 113 of {\em Mathematical Surveys and Monographs}.
\newblock American Mathematical Society, Providence, RI, 2004.

\bibitem[DK80]{dk3}
W.~G. Dwyer and D.~M. Kan.
\newblock Function complexes in homotopical algebra.
\newblock {\em Topology}, 19(4):427--440, 1980.

\bibitem[GIMM]{GIMM2}
Thomas Goodwillie, Kiyoshi Igusa, Cary Malkiewich, and Mona Merling.
\newblock The splitting of the stable space of equivariant $h$-cobordisms.
\newblock {\em to appear}.

\bibitem[GJ99]{goerss_jardine}
Paul~G. Goerss and John~F. Jardine.
\newblock {\em Simplicial homotopy theory}, volume 174 of {\em Progress in
  Mathematics}.
\newblock Birkh\"{a}user Verlag, Basel, 1999.

\bibitem[Hat78]{hatcher_concordance}
A.~E. Hatcher.
\newblock Concordance spaces, higher simple-homotopy theory, and applications.
\newblock In {\em Algebraic and geometric topology ({P}roc. {S}ympos. {P}ure
  {M}ath., {S}tanford {U}niv., {S}tanford, {C}alif., 1976), {P}art 1}, Proc.
  Sympos. Pure Math., XXXII, pages 3--21. Amer. Math. Soc., Providence, R.I.,
  1978.

\bibitem[Hir94]{hirsch}
Morris~W. Hirsch.
\newblock {\em Differential topology}, volume~33 of {\em Graduate Texts in
  Mathematics}.
\newblock Springer-Verlag, New York, 1994.
\newblock Corrected reprint of the 1976 original.

\bibitem[HLLRW21]{hllrw}
Fabian Hebestreit, Markus Land, Wolfgang L\"uck, and Oscar Randal-Williams.
\newblock A vanishing theorem for tautological classes of aspherical manifolds.
\newblock {\em Geom. Topol.}, 25(1):47--110, 2021.

\bibitem[Igu88]{igusa}
Kiyoshi Igusa.
\newblock The stability theorem for smooth pseudoisotopies.
\newblock {\em $K$-Theory}, 2(1-2):vi+355, 1988.

\bibitem[Joy12]{joyce}
Dominic Joyce.
\newblock On manifolds with corners.
\newblock In {\em Advances in geometric analysis}, volume~21 of {\em Adv. Lect.
  Math. (ALM)}, pages 225--258. Int. Press, Somerville, MA, 2012.

\bibitem[Kra22]{krannich_homological}
Manuel Krannich.
\newblock A homological approach to pseudoisotopy theory. {I}.
\newblock {\em Invent. Math.}, 227(3):1093--1167, 2022.

\bibitem[LMS86]{lms}
L.~G. Lewis, Jr., J.~P. May, and M.~Steinberger.
\newblock {\em Equivariant stable homotopy theory}, volume 1213 of {\em Lecture
  Notes in Mathematics}.
\newblock Springer-Verlag, Berlin, 1986.
\newblock With contributions by J. E. McClure.

\bibitem[Lur]{lurie_937_lecture6}
Jacob Lurie.
\newblock Topics in {G}eometric {T}opology, {L}ecture 6: {D}iffeomorphisms and
  {PL} {H}omeomorphisms.
\newblock {\em available at
  \href{https://www.math.ias.edu/~lurie/937notes/937Lecture6.pdf}{author's
  webpage}.}

\bibitem[Lur09]{lurie_htt}
Jacob Lurie.
\newblock {\em Higher topos theory}, volume 170 of {\em Annals of Mathematics
  Studies}.
\newblock Princeton University Press, Princeton, NJ, 2009.

\bibitem[May99]{concise}
J.~P. May.
\newblock {\em A concise course in algebraic topology}.
\newblock Chicago Lectures in Mathematics. University of Chicago Press,
  Chicago, IL, 1999.

\bibitem[ME23]{me_ww}
Samuel Mu{\~n}oz-Ech{\'a}niz.
\newblock A {W}eiss-{W}illiams theorem for spaces of embeddings and the
  homotopy type of spaces of long knots.
\newblock {\em arXiv preprint arXiv:2311.05541}, 2023.

\bibitem[Mel]{melrose}
Richard~B. Melrose.
\newblock Differential analysis on manifolds with corners.
\newblock {\em available at \href{https://math.mit.edu/~rbm/book.html}{author's
  webpage}.}

\bibitem[MM19]{CaryMona}
Cary Malkiewich and Mona Merling.
\newblock Equivariant {$A$}-theory.
\newblock {\em Doc. Math.}, 24:815--855, 2019.

\bibitem[MM22]{CaryMona3}
Cary Malkiewich and Mona Merling.
\newblock The equivariant parametrized {$h$}-cobordism theorem, the
  non-manifold part.
\newblock {\em Adv. Math.}, 399:Paper No. 108242, 42, 2022.

\bibitem[MW09]{christoph}
Christoph M\"{u}ller and Christoph Wockel.
\newblock Equivalences of smooth and continuous principal bundles with
  infinite-dimensional structure group.
\newblock {\em Adv. Geom.}, 9(4):605--626, 2009.

\bibitem[Pie18]{pieper}
Malte~Mario Pieper.
\newblock Assembly maps and pseudoisotopy functors.
\newblock {\em Dissertation, Rheinische Friedrich-Wilhelms-Universit{\"a}t
  Bonn}, 2018.

\bibitem[Ras17]{nima}
Nima Rasekh.
\newblock Yoneda lemma for simplicial spaces.
\newblock {\em arXiv preprint arXiv:1711.03160}, 2017.

\bibitem[Wal82]{waldhausen_manifold}
Friedhelm Waldhausen.
\newblock Algebraic {$K$}-theory of spaces, a manifold approach.
\newblock In {\em Current trends in algebraic topology, {P}art 1 ({L}ondon,
  {O}nt., 1981)}, volume~2 of {\em CMS Conf. Proc.}, pages 141--184. Amer.
  Math. Soc., Providence, R.I., 1982.

\bibitem[Was69]{wasserman}
Arthur~G. Wasserman.
\newblock Equivariant differential topology.
\newblock {\em Topology}, 8:127--150, 1969.

\bibitem[Whi34]{whitney}
Hassler Whitney.
\newblock Analytic extensions of differentiable functions defined in closed
  sets.
\newblock {\em Trans. Amer. Math. Soc.}, 36(1):63--89, 1934.

\bibitem[WJR13]{wjr}
Friedhelm Waldhausen, Bj{\o}rn Jahren, and John Rognes.
\newblock {\em Spaces of {PL} manifolds and categories of simple maps}, volume
  186 of {\em Annals of Mathematics Studies}.
\newblock Princeton University Press, Princeton, NJ, 2013.

\end{thebibliography}

\begingroup%
\setlength{\parskip}{\storeparskip}

\end{document}